\documentclass{article}
\usepackage{amsfonts}
\usepackage{amsmath}
\usepackage{amsfonts}
\usepackage{amssymb}
\usepackage{graphicx}
\usepackage{amsmath}
\usepackage[a4paper]{geometry}

\newtheorem{theorem}{Theorem}

\newtheorem{corollary}{Corollary}

\newtheorem{definition}{Definition}

\newtheorem{lemma}{Lemma}

\newtheorem{proposition}{Proposition}
\newtheorem{remark}{Remark}

\newenvironment{proof}[1][Proof]{\noindent\textbf{#1.} }{\ \rule{0.5em}{0.5em}}
\input{tcilatex}

\begin{document}

\title{Efficient Simulation-Based Minimum Distance Estimation and Indirect
Inference }
\author{Richard Nickl \\
Statistical Laboratory, University of Cambridge \\
and \and Benedikt M. P\"{o}tscher \\
Department of Statistics, University of Vienna}
\date{First version: March 2009\\
This version: June 2010}
\maketitle

\begin{abstract}
Given a random sample from a parametric model, we show how indirect
inference estimators based on appropriate nonparametric density estimators
(i.e., simulation-based minimum distance estimators) can be constructed
that, under mild assumptions, are asymptotically normal with
variance-covarince matrix equal to the Cram\'{e}r-Rao bound.
\end{abstract}

\section{Introduction\label{intro}}

Suppose we observe a random sample $X_{1},\ldots ,X_{n}$ from a distribution 
$P$, and we are in the classical situation where one maintains a \textit{%
parametric} model $\mathcal{M=}\{P(\theta ):\theta \in \Theta \}$ of
probability measures $P(\theta )$, indexed by the set $\Theta \subseteq 
\mathbb{R}^{b}$, for statistical inference. Under the assumption of correct
specification of the parametric model, i.e., $P=P(\theta _{0})$ for a
(unique) $\theta _{0}\in \Theta $, the maximum likelihood estimator (MLE) is
often a natural estimator of $\theta _{0}$ (as well as of $P(\theta _{0})$),
since it is \textit{asymptotically efficient} under well-known regularity
conditions.

There are several reasons, however, why maximum likelihood might
nevertheless not be the method of choice, and alternatives, that ideally are
also \textit{asymptotically efficient,} are of interest.

A first such reason is rather classical (e.g., Huber (1972), Beran (1977),
Millar (1981), Donoho and Liu (1988), Lindsay (1994)) and comes from
robustness considerations: A good estimator for $\theta _{0}$ should be
robust against misspecifications of $\mathcal{M}$. A lesson from the
above-mentioned literature is the following: If one wants an estimator of $%
\theta _{0}$ that is robust against perturbations of $P(\theta _{0})$ in
some metric $\chi (\cdot ,\cdot )$, then one should rather use `minimum
distance estimators' of the following form: if $\tilde{P}_{n}$ is a suitable
(typically nonparametric) $\chi $-consistent estimator of $P$, estimate $%
\theta $ by the minimizer over $\Theta $ of 
\begin{equation}
Q_{n}(\theta ):=\chi (\tilde{P}_{n},P(\theta )).  \label{mindist}
\end{equation}%
Under several assumptions, Beran (1977) showed the interesting result that,
if $\chi $ is the Hellinger-distance, and if $\tilde{P}_{n}$ is some kernel
density estimator, such minimum-distance estimators are not only robust, but
actually simultaneously \textit{asymptotically efficient}, so that they
outperform the MLE in this sense. We will discuss the asymptotic efficiency
aspect of his result in more detail below.

A second, more practical reason against the use of the MLE that has arisen
in recent applications in econometrics and biostatistics is related to the
fact that in these applications analytic expressions for the densities in
the parametric model, and hence for the likelihood function, are not
available (or intractable for numerical purposes). For example, the data may
be modeled by an equation of the form $X_{i}=g(\varepsilon _{i},\theta _{0})$%
, but the implied parametric density may not be analytically tractable,
e.g., because $g$ is complicated or $\varepsilon _{i}$ is high-dimensional.
The same problem occurs naturally also in estimation of dynamic nonlinear
models including stochastic differential equations, we refer to Smith
(1993), Gourieroux, Monfort and Renault (1993), Gallant and Tauchen (1996),
Gallant and Long (1997) and the monograph Gourieroux and Monfort (1996) for
several concrete examples. This problem has led to a growing literature
about so-called \textit{indirect inference} methods, where other estimators
than the MLE are suggested, often based on simulations, see the just
mentioned references and Jiang and Turnbull (2004). From a conceptual point
of view, the main idea behind the indirect inference approach can be phrased
as follows:

\begin{enumerate}
\item Simulate a sample $X_{1}(\theta ),...,X_{k}(\theta )$ of size $k$ from
the distribution $P(\theta )$ for $\theta \in \Theta $ (which is often
possible in the examples alluded to above, e.g., by perusing the equations
defining the model; see also Remark \ref{Neu1}).

\item Based on the simulated sample \textit{as well as} on the true data,
compute estimators $\tilde{P}_{k}(\theta )$ and $\tilde{P}_{n}$ in a not
necessarily correctly-specified but numerically tractable auxiliary model $%
\mathcal{M}^{aux}$. [For example, by maximum likelihood if $\mathcal{M}%
^{aux} $ is finite-dimensional.]

\item Choose a suitable metric $\chi $ on $\mathcal{M}^{aux}$, and estimate $%
\theta _{0}$ by minimizing over $\Theta $ the objective function 
\begin{equation}
\mathcal{Q}_{n,k}(\theta ):=\chi (\tilde{P}_{n},\tilde{P}_{k}(\theta )).
\label{mindist1}
\end{equation}
\end{enumerate}

In most of the indirect inference literature, the auxiliary model $\mathcal{M%
}^{aux}$ is also finite-dimensional (so that one in fact estimates a
finite-dimensional parameter in Step 2 rather than the probability measure
directly), and the resulting procedure can be shown to be consistent and
asymptotically normal (under standard regularity conditions, see Gourieroux
and Monfort (1996)). However, the procedure is asymptotically efficient only
if $\mathcal{M}^{aux}$ happens to be \textit{correctly specified}. This
assumption is certainly restrictive and often unnatural if $\mathcal{M}%
^{aux} $\ is of fixed finite dimension. Therefore Gallant and Long (1997)
suggested that choosing $\mathcal{M}^{aux}$ with dimension increasing in
sample size should result in estimators that are asymptotically efficient,
the idea being that this essentially amounts to choosing an
infinite-dimensional auxiliary model $\mathcal{M}^{aux}$ for which the
assumption of correct specification is much less restrictive.

\textit{In the present paper we show in some generality that indirect
inference estimators based on suitable nonparametric estimators $\tilde{P}%
_{n}$ and $\tilde{P}_{k}(\theta )$ with common choices for the tuning
parameters (`sieve'-dimensions), including rate-optimal choices, are
asymptotically efficient in the sense that they are asymptotically normal
with asymptotic variance equal to the Cram\'{e}r-Rao bound.} To the best of
our knowledge, no proof of this fact was known before, although there are
some related results that need mentioning. We comment on the literature in
some detail below, but first wish to discuss the main ideas behind our
results. [Robustness issues, misspecification of $\mathcal{M}$, as well as
uniformity in the asymptotic normality result are not treated explicitly in
this paper; for the latter two issues in a related context see Gach (2010).]

From the discussion so far it transpires that indirect inference estimators
from (\ref{mindist1}) are \textit{minimum distance estimators}, with the
important (and nontrivial) modification that $P(\theta )$ in (\ref{mindist})
is replaced by an \textit{estimator} based on simulations from $P(\theta )$.
It is therefore of interest to first briefly revisit Beran's (1977)
asymptotic efficiency result: For simplicity, consider the Fisher-metric $%
\chi _{F}(f,g)^{2}:=\int (f-g)^{2}p_{0}^{-1}$, where $p_{0}$ is the density
of $P$, instead of the Hellinger distance. [Note that the Fisher-metric is
closely related to the Hellinger distance when $f$ and $g$ are near $p_{0}$%
.] If $\hat{\theta}_{n}$ is the minimizer of $Q_{n}$ in (\ref{mindist}),
then, after a suitable Taylor expansion, asymptotic efficiency of $\sqrt{n}(%
\hat{\theta}_{n}-\theta _{0})$ essentially reduces to proving two separate
results: The first is to prove asymptotic normality for the gradient of (\ref%
{mindist}) at $\theta _{0}$, namely 
\begin{equation}
\sqrt{n}\int s(\theta _{0})d(\tilde{P}_{n}-P(\theta _{0})),  \label{score}
\end{equation}%
where the `influence function' $s(\theta _{0})$ equals $\nabla _{\theta
}p(\theta _{0})p_{0}^{-1}$. Note that $s(\theta _{0})$ coincides with the 
\textit{efficient} influence function in this problem, showing that $\chi
=\chi _{F}$ is a natural choice. The second step is to control the remainder
term in the Taylor expansion, which essentially requires convergence of $%
\tilde{P}_{n}$ to $P=P(\theta _{0})$ (in the sense of $L^{p}$-convergence of
the respective densities for certain values of $p$). Beran (1977) implicitly
proved these two results under relatively restrictive conditions if $\tilde{P%
}_{n}$ is a kernel density estimator with certain bandwidths, and if $\chi $
is the Hellinger metric. It is typically not sensible (and for the most
interesting metrics $\chi $ in fact not possible) to take $\tilde{P}_{n}$ to
be the \textit{empirical measure} itself, but rather $\tilde{P}_{n}$ should
be some smoothed version of it. In this case, one cannot directly apply a
standard central limit theorem to (\ref{score}). However, recent results in
empirical process theory (Nickl (2007), Gin\'{e} and Nickl (2008, 2009b))
establish exactly such limit theorems for various density estimators.
Furthermore, these limit theorems also hold for density estimators that
simultaneously deliver optimal convergence rates in $L^{p}$-type loss
functions, which is potentially relevant for good control of the remainder
term. (We should note that this simultaneous optimality property is related
to what Bickel and Ritov (2003) label the 'plug-in property' of the density
estimator $\tilde{P}_{n}$, cf.~also Section 3 in Nickl (2007) for more
discussion.) Using similar methods we first prove a Beran-type result
(Theorem \ref{0815}), under quite weak (if not sharp) conditions, for the
case where $\chi =\chi _{F}$ (but with the unknown $p_{0}$ replaced by an
estimator), and where the underlying nonparametric estimator is based on a $%
\mathcal{L}^{2}$-projection of the empirical measure onto spaces of
piecewise polynomials spanned by dyadic $B$-splines.

\bigskip

Once asymptotic normality of the minimum distance estimator in (\ref{mindist}%
) is established, the question arises how the simulation step in (\ref%
{mindist1}) should be approached. Here two proof strategies arise:

\begin{enumerate}
\item The first method is to show that the objective function $\mathcal{Q}%
_{n,k}$ \textit{with} simulations is stochastically close, uniformly over $%
\Theta $, to the objective function $Q_{n}$ where no simulation is
performed. If 
\begin{equation}
\sup_{\theta \in \Theta }|\mathcal{Q}_{n,k}(\theta )-Q_{n}(\theta )|
\label{wurz}
\end{equation}%
has a sufficiently fast rate of convergence to zero (in probability), then
it is not difficult to show, using a result from Gach (2010), that the
asymptotic distribution of the simulated indirect inference estimator
obtained from minimizing (\ref{mindist1}) is the \textit{same} as the one of
the classical minimum distance estimator discussed in the previous
paragraph. It turns out that proving that the expression in (\ref{wurz}) has
a sufficiently fast rate of convergence to zero can be done by deriving
sharp bounds for the stochastic processes 
\begin{equation*}
\left\{ \sqrt{n}\int fd(\tilde{P}_{k}(\theta )-P(\theta ))\right\} _{\theta
\in \Theta ,f\in \mathcal{F}},
\end{equation*}%
where $\mathcal{F}$ is a relevant class of functions, and again we can apply
recent techniques from empirical processes here (cf. Nickl (2007), Gin\'{e}
and Nickl (2008, 2009b) together with moment inequalities in Gin\'{e} and
Koltchinskii (2006)). We prove that if one performs simulations of order $%
k>>n^{2}$, then the indirect inference estimators are asymptotically
equivalent to the classical minimum distance estimators. A main advantage of
this proof strategy is that \textit{no} differentiability properties of the
objective function $\mathcal{Q}_{n,k}$ have to be used, and that in turn a
large class of simulation mechanisms is admissible. More importantly, this
proof strategy allows for the presumably critical condition $\tau >1/2$ on
the underlying density $p_{0}$, where $\tau $ is the index governing the
regularity of $p_{0}$.

\item The method of proof described above works if many simulations are
performed ($k>>n^{2}$). However, this condition is not intrinsic to the
problem, and the case where the number of simulations $k$ is of a smaller
order than $n^{2}$ is also of interest. In particular, in the case where $%
k/n\rightarrow \kappa $, $0<\kappa <\infty $, one has to expect that the
asymptotic variance of simulated indirect inference estimators is inflated
by the factor $(1+1/\kappa )$. If one is interested in these cases, the
(comparably) `brute force' methods described in the previous paragraph
cannot be used. Alternatively, one can try to apply the usual $M$-estimation
asymptotic normality proof to the criterion function $\mathcal{Q}%
_{n,k}(\theta )$. Among other things this requires differentiation of the
simulated estimators $P_{k}(\theta )$ with respect to $\theta $. Since $%
P_{k}(\theta )$ is constructed by applying an \textit{approximate identity}
to the empirical measure from the simulated sample, the proofs become more
delicate in this case. [Differentiating an approximate identity $%
h^{-1}K(X(\theta )/h)$ w.r.t.~$\theta $ introduces a 'penalty' of an
additional $h^{-1}$ from the chain rule.] We are able, nevertheless, to
establish asymptotic normality of the simulated indirect inference estimator
with these simulation sizes as well, under slightly stronger conditions (on
the underlying density and the simulation mechanism), and with the expected
inflation of variances if $\lim_{n}k/n<\infty $. Again, the empirical
process techniques mentioned in the previous paragraphs, together with some
facts from approximation theory, are central to our proofs.
\end{enumerate}

We should comment on some related literature. Related papers are Gallant and
Long (1997) and Fermanian and Salani\'{e} (2004). The first paper studies
the case where $\tilde{P}_{n}$ is based on nonparametric MLEs over sieves
spanned by Hermite-polynomials, but their limiting result is only
informative if the sieve dimension stays bounded (so that efficiency of the
estimator is only established if the true density is a \textit{finite}
linear combination of Hermite-polynomials). Fermanian and Salani\'{e} (2004)
propose different (but somewhat related) procedures, and establish
asymptotic efficiency of their estimators under several high level
conditions, which, as they admit themselves, are very stringent. Even in the
simplest model they consider, they need to have simulations of order $k\sim
n^{6}$, and the nonparametric estimators considered seem to be only sensible
if the true density is very smooth. There are also some other related recent
papers on this topic, Altissimo and Mele (2009) and Carrasco, Chernov,
Florens, Ghysels (2007), whose proofs, however, we were not able to follow.

The outline of the paper is as follows: After some preliminaries in Section
2, we introduce the model and assumptions, define the auxiliary spline
projection estimators as well as the indirect inference estimator in Section
3 and present the main result (Theorem \ref{main}) on asymptotic efficiency
of the indirect inference estimator. Some basic facts on dyadic splines are
summarized in Section 4. Section 5 is devoted to the proof of Theorem \ref%
{main}. Section 6 develops auxiliary convergence rate results for the
auxiliary spline projection estimators needed in the proof of Theorem \ref%
{main}. Section 7 establishes a uniform central limit theorem for spline
projection estimators that is also essential in the proof of the main
result. Three appendices contain further technical results on Besov spaces,
projections onto Schoenberg spaces, and moment inequalities for empirical
processes.

\section{Preliminaries and Notation}

We denote the Euclidean norm of a vector $x\in \mathbb{R}^{b}$ by $%
\left\Vert x\right\Vert $ and the associated operator norm of a matrix $A$
by $\left\Vert A\right\Vert $. With $\mathcal{L}^{p}:=\mathcal{L}%
^{p}([0,1],\lambda )$, $1\leq p<\infty $, we denote the vector space of
Borel-measurable $p$-fold integrable real-valued functions on $[0,1]$, where 
$\lambda $ denotes Lebesgue measure on $[0,1]$, the (semi)norm on $\mathcal{L%
}^{p}$\ being denoted by $\left\Vert h\right\Vert _{p}$. Furthermore, $%
\left\Vert h\right\Vert _{\infty }$ stands for the supremum norm (\textit{not%
} the essential supremum norm) of a real-valued function $h$ defined on $%
[0,1]$. If $H$ is a vector- or matrix-valued function on $[0,1]$ then $%
\left\Vert H\right\Vert _{p}$ is shorthand for $\left\Vert \left\Vert
H\right\Vert \right\Vert _{p}$ and similarly for the supremum norm. By $%
\mathsf{L}^{\infty }$ we denote the space of all bounded Borel-measurable
real-valued functions on $[0,1]$ endowed with the supremum norm. For a
(measurable) real-valued function $g$ on $\mathbb{R}$ and $1\leq p<\infty $
we write $\left\Vert g\right\Vert _{p,\mathbb{R}}$ to denote its $\mathcal{L}%
^{p}$-(semi)norm (w.r.t. Lebesgue measure on $\mathbb{R}$); and we write $%
\left\Vert g\right\Vert _{\infty ,\mathbb{R}}$ for the supremum norm (not
the essential supremum norm). For sequences $a_{n}$ and $b_{n}$ of positive
real numbers we write $a_{n}\sim b_{n}$ to denote the fact that the sequence 
$a_{n}/b_{n}$ is bounded away from zero and infinity.

We next introduce Besov spaces. For a function $g:\mathbb{R}\rightarrow 
\mathbb{R}$ and $z\in \mathbb{R}$, the difference operator $\Delta _{z}$ is
defined by $\Delta _{z}g(\cdot )=g(\cdot +z)-g(\cdot )$ and inductively by $%
\Delta _{z}^{a}g(\cdot )=\Delta _{z}(\Delta _{z}^{a-1}g(\cdot ))$ for
integer $a\geq 2$. For $h:[0,1]\rightarrow \mathbb{R}$, we define $\Delta
_{z}^{a}(h)(x)$ as above if $x,x+az\in \lbrack 0,1]$, and set $\Delta
_{z}^{a}(h)(x)=0$ otherwise. For $0<s<\infty $ we define function spaces $%
\mathcal{B}_{s}$ on $[0,1]$ as follows.

\begin{definition}
\label{intrinsic}For $s\in (0,\infty )$, $a\in (s,\infty )\cap \mathbb{N}$,
and $h\in \mathcal{L}^{2}$ define 
\begin{equation*}
\Vert h\Vert _{s,2}:=\Vert h\Vert _{2}+\sup_{0\neq |z|<1}|z|^{-s}\Vert
\Delta _{z}^{a}(h)\Vert _{2}.
\end{equation*}%
Define further 
\begin{equation*}
\mathcal{B}_{s}:=\mathcal{B}_{2\infty }^{s}=\{h\in \mathcal{L}^{2}:\Vert
h\Vert _{s,2}<\infty \}.
\end{equation*}
\end{definition}

The space $\mathcal{B}_{s}$ does not depend on $a$ in the sense that
different choices of $a>s$ result in equivalent (semi)norms. For
definiteness we shall always choose $a$ to be the smallest integer larger
than $s$ in the sequel. It is well-known (Proposition \ref{elem} in Appendix %
\ref{App_Spline}) that for $s>1/2$ every function in $\mathcal{B}_{s}$ is $%
\lambda $-almost everywhere equal to a (uniquely determined) \textit{%
continuous} function in $\mathcal{B}_{s}$. It thus proves useful to define
for $s>1/2$ the Banach-space $(\mathsf{B}_{s},\Vert \cdot \Vert _{s,2})$
where $\mathsf{B}_{s}=\mathcal{B}_{s}\cap \mathsf{C}([0,1])$ and $\mathsf{C}%
([0,1])$ denotes the set of continuous real-valued functions on $[0,1]$.

A little reflection shows that $\mathcal{B}_{s}$ is just the usual Besov (or
generalized Lipschitz) space $\mathcal{B}_{2\infty }^{s}$ as, e.g., defined
in Chapter 2, Section 10 of DeVore and Lorentz (1993) (with the only
difference that there $\mathcal{B}_{s}$ is viewed as a space of equivalence
classes of functions). The space $\mathcal{B}_{s}$ contains the classical
Sobolev space of order $s$ as a subset. Recall that for \emph{integer} $s$
the Sobolev space of order $s>0$ is given by%
\begin{equation*}
\mathcal{W}_{2}^{s}=\left\{ h\in \mathcal{L}^{2}:D_{w}^{i}h\in \mathcal{L}%
^{2}\text{ for }0\leq i\leq s\text{, }i\text{ integer}\right\} ,
\end{equation*}%
where $D_{w}$ denotes the weak differential operator. Then for integer $s>0$%
\begin{equation}
\Vert h\Vert _{s,2}\leq C(s)\sum_{0\leq i\leq s}\Vert D_{w}^{i}h\Vert _{2}
\label{sob}
\end{equation}%
holds for some universal constant $C(s)$ and all $h$ in the Sobolev space of
order $s$; cf.~p.46 and p.52f in DeVore and Lorentz (1993). Some further
properties of Besov spaces and their relationship to splines that we shall
need in the sequel are summarized in Appendix \ref{App_Spline}.

\section{Main Results\label{indinf}}

Let $X_{1},\ldots ,X_{n}$ be independent and identically distributed
(i.i.d.) on a compact interval in $\mathbb{R}$ with law $P$ and
Lebesgue-density $p_{0}$. Without loss of generality we shall take this
interval to be $[0,1]$. We assume that a parametric model $\mathcal{P}%
_{\Theta }$ is given, i.e., $\mathcal{P}_{\Theta }=\{p(\theta ):\theta \in
\Theta \}$, where the \emph{functions} $p(\theta ):[0,1]\rightarrow \mathbb{R%
}$ are probability densities and the parameter space $\Theta $ is a subset
of $\mathbb{R}^{b}$. The probability measure on $[0,1]$ corresponding to $%
p(\theta )$ will be denoted by $P(\theta )$. We consider here the case where
direct likelihood methods for estimation of $\theta $ cannot be used for the
reasons outlined in the introduction. Suppose, however, that it is feasible
to obtain for each $\theta \in \Theta $ simulated data $X_{i}(\theta )$ via 
\begin{equation}
X_{i}(\theta )=\rho (V_{i},\theta ),\quad \quad i=1,...,k,  \label{Sim}
\end{equation}%
that are distributed i.i.d. with density $p(\theta )$ and that are
independent of the original sample. [The simulation mechanism may result
from an equation for the data as described in Section \ref{intro}, but may
also be obtained in some other way.] More precisely, we assume that the
random variables $V_{i}$ driving the simulation mechanism are i.i.d. with
values in some measurable space $(\mathcal{V},\mathfrak{V})$, the
distribution on $\mathcal{V}$ induced by $V_{i}$ being denoted by $\mu $;
furthermore, we assume that for every $\theta \in \Theta $, the $\mathfrak{V}
$-measurable function $\rho (\cdot ,\theta ):\mathcal{V}\rightarrow \lbrack
0,1]$ is such that the law of $\rho (V_{i},\theta )$ has density $p(\theta )$%
; and that the collection of random variables $\{V_{i}\}$ is independent of
the collection $\{X_{i}\}$. As the main result depends only on the \textit{%
distribution} of the random variables $X_{i}$ and $V_{i}$, we can assume
without loss of generality that the original data $X_{i}$ as well as the
variables $V_{i}$ are defined as the respective coordinate projections on
the product probability space $([0,1]^{\infty }\times \mathcal{V}^{\infty },%
\mathfrak{B}_{[0,1]}^{\infty }\otimes \mathfrak{V}^{\infty },P^{\infty
}\otimes \mu ^{\infty })$; we shall denote by $\Pr $ the product probability
measure $P^{\infty }\otimes \mu ^{\infty }$. The basic framework outlined
above will be maintained throughout the rest of the paper.

\begin{remark}
\label{Neu1}To avoid possible misunderstanding we note the following: (i)
Equation (\ref{Sim}) implies that one needs to obtain \emph{one and only} 
\emph{one }simulated sample $V_{1},\ldots ,V_{k}$ in order to compute $%
X_{i}(\theta )$ for any $\theta \in \Theta $. There is no need to separately
draw random samples for every $\theta $. (ii) Simulation mechanisms like (%
\ref{Sim}) naturally occur in the domain of application of indirect
inference which consists of statistical models where the data $X_{i}$ are
assumed to arise as the output of an equation that is parameterized by $%
\theta $ and is driven by some stochastic noise variables. These stochastic
noise variables then often play the r\^{o}le of $V_{i}$.
\end{remark}

We next construct auxiliary estimators for $p_{0}$ from the original data as
well as from the simulated data. The estimator of $p_{0}$ based on the
original data is a spline projection estimator based on B-splines of order $%
r_{\ast }\geq 1$ and is given by%
\begin{equation*}
p_{n,j,r_{\ast }}(y)=\sum_{l=-r_{\ast }+1}^{2^{j}-1}\hat{\gamma}%
_{lj}^{(r_{\ast })}N_{lj}^{(r_{\ast })}(y)
\end{equation*}%
with 
\begin{equation*}
\hat{\gamma}_{lj}^{(r_{\ast })}=\sum_{m=-r_{\ast
}+1}^{2^{j}-1}2^{j}g_{j}^{(r_{\ast })lm}\int_{[0,1]}N_{mj}^{(r_{\ast
})}(x)dP_{n}(x).
\end{equation*}%
Here $N_{lj}^{(r_{\ast })}$ denote the B-spline basis functions forming a
basis for the Schoenberg space $\mathcal{S}_{j}(r_{\ast })$ and the
coefficients $g_{j}^{(r_{\ast })lm}$ are the elements of $2^{-j}$ times the
inverse of the Gram matrix of the B-spline basis $N_{lj}^{(r_{\ast })}$; see
Section \ref{splines} for definitions. Furthermore, $P_{n}=n^{-1}%
\sum_{i=1}^{n}\delta _{X_{i}}$ denotes the empirical measure of the original
data. The positive integer $j$ represents a tuning parameter that governs
the dimension of the approximating space (`sieve') spanned by the B-spline
basis. Similarly, from each simulated data set $X_{i}(\theta )$, we
construct estimators for $p(\theta )$ based on order-$r$ B-splines via%
\begin{equation}
p_{k,J,r}(\theta )(y)=\sum_{l=-r+1}^{2^{J}-1}\hat{\gamma}_{lJ}^{(r)}(\theta
)N_{lJ}^{(r)}(y)  \label{aux_est_2a}
\end{equation}%
with%
\begin{equation}
\hat{\gamma}_{lJ}^{(r)}(\theta
)=\sum_{m=-r+1}^{2^{J}-1}2^{J}g_{J}^{(r)lm}%
\int_{[0,1]}N_{mJ}^{(r)}(x)dP_{k}(\theta )(x)  \label{aux_est_2b}
\end{equation}%
and $P_{k}(\theta )=k^{-1}\sum_{i=1}^{k}\delta _{X_{i}(\theta )}$. Note that 
$r_{\ast }$ and $r$ need not take the same value, nor need $j$ and $J$. [For
example, $r=4$ would correspond to using cubic splines for the construction
of $p_{k,J,r}(\theta )$, while $r_{\ast }=1$ would correspond to using the
Haar basis for the construction of $p_{n,j,r_{\ast }}$.] In the sequel we
shall often write $p_{k,J,r}(\theta ,y)$ for$\ p_{k,J,r}(\theta )(y)$ and
similarly $p(\theta ,x)$ for $p(\theta )(x)$.

The idea behind indirect inference is that, given the parametric model is
correctly specified in the sense that $p_{0}=p(\theta _{0})$ $\lambda $%
-almost everywhere for some $\theta _{0}\in \Theta $, the particular value
of $\theta $ corresponding to the simulation-based estimator $%
p_{k,J,r}(\theta )$ closest to $p_{n,j,r_{\ast }}$ (in an appropriate
metric) should provide a reasonable estimator $\hat{\theta}_{n,k}$ of $%
\theta _{0}$, since $p_{n,j,r_{\ast }}$ will estimate $p_{0}=p(\theta _{0})$
($\lambda $-a.e.) consistently (under appropriate assumptions and choices of 
$j$, $J$, and $k$). That is, as explained in Section \ref{intro}, the
estimator $\hat{\theta}_{n,k}$\ can be viewed as a simulation-based version
of a minimum distance estimator.

To implement this idea we introduce the indirect inference objective
function measuring closeness of $p_{n,j,r_{\ast }}$and $p_{k,J,r}(\theta )$%
\begin{equation}
\mathcal{Q}_{n,k}(\theta ):=\mathcal{Q}_{n,k,j,J,r_{\ast },r}(\theta
)=\left\{ 
\begin{array}{cc}
\int_{0}^{1}(p_{n,j,r_{\ast }}-p_{k,J,r}(\theta ))^{2}p_{n,j,r_{\ast
}}^{-1}d\lambda & \text{ \ on the event }A_{n} \\ 
0 & \text{otherwise}%
\end{array}%
\right. ,  \label{objective}
\end{equation}%
where $A_{n}=\left\{ p_{n,j_{n},r_{\ast }}(y)>0\text{ for every }y\in
\lbrack 0,1]\right\} $, which is measurable as is easily seen. Note that $%
\mathcal{Q}_{n,k}(\theta ):[0,1]^{\infty }\times \mathcal{V}^{\infty
}\rightarrow \mathbb{R}$ is $\mathfrak{B}_{[0,1]}^{\infty }\otimes \mathfrak{%
V}^{\infty }$-measurable for every $\theta \in \Theta $ as a consequence of
Tonelli's Theorem since $p_{n,j,r_{\ast }}$ and $p_{k,J,r}(\theta )$ are
both jointly measurable (w.r.t. the combined data and the argument $y$) and
since $A_{n}$ is measurable. Furthermore, since all functions involved are
piecewise polynomials with dyadic breakpoints, the integral featuring in the
definition of $\mathcal{Q}_{n,k}(\theta )$ can be computed in a numerically
efficient way.

\begin{remark}
(i) We have chosen to assign $\mathcal{Q}_{n,k}(\theta )$ the value zero on
the complement of $A_{n}$ for convenience. Since the event $A_{n}$ will be
seen to have probability approaching $1$ under our assumptions, this
particular assignment is irrelevant for asymptotic considerations. However,
from a more practical point of view, one might want to use the objective
function $\int_{p_{n,j,r_{\ast }}>0}(p_{n,j,r_{\ast }}-p_{k,J,r}(\theta
))^{2}p_{n,j,r_{\ast }}^{-1}d\lambda $ instead, which clearly coincides with 
$\mathcal{Q}_{n,k}$ on $A_{n}$.

(ii) In principle, auxiliary estimators other than spline projection
estimators could be used in the definition of $\mathcal{Q}_{n,k}(\theta )$.
We do not pursue this in this paper but see Gach (2010). We note that
standard kernel density estimators are inappropriate here because of
boundary effects.
\end{remark}

An \textit{indirect inference estimator} $\hat{\theta}_{n,k}:=\hat{\theta}%
_{n,k,j,J,r_{\ast },r}$ is now defined to be any measurable function that
satisfies 
\begin{equation}
\inf_{\theta \in \Theta }\mathcal{Q}_{n,k}(\theta )=\mathcal{Q}_{n,k}(\hat{%
\theta}_{n,k}).  \label{iie}
\end{equation}%
For the sake of simplicity, we shall use the abbreviation $\mathcal{Q}_{n,k}$
to denote $\mathcal{Q}_{n,k,j,J,r_{\ast },r}$ as well as $\mathcal{Q}%
_{n,k,j_{n},J_{k},r_{\ast },r}$, the precise meaning always being clear from
the context. [A similar comment applies to $\hat{\theta}_{n,k}$, as well as
to $Q_{n}$ and $\hat{\theta}_{n}$ defined later in Section \ref{intermed}.]
That such an estimator exists is shown in the next proposition, the proof of
which can be found in Appendix \ref{App_Meas}.

\begin{proposition}
\label{exist} Suppose $\Theta $ is compact in $\mathbb{R}^{b}$ and that the
simulation mechanism $\rho (v,\cdot )$ is continuous on $\Theta $ for every $%
v\in \mathcal{V}$. Furthermore, assume that $r_{\ast }\geq 1$ and $r\geq 2$
hold. Then there exists a $\mathfrak{B}_{[0,1]}^{\infty }\otimes \mathfrak{V}%
^{\infty }$-measurable mapping $\hat{\theta}_{n,k}$ satisfying (\ref{iie}).
\end{proposition}

\begin{remark}
\label{Neu2} (Computational issues) (i) As noted in Remark \ref{Neu1}, only
one sample of $V_{1},\ldots ,V_{k}$ needs to be drawn before $\hat{\gamma}%
_{lJ}^{(r)}(\theta )$ can be evaluated for any arbitrary $\theta \in \Theta $
via (\ref{aux_est_2b}). The computational costs for evaluating $\hat{\gamma}%
_{lJ}^{(r)}(\theta )$ are trivial.

(ii) The evaluation of the objective function $\mathcal{Q}_{n,k}(\theta )$
at an arbitrary $\theta \in \Theta $ is not computationally expensive
either: Note that in view of (\ref{aux_est_2a}) the objective function $%
\mathcal{Q}_{n,k}(\theta )$ can be written as a linear-quadratic form in the
variables $\hat{\gamma}_{lJ}^{(r)}(\theta )$ where the entries of the
weight-matrix and the coefficients of the linear part are integrals of
functions that do \emph{not }depend on $\theta $ (and are simple functions
of linear combinations of B-spline basis functions). Consequently, the
integrations have to be done only once and the evaluation of $\mathcal{Q}%
_{n,k}(\theta )$ then reduces to computation of the linear-quadratic form in
the variables $\hat{\gamma}_{lJ}^{(r)}(\theta )$.

(iii)\ Minimization of $\mathcal{Q}_{n,k}(\theta )$ over $\Theta $ is now a
standard optimization problem and has a level of computational complexity
comparable to computation of common (non-simulation-based) optimization
estimators. Standard techniques like grid-search, Newton-Raphson-type
procedures, or stochastic search procedures as in Beran and Millar (1987)
can be applied. Similarly as in the case of non-simulation-based
optimization estimators, it is in fact feasible to show that the estimators
generated by such a numerical procedure have the same asymptotic properties
as the estimator $\hat{\theta}_{n,k}$ under appropriate assumptions.
\end{remark}

We now introduce the following assumptions on the parametric model that will
be used to prove the main result.

\bigskip

\textbf{Assumption P1: (i) }The parameter space $\Theta $ is a compact
subset of $\mathbb{R}^{b}$. There exists a $\theta _{0}\in \Theta $ such
that $p_{0}=p(\theta _{0})$ $\lambda $-almost everywhere. Furthermore, $%
p(\theta )=p(\theta _{0})$ $\lambda $-almost everywhere implies $\theta
=\theta _{0}$. The mapping $\theta \mapsto p(\theta ,x)$ is continuous on $%
\Theta $ for every $x\in \lbrack 0,1]$. The density $p(\theta _{0})$ is
positive on $[0,1]$.

\textbf{(ii) }$\mathcal{P}_{\Theta }$ is a bounded subset of $\mathsf{B}%
_{\tau }$ for some $\tau >1/2$.

\textbf{(iii)} $\theta _{0}$ is an interior point of $\Theta $. There is an
open ball $B(\theta _{0})\subseteq \Theta $ with center $\theta _{0}$ such
that the map $\theta \mapsto p(\theta ,x)$ is twice continuously
differentiable on $B(\theta _{0})$ for every $x\in \lbrack 0,1]$.
Furthermore, 
\begin{equation*}
\int_{0}^{1}\sup_{\theta \in B(\theta _{0})}\left\Vert \nabla _{\theta
}p(\theta ,x)\right\Vert ^{2}dx<\infty ,\quad \int_{0}^{1}\sup_{\theta \in
B(\theta _{0})}\left\Vert \nabla _{\theta }^{2}p(\theta ,x)\right\Vert
dx<\infty ,
\end{equation*}%
and $\int_{0}^{1}\nabla _{\theta }p(\theta _{0},x)\nabla _{\theta }p(\theta
_{0},x)^{\prime }p(\theta _{0},x)^{-1}dx$ is positive definite. [Here $%
\nabla _{\theta }$ denotes the gradient w.r.t. $\theta $ written as a column
vector and $\nabla _{\theta }^{2}$ denotes the matrix of second derivatives.]

\textbf{(iv)} For some $\varsigma >1/2$%
\begin{equation*}
\frac{\partial p(\theta _{0},\cdot )}{\partial \theta _{q}}\in \mathsf{B}%
_{\varsigma }
\end{equation*}%
holds for every $q=1,...,b$.

\bigskip

Assumption P1(i) is a standard assumption that implies consistency of the
maximum likelihood estimator. In particular, it expresses the fact that the
parametric model is correctly specified and that the true parameter value is
identifiable. Assumption P1(iii) in conjunction with P1(i) is a typical
assumption used to establish asymptotic normality of the maximum likelihood
estimator and the information matrix equality. Assumption P1(ii) requires
the parametric density functions to behave "regularly" as functions of $x$
(uniformly in $\theta $), the condition being quite weak: Note that if $\tau 
$ is close to $1/2$ the density functions are not even required to be
differentiable, all that is required is essentially that the functions are "$%
\mathcal{L}^{2}$-H\"{o}lder continuous" of order $\tau $, uniformly over $%
\theta $. [Given compactness of $\Theta $, a sufficient condition for
Assumption P1(ii) is that $\mathcal{P}_{\Theta }\subseteq \mathsf{B}_{\tau }$
for some $\tau >1/2$ and that the map $\theta \rightarrow p(\theta )$ from $%
\Theta $ to $\mathsf{B}_{\tau }$ is continuous; in fact, continuity of the
map $\theta \rightarrow \Vert p(\theta )\Vert _{\tau ,2}$ already suffices.
A simple sufficient condition for this (with $\tau =1$) is continuity of $%
\theta \rightarrow \Vert p(\theta )\Vert _{2}$ and $\theta \rightarrow \Vert
D_{w}p(\theta )\Vert _{2}$ on $\Theta $, cf. (\ref{sob}).] In a similar
vein, Assumption P1(iv) imposes an analogous weak regularity condition on
the derivative of $p(\theta )$ (w.r.t. $\theta $) at $\theta =\theta _{0}$.

For parts of the main result we will need to supplement assumption P1 by the
following assumption.

\bigskip

\textbf{Assumption P2: (i) }The set $\left\{ \frac{\partial p(\theta ,\cdot )%
}{\partial \theta _{q}}:q=1,\ldots ,b,\ \theta \in B(\theta _{0})\right\} $
is a relatively compact subset of $\mathcal{L}^{2}$ where $B(\theta _{0})$
is defined in Assumption P1.

\textbf{(ii) }The set $\left\{ \frac{\partial ^{2}p(\theta ,\cdot )}{%
\partial \theta _{q}\partial \theta _{q^{\prime }}}:q,q^{\prime }=1,\ldots
,b,\ \theta \in B(\theta _{0})\right\} $ is a bounded subset of $\mathcal{L}%
^{2}$, i.e.,%
\begin{equation*}
\sup_{\theta \in B(\theta _{0})}\int_{0}^{1}\left\Vert \nabla _{\theta
}^{2}p(\theta ,x)\right\Vert ^{2}dx<\infty .
\end{equation*}

\bigskip

These assumptions are not restrictive. For example, Assumption P2(i) is
satisfied if the indicated set of functions is a bounded subset of a Besov
space $\mathcal{B}_{s}$ with $s$ only satisfying $s>0$, which is a very weak
condition.

We also need assumptions on the simulation mechanism $\rho $. The basic
assumption will be that the function $\rho $ satisfies a H\"{o}lder
continuity condition in $\theta $ (Assumption R(i)). For some of the results
we shall need an additional assumption including twice differentiability in
a neighborhood of $\theta _{0}$ (Assumption R(ii)).

\bigskip

\textbf{Assumption R:} \textbf{(i)} The function $\rho $ is uniformly H\"{o}%
lder in $\theta $, more precisely, for some $0<L<\infty $ and some $0<\alpha
\leq 1$ 
\begin{equation*}
\sup_{v\in \mathcal{V}}\left\vert \rho (v,\theta )-\rho (v,\theta ^{\prime
})\right\vert \leq L\left\Vert \theta -\theta ^{\prime }\right\Vert ^{\alpha
}
\end{equation*}%
holds for all $\theta $, $\theta ^{\prime }\in \Theta $.

\textbf{(ii)} There is an open ball $B(\theta _{0})\subseteq \Theta $ with
center $\theta _{0}$ such that the map $\theta \rightarrow \rho (v,\theta )$
is twice continuously differentiable on $B(\theta _{0})$ for every $v\in 
\mathcal{V}$ and 
\begin{equation*}
\sup_{v\in \mathcal{V},\theta \in B(\theta _{0})}\Vert \nabla _{\theta }\rho
(v,\theta )\Vert <\infty ,\text{ \ \ \ }\sup_{v\in \mathcal{V},\theta \in
B(\theta _{0})}\Vert \nabla _{\theta }^{2}\rho (v,\theta )\Vert <\infty .
\end{equation*}%
Furthermore, for some $0<L^{\prime }<\infty $ and some $0<\beta \leq 1$%
\begin{equation*}
\sup_{v\in \mathcal{V}}\left\Vert \nabla _{\theta }^{2}\rho (v,\theta
)-\nabla _{\theta }^{2}\rho (v,\theta ^{\prime })\right\Vert \leq L^{\prime
}\left\Vert \theta -\theta ^{\prime }\right\Vert ^{\beta }
\end{equation*}%
holds for all $\theta $, $\theta ^{\prime }\in B(\theta _{0})$.

\bigskip

Assumptions on the parametric model $\mathcal{P}_{\Theta }$ and assumptions
on the simulation mechanism $\rho $ are of course interrelated. For example,
one could in principle only impose appropriate assumptions on $\rho $ and
then deduce the existence of a $\mathcal{P}_{\Theta }$ with the required
properties from those assumptions; see Gach (2010) for some discussion.
However, as this does not seem to lead to a transparent catalogue of
assumptions, we have chosen to formulate the assumptions in the form given
above.

We now first establish consistency of the indirect inference estimator. The
assumptions used for the consistency result in the subsequent proposition
are stronger than what is actually needed for such a result, but we do not
strive for utmost generality in the consistency result as this is not the
main focus of the paper. The proof is given in Section \ref{proof_consist}.

\begin{proposition}
\label{consist}Suppose Assumptions P1(i),(ii) and R(i) are satisfied and
that $r_{\ast }\geq 2$ and $r\geq 2$ hold. If $j_{n}\rightarrow \infty $ as $%
n\rightarrow \infty $ and $J_{k}\rightarrow \infty $ as $k\rightarrow \infty 
$ in such a way that for some $\delta >1/2$ we have $\sup_{n\geq
1}2^{j_{n}(2\delta +1)}/n<\infty $ and $\sup_{k\geq 1}J_{k}2^{J_{k}(2\delta
+1)}/k<\infty $, then%
\begin{equation*}
\hat{\theta}_{n,k}\rightarrow \theta _{0}\text{ in }\Pr \text{-probability
as }n\wedge k\rightarrow \infty .
\end{equation*}
\end{proposition}

We note that the condition on $j_{n}$ is, e.g., satisfied if $2^{j_{n}}\sim
n^{\psi }$ with $0<\psi <1/2$. A similar comment applies to $J_{k}$. In
particular, the `textbook'-choice $\psi =1/(2\tau +1)$ with $\tau $ from
Assumption P1(ii) is covered.

For the main result we need to distinguish several cases characterized by
the behavior of the number $k(n)\in \mathbb{N}$ of simulated data as a
function of sample size $n$:

\bigskip

\textbf{Assumption S1:} $\lim_{n\rightarrow \infty }k(n)/n^{2}=\infty $.

\textbf{Assumption S2:} $\lim_{n\rightarrow \infty }k(n)/n=\infty $.

\textbf{Assumption S3:} $\lim_{n\rightarrow \infty }k(n)/n=\kappa $ for some 
$0<\kappa <\infty $.

\bigskip

The theorem given below is the main result and shows that, under appropriate
conditions on the resolution levels $j_{n}$ and $J_{k}$, the indirect
inference estimator $\hat{\theta}_{n,k}$\ is asymptotically normal and has
the same limiting distribution as the maximum likelihood estimator provided
the number $k(n)$ of simulated data grows sufficiently fast as a function of
sample size $n$. This is established under the quite weak assumption R(i) if 
$k(n)$ grows faster than $n^{2}$. If $k(n)$ is only required to grow faster
than $n$, the same result is obtained under somewhat stronger assumptions
(Assumption R, $\tau >3/2$, $r\geq 4$). Under the latter assumptions, the
theorem also shows that in case $k(n)$ behaves asymptotically like $n$, the
indirect inference estimator is still asymptotically normal but its
asymptotic variance covariance matrix is then inflated by a factor $%
1+1/\kappa $, where $\kappa =\lim_{n\rightarrow \infty }k(n)/n$. We also
note that the condition $\tau <r_{\ast }\wedge r$ in the subsequent theorem
is virtually no restriction as discussed in Remark \ref{rem_main} below. The
proof of the subsequent theorem is deferred to Section \ref{proof_main}.

\begin{theorem}
\label{main}Suppose $r\geq 2$ and $r_{\ast }\geq 2$ hold and Assumption P1
is satisfied for some $1/2<\tau <r_{\ast }\wedge r$. Suppose that $%
2^{j_{n}}\sim n^{1/(2\tau +1)}$ and $2^{J_{k(n)}}\sim k(n)^{1/(2\tau +1)}$.

a. Suppose one of the following two conditions holds:

\qquad 1. Assumptions R(i) and S1 hold.

\qquad 2. Assumptions P2, R, and S2 hold, and that $\tau >3/2$, $r\geq 4$
are satisfied.

Then 
\begin{equation*}
\sqrt{n}\left( \hat{\theta}_{n,k(n)}-\theta _{0}\right) \rightarrow
^{d}N(0,I(\theta _{0}))
\end{equation*}%
as $n\rightarrow \infty $ where $I(\theta _{0})=\left( \int_{0}^{1}\nabla
_{\theta }p(\theta _{0},x)\nabla _{\theta }p(\theta _{0},x)^{\prime
}p(\theta _{0},x)^{-1}dx\right) ^{-1}$ is the Cram\'{e}r-Rao bound.

b. Suppose Assumptions P2, R, and S3 hold for some $0<\kappa <\infty $, and
that $\tau >3/2$, $r\geq 4$ are satisfied. Then 
\begin{equation*}
\sqrt{n}\left( \hat{\theta}_{n,k(n)}-\theta _{0}\right) \rightarrow
^{d}N\left( 0,(1+1/\kappa )I(\theta _{0})\right)
\end{equation*}%
as $n\rightarrow \infty $.
\end{theorem}

We note that the rates of increase for $2^{j_{n}}$ and $2^{J_{k(n)}}$
specified in the above theorem are precisely the rate-optimal choices based
on mean integrated squared error. As already alluded to prior to the
theorem, in Part a of the theorem there is a trade-off between the
stringency of assumptions on the model and the simulation mechanism on the
one hand and the assumptions on the rate of increase of $k(n)$ (Assumptions
S1 versus S2) on the other hand. While the particular form of the trade-off
is a consequence of two different methods of proof employed for Part a1 and
Part a2 (and thus may in principle be an artefact), it seems plausible that
some sort of trade-off is intrinsic to the problem.

\begin{remark}
\label{rem_main}(i) The condition $\tau <r_{\ast }\wedge r$ in the above
theorem is not really a restriction on $\mathcal{P}_{\Theta }$ and can
always be achieved in the following sense: If Assumption P1 holds with $\tau
\geq r_{\ast }\wedge r$, it holds with $\tau $ replaced by any $\tau
^{\prime }$ satisfying $1/2<\tau ^{\prime }<r_{\ast }\wedge r$ as well,
since $\mathsf{B}_{\tau }$ is continuously imbedded in $\mathsf{B}_{\tau
^{\prime }}$ for $\tau ^{\prime }\leq \tau $. Consequently, the above
theorem can be applied with $\tau ^{\prime }$ replacing $\tau $ (requiring
also $\tau ^{\prime }>3/2$ for Parts a2 and b). [The restriction $\tau
<r_{\ast }\wedge r$ in the theorem simply expresses the fact that the rate
of increase of $j_{n}$ and $J_{k}$ is not only governed by the degree of
"regularity" $\tau $ of the densities in $\mathcal{P}_{\Theta }$, but also
by the degrees of "regularity" of the splines used to estimate $p_{0}$ and $%
p(\theta )$, respectively, i.e., by $r_{\ast }$ and $r$.]

(ii) The argument underlying (i) also shows that $2^{j_{n}}\sim n^{1/(2\tau
^{\prime }+1)}$ and $2^{J_{k(n)}}\sim k(n)^{1/(2\tau ^{\prime }+1)}$ are
feasible in Theorem \ref{main} as it stands as long as $1/2<\tau ^{\prime
}\leq \tau $ (and $\tau ^{\prime }>3/2$ for Parts a2 and b) are satisfied. A
careful examination of the proof shows that the range for $2^{j_{n}}$ and $%
2^{J_{k(n)}}$, under which the conclusion of the theorem holds, is actually
somewhat wider. However, we abstain from providing such results as they
quickly get unwieldy.

(iii) If in Part a2 of Theorem \ref{main} the Assumption S2 is strengthened
by assuming a particular growth-rate for $k(n)$ such as, e.g., $%
k(n)=n^{\delta }$, $1<\delta \leq 2$, this can be used to relax the
assumption $\tau >3/2$. We refrain from presenting such results.

(iv) If $k(n)$ is such that $0<\liminf k(n)/n<\infty $, but $\limsup
k(n)/n=\infty $, then the distribution $\sqrt{n}\left( \hat{\theta}%
_{n,k(n)}-\theta _{0}\right) $ does not possess a limit, but `oscillates'
between accumulation points of the form $N\left( 0,I(\theta _{0})\right) $
and $N\left( 0,(1+1/\kappa )I(\theta _{0})\right) $ where now $\kappa
=\liminf_{n\rightarrow \infty }k(n)/n$.

(v) A result similar to Part a1 of Theorem \ref{main} can be proved in case $%
r^{\ast }=1$. Since this requires a separate proof, we do not give such a
result for the sake of brevity.
\end{remark}

Under Assumption P1 the expression $\Psi (\theta )=\int_{0}^{1}\nabla
_{\theta }p(\theta )\nabla _{\theta }p(\theta )^{\prime }p(\theta
)^{-1}d\lambda $ depends continuously on $\theta $ by dominated convergence.
Hence, $\Psi (\bar{\theta})^{-1}$ is a consistent estimator for $I(\theta
_{0})$ for every consistent estimator $\bar{\theta}$. However, this
observation is not very helpful in the context of indirect inference as then
expressions for the density $p(\theta )$ are typically not available. An
alternative consistent estimator that is feasible to compute is described in
the next proposition which is proved in Section \ref{proof_vc}. In the
following proposition let $\bar{\theta}_{n,k}$ stand for an arbitrary
consistent estimator that depends on the original data and perhaps also on
the simulated data. Of course, under the assumptions of Proposition \ref%
{consist} we may take $\bar{\theta}_{n,k}=\hat{\theta}_{n,k}$.

\begin{proposition}
\label{vc}Suppose Assumptions P1(i)-(iii), P2(i), and R(ii) hold. Suppose
further that $\bar{\theta}_{n,k}\rightarrow \theta _{0}$ in probability as $%
n\wedge k\rightarrow \infty $. Assume $r_{\ast }^{\prime }\geq 2$ and $%
r^{\prime }\geq 3$. If $j_{n}^{\prime }\rightarrow \infty $ as $n\rightarrow
\infty $ and $J_{k}^{\prime }\rightarrow \infty $ as $k\rightarrow \infty $
in such a way that for some $\delta >1/2$ we have $\sup_{n\geq
1}2^{j_{n}^{\prime }(2\delta +1)}/n<\infty $ and also $J_{k}^{\prime
}2^{3J_{k}^{\prime }}/k\rightarrow 0$, then 
\begin{equation*}
\left( \int_{0}^{1}\nabla _{\theta }p_{k,J_{k}^{\prime },r^{\prime }}(\bar{%
\theta}_{n,k})\nabla _{\theta }p_{k,J_{k}^{\prime },r^{\prime }}(\bar{\theta}%
_{n,k})^{\prime }p_{n,j_{n},r_{\ast }^{\prime }}^{-1}d\lambda \right) ^{-1}
\end{equation*}%
is well-defined on an event that has probability converging to $1$, and is a
consistent estimator for $I(\theta _{0})$ as $n\wedge k\rightarrow \infty $.
\end{proposition}

Observe that the condition on $j_{n}^{\prime }$ is satisfied if $%
2^{j_{n}^{\prime }}\sim n^{\psi }$ with $0<\psi <1/2$; similarly, the
condition on $J_{k}^{\prime }$ is satisfied if $2^{J_{k}^{\prime }}\sim
n^{\psi }$ with $0<\psi <1/3$. The reason for allowing $r^{\prime }$ to
differ from $r$ in Theorem \ref{main}, is to be able to construct a
consistent estimator for $I(\theta _{0})$ also in cases where $r=2$.
Allowing $J_{k}^{\prime }$ to be different from $J_{k}$ has the advantage of
avoiding a constraint on $\tau $.

\section{Dyadic Splines\label{splines}}

Let $T_{j}=\{t_{l}:=l2^{-j}:l=1,\ldots ,2^{j}-1\}$ be a dyadic set of knots
in $[0,1]$, where $j\in \mathbb{N}$, the set of nonnegative integers. A
function $S:[0,1]\rightarrow \mathbb{R}$ is a (dyadic) spline of order $%
r\geq 2$ if on each of the intervals $[0,t_{1})$, $(t_{l},t_{l+1})$ for $%
l=1,\ldots ,2^{j}-2$, and $(t_{2^{j}-1},1]$, it is a polynomial of degree
not larger than $r-1$, and on at least one of the intervals it is a
polynomial of degree exactly $r-1$. The Schoenberg spaces $\mathcal{S}%
_{j}(r) $ considered here consist of all splines of order less than or equal
to $r$ that are $r-2$ times continuously differentiable on $[0,1]$ (using
one-sided derivatives on the boundary of $[0,1]$). For $r=1$ we define the
Schoenberg space $\mathcal{S}_{j}(1)$ to be the space of all functions $%
S:[0,1]\rightarrow \mathbb{R}$ that are constant on the intervals $[0,t_{1})$%
, $[t_{l},t_{l+1})$ for $l=1,\ldots ,2^{j}-2$, and $[t_{2^{j}-1},1]$. The
Schoenberg spaces are linear spaces of dimension $2^{j}+r-1$. For $r\geq 2$
the B-spline basis for $\mathcal{S}_{j}(r)$ is given by $%
\{N_{lj}^{(r)}:l=-r+1,\ldots ,0,1,\ldots ,2^{j}-1\}$ with 
\begin{equation*}
N_{lj}^{(r)}(x)=N^{(r)}(2^{j}x-l)\qquad \text{for }x\in \lbrack 0,1],
\end{equation*}%
where $N^{(r)}$ is the B-spline-function (of order $r$) given by the $r$%
-fold convolution%
\begin{equation*}
N^{(r)}(u)=\boldsymbol{1}_{[0,1)}\ast ...\ast \boldsymbol{1}%
_{[0,1)}(u)\qquad \text{for }u\in \mathbb{R};
\end{equation*}%
cf., e.g., Chapter 5 in DeVore and Lorentz (1993). In case $r=1$ we set 
\begin{equation*}
N_{lj}^{(1)}(x)=N^{(1)}(2^{j}x-l)\qquad \text{for }x\in \lbrack 0,1],
\end{equation*}%
for $l=0,1,\ldots ,2^{j}-2$, where $N^{(1)}(u)=\boldsymbol{1}_{[0,1)}(u)$,
but we set 
\begin{equation*}
N_{lj}^{(1)}(x)=\boldsymbol{1}_{[0,1]}(2^{j}x-l)\qquad \text{for }x\in
\lbrack 0,1]
\end{equation*}%
if $l=2^{j}-1$. The B-spline basis functions $N_{lj}^{(r)}$ are nonnegative,
bounded by $1$ in absolute value, and form a partition of unity, i.e., 
\begin{equation}
\sum_{l=-r+1}^{2^{j}-1}N_{lj}^{(r)}(x)=1\qquad \text{for }x\in \lbrack 0,1],
\label{unity}
\end{equation}%
for every $j,r\in \mathbb{N}$.

The Schoenberg space $\mathcal{S}_{j}(r)$ is a finite-dimensional linear
subspace of $\mathcal{L}^{2}$. The ortho-projection $\pi _{j}^{(r)}$ from $%
\mathcal{L}^{2}$ onto $\mathcal{S}_{j}(r)$ is given by 
\begin{equation*}
\pi _{j}^{(r)}(f)=\sum_{l=-r+1}^{2^{j}-1}\gamma _{lj}^{(r)}(f)N_{lj}^{(r)}
\end{equation*}%
where 
\begin{equation*}
\gamma
_{lj}^{(r)}(f)=\sum_{m=-r+1}^{2^{j}-1}2^{j}g_{j}^{(r)lm}%
\int_{0}^{1}N_{mj}^{(r)}(x)f(x)dx
\end{equation*}%
and $g_{j}^{(r)lm}$ is the $(l,m)$-element of the inverse of the $%
(2^{j}+r-1)\times (2^{j}+r-1)$ matrix 
\begin{equation*}
G_{j}^{(r)}=\left( \int_{0}^{2^{j}}N^{(r)}(u-l)N^{(r)}(u-m)du\right) _{l,m}.
\end{equation*}%
Note that $G_{j}^{(r)}$ is a symmetric bandmatrix with bandwidth $r$. The
projection can now also be written as 
\begin{equation}
\pi _{j}^{(r)}(f)(y)=\int_{0}^{1}K_{j}^{(r)}(x,y)f(x)dx  \label{proj_kern}
\end{equation}%
with the kernel given by 
\begin{equation*}
K_{j}^{(r)}(x,y)=2^{j}\sum_{l=-r+1}^{2^{j}-1}%
\sum_{m=-r+1}^{2^{j}-1}g_{j}^{(r)lm}N^{(r)}(2^{j}x-m)N^{(r)}(2^{j}y-l).
\end{equation*}%
We shall frequently need to bound the maximal row-sum of the absolute values
of the elements of the inverse of $G_{j}^{(r)}$, i.e., the $\ell ^{\infty }$%
-operator norm of the inverse of $G_{j}^{(r)}$. For this we use the
following special case of a result in Shadrin (2001, Theorem I and Section
4.2).

\begin{proposition}
\label{gram}For every $r\in \mathbb{N}$ there exist constants $%
0<d_{r}<\infty $ (independent of $j)$ such that for every $j\in \mathbb{N}$%
\begin{equation*}
\left\Vert \left( G_{j}^{(r)}\right) ^{-1}\right\Vert _{\infty \rightarrow
\infty }\leq d_{r}
\end{equation*}%
where $\left\Vert \cdot \right\Vert _{\infty \rightarrow \infty }$ denotes
the $\ell ^{\infty }$-operator norm on $\mathbb{R}^{2^{j}+r-1}$.
\end{proposition}

We furthermore note that for $r\geq 2$ the Schoenberg space $\mathcal{S}%
_{j}(r)$ is contained in the Sobolev space of order $r-1$, and thus is also
contained in $\mathsf{B}_{r-1}$. In fact, for every $r\geq 1$ we have that $%
\mathcal{S}_{j}(r)$ is contained in $\mathsf{B}_{s}$ for $s\leq r-1/2$
(DeVore and Lorentz (1993), Chap. 12, Lemma 3.1). Some approximation
properties of splines that we shall use in the sequel are summarized in
Appendix \ref{App_Spline}.

For the spline projection estimators defined in Section \ref{indinf} we make
the useful observation that for every $J\geq 1$ and $r\geq 1$%
\begin{equation}
\left\Vert p_{k,J,r}(\theta )\right\Vert _{\infty }\leq 2^{J}d_{r}(2^{J}+r-1)
\label{sup-norm_estim}
\end{equation}%
holds uniformly in $\theta \in \Theta $, $k\geq 1$, and $v_{1},\ldots
,v_{k}\in \mathcal{V}$. [To see this note that the B-spline basis functions
are uniformly bounded by $1$ and that the coefficients satisfy $\left\vert 
\hat{\gamma}_{lJ}^{(r)}(\theta )\right\vert \leq 2^{J}d_{r}$ uniformly in $%
\theta \in \Theta $, $k\geq 1$, $-r+1\leq l\leq 2^{J}-1$, and $v_{1},\ldots
,v_{k}\in \mathcal{V}$ by Proposition \ref{gram}.] The analogous relation is
true for $\left\Vert p_{n,j,r_{\ast }}\right\Vert _{\infty }$, as well as
for $\left\Vert Ep_{k,J,r}(\theta )\right\Vert _{\infty }$ and $\left\Vert
Ep_{n,j,r_{\ast }}\right\Vert _{\infty }$.

\section{Proofs\label{proof_main}}

We shall use repeatedly in this section the fact that $\xi _{0}:=\inf_{x\in
\lbrack 0,1]}p(\theta _{0},x)>0$ under Assumptions P1(i),(ii) (as $p(\theta
_{0})$ is continuous and positive on $[0,1]$ under these assumptions).

\subsection{Proof of Proposition \protect\ref{consist}\label{proof_consist}}

Define the function 
\begin{equation}
Q(\theta )=\int_{0}^{1}(p(\theta _{0})-p(\theta ))^{2}p^{-1}(\theta
_{0})d\lambda ,  \label{Q}
\end{equation}%
which is real-valued and is continuous in $\theta $ by dominated
convergence, observing that $\xi _{0}>0$ and that Assumption P1(ii) implies
sup-norm boundedness of $\mathcal{P}_{\Theta }$ in view of the discussion
following Proposition \ref{elem} in Appendix \ref{App_Spline}. The unique
minimizer of $Q(\theta )$ over $\Theta $ is $\theta _{0}$ in view of the
identifiability assumption made in Assumption P1(i). To establish
consistency, it is hence sufficient to prove 
\begin{equation*}
\sup_{\theta \in \Theta }\left\vert \mathcal{Q}_{n,k}(\theta )-Q(\theta
)\right\vert \rightarrow 0
\end{equation*}%
in probability as $n\wedge k\rightarrow \infty $. Note that this supremum is
measurable as $\mathcal{Q}_{n,k}(\theta )$ and $Q(\theta )$ are continuous
and $\Theta $ is separable. [For continuity of $\mathcal{Q}_{n,k}$ see the
proof of Proposition \ref{exist} in Appendix \ref{App_Meas}.] Consider the
set $A_{n}^{\ast }=\left\{ \inf_{y\in \lbrack 0,1]}p_{n,j_{n},r_{\ast
}}(y)\geq \xi _{0}/2\right\} $, which is clearly measurable. Since $\xi
_{0}>0$ as noted above, Corollary \ref{consistency_2} (applied with $%
t=\delta \wedge \tau \wedge 1$ and noting that $p(\theta _{0})$ is a
continuous version of $p_{0}$ in view of Assumption P1(i)) implies that $\Pr
(A_{n}^{\ast })\rightarrow 1$ as $n\rightarrow \infty $. A simple
calculation now shows that on the event $A_{n}^{\ast }$ (since $A_{n}^{\ast
}\subseteq A_{n}$)%
\begin{eqnarray*}
\mathcal{Q}_{n,k}(\theta )-Q(\theta ) &=&\int_{0}^{1}(p_{n,j_{n},r_{\ast
}}-p(\theta _{0}))\left[ 1-\frac{p(\theta )^{2}}{p_{n,j_{n},r_{\ast
}}p(\theta _{0})}\right] d\lambda +\int_{0}^{1}(p_{k,J_{k},r}(\theta
)-p(\theta ))^{2}p_{n,j_{n},r_{\ast }}^{-1} \\
&&+2\int_{0}^{1}(p_{k,J_{k},r}(\theta )-p(\theta ))\left[ \frac{p(\theta )}{%
p_{n,j_{n},r_{\ast }}}-1\right] d\lambda
\end{eqnarray*}%
holds.\ On $A_{n}^{\ast }$ we can then obtain the bound%
\begin{eqnarray*}
\sup_{\theta \in \Theta }\left\vert \mathcal{Q}_{n,k}(\theta )-Q(\theta
)\right\vert &\leq &\left\Vert p_{n,j_{n},r_{\ast }}-p(\theta
_{0})\right\Vert _{\infty }\left( 1+2\xi _{0}^{-2}\sup_{\theta \in \Theta
}\left\Vert p(\theta )\right\Vert _{\infty }^{2}\right) \\
&&+2\xi _{0}^{-1}\sup_{\theta \in \Theta }\left\Vert p_{k,J_{k},r}(\theta
)-p(\theta )\right\Vert _{\infty }^{2} \\
&&+\sup_{\theta \in \Theta }\left\Vert p_{k,J_{k},r}(\theta )-p(\theta
)\right\Vert _{\infty }\left( 2+4\xi _{0}^{-1}\sup_{\theta \in \Theta
}\left\Vert p(\theta )\right\Vert _{\infty }\right) .
\end{eqnarray*}%
The sup-norm boundedness of $\mathcal{P}_{\Theta }$ together with
Corollaries \ref{consistency} and \ref{consistency_2} (applied with $%
t=\delta \wedge \tau \wedge 1$) then complete the proof.

\subsection{An Intermediate Result\label{intermed}}

Consider the objective function%
\begin{equation}
Q_{n}(\theta ):=Q_{n,j,r_{\ast }}(\theta )=\left\{ 
\begin{array}{cc}
\int_{0}^{1}\left( p_{n,j,r_{\ast }}-p(\theta )\right) ^{2}p_{n,j,r_{\ast
}}^{-1}d\lambda & \text{ \ on the event }A_{n} \\ 
0 & \text{otherwise}%
\end{array}%
\right. ,  \label{obj_inter}
\end{equation}%
corresponding to the `ideal' case $k=\infty $. Let $\hat{\theta}_{n}:=\hat{%
\theta}_{n,j,r_{\ast }}$ denote an arbitrary measurable minimizer of (\ref%
{obj_inter}) over $\Theta $. [The existence of such an estimator is
established in Proposition \ref{exist_1} in Appendix \ref{App_Meas}.]

\begin{theorem}
\label{0815}Suppose $r_{\ast }\geq 2$ holds and Assumption P1 is satisfied
with $1/2<\tau <r_{\ast }$. If $2^{j_{n}}\sim n^{1/(2\tau +1)}$, then, as $%
n\rightarrow \infty $, 
\begin{equation*}
\sqrt{n}\left( \hat{\theta}_{n}-\theta _{0}\right) \rightarrow ^{d}N\left(
0,I(\theta _{0})\right) .
\end{equation*}
\end{theorem}

\begin{proof}
Consistency of $\hat{\theta}_{n}$ follows from Proposition \ref{cons} in
Appendix \ref{App_Meas} by choosing $\delta $ in that proposition
sufficiently close to $1/2$. It follows that $\hat{\theta}_{n}\in B(\theta
_{0})$ with probability tending to $1$, and hence $\hat{\theta}_{n}$ belongs
to the interior of $\Theta $ with probability tending to $1$. In the
following we work only on the intersection of the event $\left\{ \hat{\theta}%
_{n}\in B(\theta _{0})\right\} $ with $A_{n}^{\ast }=\left\{ \inf_{y\in
\lbrack 0,1]}p_{n,j_{n},r_{\ast }}(y)\geq \xi _{0}/2\right\} $ which also
has probability converging to $1$ as a consequence of Corollary \ref%
{consistency_2} (applied with some $t$ satisfying $1/2<t\leq \tau \wedge 1$%
). Note that $\left\Vert p_{n,j_{n},r_{\ast }}\right\Vert _{\infty }<\infty $
holds, and that $\left\Vert p_{n,j_{n},r_{\ast }}^{-1}\right\Vert _{\infty
}\leq 2/\xi _{0}$ on the event $A_{n}^{\ast }$. Furthermore, by Assumption
P1(ii) the function $p(\theta )$ is bounded, uniformly in $\theta $, cf.
Proposition \ref{elem} and the attending discussion in Appendix \ref%
{App_Spline}. Assumption P1(iii) and dominated convergence then show that $%
Q_{n}(\theta )$ is twice continuously differentiable on the open ball $%
B(\theta _{0})$ with derivatives given by%
\begin{equation*}
\nabla _{\theta }Q_{n}(\theta )=-2\int_{0}^{1}\left( p_{n,j_{n},r_{\ast
}}-p(\theta )\right) p_{n,j_{n},r_{\ast }}^{-1}\nabla _{\theta }p(\theta
)d\lambda ,
\end{equation*}%
\begin{equation}
\nabla _{\theta }^{2}Q_{n}(\theta )=2\int_{0}^{1}p_{n,j_{n},r_{\ast
}}^{-1}\nabla _{\theta }p(\theta )\nabla _{\theta }p(\theta )^{\prime
}d\lambda -2\int_{0}^{1}\left( p_{n,j_{n},r_{\ast }}-p(\theta )\right)
p_{n,j_{n},r_{\ast }}^{-1}\nabla _{\theta }^{2}p(\theta )d\lambda ,
\label{Q_n_2nd_deriv}
\end{equation}%
and these derivatives are measurable functions for every $\theta \in
B(\theta _{0})$. Since $\hat{\theta}_{n}$ is an interior maximizer of $Q_{n}$
(on the event considered), we have that $\nabla _{\theta }Q_{n}(\hat{\theta}%
_{n})=0$. Consequently, a standard Taylor expansions gives%
\begin{equation}
0=\nabla _{\theta }Q_{n}(\hat{\theta}_{n})=\nabla _{\theta }Q_{n}(\theta
_{0})+\nabla _{\theta }^{2}Q_{n}^{\ast }(\hat{\theta}_{n}-\theta _{0}),
\label{classic_0}
\end{equation}%
where the $i$-th row of $\nabla _{\theta }^{2}Q_{n}^{\ast }$ equals the
corresponding row of $\nabla _{\theta }^{2}Q_{n}$ evaluated at a mean-value $%
\tilde{\theta}_{n}^{(i)}$ which may depend on the row-index (measurability
of $\tilde{\theta}_{n}^{(i)}$ being no concern here). We now first establish
that $n^{1/2}\nabla _{\theta }Q_{n}(\theta _{0})$ is asymptotically normal
with mean zero and variance-covariance matrix $4\int_{0}^{1}\nabla _{\theta
}p(\theta _{0})\nabla _{\theta }p(\theta _{0})^{\prime }p^{-1}(\theta
_{0})d\lambda $. To this end write $(-1/2)n^{1/2}\nabla _{\theta
}Q_{n}(\theta _{0})$ as%
\begin{eqnarray*}
&&\sqrt{n}\int_{0}^{1}\left( p_{n,j_{n},r_{\ast }}-p(\theta _{0})\right)
p(\theta _{0})^{-1}\nabla _{\theta }p(\theta _{0})d\lambda + \\
&&\sqrt{n}\int_{0}^{1}\left( p_{n,j_{n},r_{\ast }}-p(\theta _{0})\right)
(p_{n,j_{n},r_{\ast }}^{-1}-p(\theta _{0})^{-1})\nabla _{\theta }p(\theta
_{0})d\lambda ,
\end{eqnarray*}%
both terms being measurable. The first term in the above display now
converges to the required limit by Theorem \ref{UCLT_2} (applied with $%
t=\tau $, and some $s$ satisfying $1/2<s<1$, $s\leq \varsigma \wedge \tau $)
and the Cram\'{e}r-Wold device: To see this, observe that $p_{0}\in \mathcal{%
B}_{t}$ by Assumption P1(i),(ii) (since $p_{0}=p(\theta _{0})$ $\lambda $%
-a.e.). Furthermore, for every $\alpha \in \mathbb{R}^{b}$, $\alpha \neq 0$,
the function $f=p(\theta _{0})^{-1}\alpha ^{\prime }\nabla _{\theta
}p(\theta _{0})$ belongs to $\mathsf{B}_{\varsigma \wedge \tau }$ as a
consequence of Assumption P1(ii),(iv) and Proposition \ref{elem} in Appendix %
\ref{App_Spline}. Hence $\mathcal{F=}\left\{ f\right\} \subseteq \mathsf{B}%
_{s}$. The conditions on $j_{n}$ in Theorem \ref{UCLT_2} follow from the
assumption on $j_{n}$ in the current theorem. Finally note that $P(f)=0$
under Assumption P1. The second term in the above display is bounded in norm
(on the event $A_{n}^{\ast }$) by 
\begin{eqnarray*}
&&n^{1/2}\int_{0}^{1}\left( p_{n,j_{n},r_{\ast }}-p(\theta _{0})\right)
^{2}p(\theta _{0})^{-1}p_{n,j_{n},r_{\ast }}^{-1}\left\Vert \nabla _{\theta
}p(\theta _{0})\right\Vert d\lambda \\
&\leq &(2/\xi _{0}^{2})\sup_{x\in \lbrack 0,1]}\left\Vert \nabla _{\theta
}p(\theta _{0},x)\right\Vert n^{1/2}\left\Vert p_{n,j_{n},r_{\ast
}}-p(\theta _{0})\right\Vert _{2}^{2},
\end{eqnarray*}%
noting that $\left\Vert p(\theta _{0})^{-1}\right\Vert _{\infty }\leq \xi
_{0}^{-1}$, and that $\frac{\partial }{\partial \theta _{q}}p(\theta _{0})$
is bounded on $[0,1]$ for every $q$ since it belongs to $\mathsf{B}%
_{\varsigma }$ with $\varsigma >1/2$ by Assumption P1(iv). By Lemma \ref%
{havoc0} the r.h.s in the above display is $%
O_{p}(n^{-1/2}2^{j_{n}}+n^{1/2}2^{-2j_{n}\tau })$ which is $o_{p}(1)$
because of $\tau >1/2$.

Next we show that $\nabla _{\theta }^{2}Q_{n}^{\ast }$ converges to the
positive definite matrix $\nabla _{\theta }^{2}Q(\theta _{0})$ in (outer)
probability. To this end we first show that $\nabla _{\theta
}^{2}Q_{n}(\theta )$ converges to $\nabla _{\theta }^{2}Q(\theta )$
uniformly over $B(\theta _{0})$ in probability where $Q(\theta )$ has been
defined in (\ref{Q}). By Assumption P1 and dominated convergence we have
that $Q(\theta )$ is twice continuously differentiable on $B(\theta _{0})$
with%
\begin{equation*}
\nabla _{\theta }^{2}Q(\theta )=2\int_{0}^{1}p(\theta _{0})^{-1}\nabla
_{\theta }p(\theta )\nabla _{\theta }p(\theta )^{\prime }d\lambda
-2\int_{0}^{1}\left( p(\theta _{0})-p(\theta )\right) p(\theta
_{0})^{-1}\nabla _{\theta }^{2}p(\theta )d\lambda .
\end{equation*}%
We now see that 
\begin{eqnarray*}
&&\nabla _{\theta }^{2}Q_{n}(\theta )-\nabla _{\theta }^{2}Q(\theta ) \\
&=&2\int_{0}^{1}(p_{n,j_{n},r_{\ast }}^{-1}-p(\theta _{0})^{-1})\nabla
_{\theta }p(\theta )\nabla _{\theta }p(\theta )^{\prime
}-2\int_{0}^{1}\left( p_{n,j_{n},r_{\ast }}-p(\theta )\right)
(p_{n,j_{n},r_{\ast }}^{-1}-p(\theta _{0})^{-1})\nabla _{\theta
}^{2}p(\theta ) \\
&&+2\int_{0}^{1}\left( p(\theta _{0})-p_{n,j_{n},r_{\ast }}\right) p(\theta
_{0})^{-1}\nabla _{\theta }^{2}p(\theta )
\end{eqnarray*}%
and we obtain (the supremum being measurable because of continuity of $%
\nabla _{\theta }^{2}Q_{n}$ and $\nabla _{\theta }^{2}Q$ on $B(\theta _{0})$)%
\begin{eqnarray}
&&\sup_{\theta \in B(\theta _{0})}\left\Vert \nabla _{\theta
}^{2}Q_{n}(\theta )-\nabla _{\theta }^{2}Q(\theta )\right\Vert
\label{secondhalf} \\
&\leq &2\left\Vert p_{n,j_{n},r_{\ast }}-p(\theta _{0})\right\Vert _{\infty
}\sup_{\theta \in B(\theta _{0})}\left[ \int_{0}^{1}p_{n,j_{n},r_{\ast
}}^{-1}p(\theta _{0})^{-1}\left\Vert \nabla _{\theta }p(\theta )\right\Vert
^{2}d\lambda \right.  \notag \\
&&\left. +\int_{0}^{1}\left\vert p_{n,j_{n},r_{\ast }}-p(\theta )\right\vert
p_{n,j_{n},r_{\ast }}^{-1}p(\theta _{0})^{-1}\left\Vert \nabla _{\theta
}^{2}p(\theta )\right\Vert d\lambda +\int_{0}^{1}p(\theta
_{0})^{-1}\left\Vert \nabla _{\theta }^{2}p(\theta )\right\Vert d\lambda 
\right]  \notag \\
&\leq &\left\Vert p_{n,j_{n},r_{\ast }}-p(\theta _{0})\right\Vert _{\infty } 
\left[ 4\xi _{0}^{-2}\int_{0}^{1}\sup_{\theta \in B(\theta _{0})}\left\Vert
\nabla _{\theta }p(\theta )\right\Vert ^{2}d\lambda \right.  \notag \\
&&+\left. \left( 4\xi _{0}^{-2}\left( \left\Vert p_{n,j_{n},r_{\ast
}}\right\Vert _{\infty }+\sup_{\theta \in B(\theta _{0})}\left\Vert p(\theta
)\right\Vert _{\infty }\right) +2\xi _{0}^{-1}\right)
\int_{0}^{1}\sup_{\theta \in B(\theta _{0})}\left\Vert \nabla _{\theta
}^{2}p(\theta )\right\Vert d\lambda \right] =o_{p}(1),  \notag
\end{eqnarray}%
by Assumption P1 and Corollary \ref{consistency_2} (applied with a $t$
satisfying $1/2<t\leq \tau \wedge 1$). Since $\nabla _{\theta }^{2}Q(\theta
) $ is continuous at $\theta _{0}$ as shown above and since $\hat{\theta}%
_{n} $ is consistent, convergence of $\nabla _{\theta }^{2}Q_{n}^{\ast }$ to 
$\nabla _{\theta }^{2}Q(\theta _{0})$ in (outer) probability follows.

The central limit theorem for the score together with the convergence result
for $\nabla _{\theta }^{2}Q_{n}^{\ast }$ just established delivers now the
desired result: rewrite (\ref{classic_0}) as%
\begin{equation*}
0=n^{1/2}\nabla _{\theta }Q_{n}(\theta _{0})+\nabla _{\theta }^{2}Q(\theta
_{0})n^{1/2}(\hat{\theta}_{n}-\theta _{0})+\left( \nabla _{\theta
}^{2}Q_{n}^{\ast }-\nabla _{\theta }^{2}Q(\theta _{0})\right) n^{1/2}(\hat{%
\theta}_{n}-\theta _{0}),
\end{equation*}%
observe that $\nabla _{\theta }^{2}Q(\theta _{0})$ is positive definite by
Assumption P1(iii), and that the third term on the r.h.s. is of lower order
than the second one. This implies that $n^{1/2}(\hat{\theta}_{n}-\theta
_{0}) $ is stochastically bounded, and the desired result then easily
follows.
\end{proof}

\bigskip

For the same reasons as given in Remark \ref{rem_main}, the condition $\tau
<r_{\ast }$ in the above theorem is not really a restriction. Furthermore,
examining the proof shows that the conclusions of the theorem also hold for
other choices of $2^{j_{n}}$: e.g., the theorem (without the condition $\tau
<r_{\ast }$) holds for $2^{j_{n}}\sim n^{\nu }$ with $\nu $ satisfying $%
1/\left( 2\left( (\tau \wedge r_{\ast })+(\varsigma \wedge \tau \wedge
1)\right) \right) <\nu <1/2$.

\subsection{Proof of Part a1 of Theorem \protect\ref{main}\label%
{proof_parta1}}

We first provide an auxiliary result that relates the objective function $%
\mathcal{Q}_{n,k}(\theta )$ to the somewhat simpler objective function $%
Q_{n}(\theta )$ studied in the preceding section. Note that $k$ is not
linked to $n$ in the subsequent proposition.

\begin{proposition}
\label{wurzel0}Suppose $r\geq 2$ and $r_{\ast }\geq 2$ hold and Assumptions
P1(i),(ii) are satisfied for some $1/2<\tau <r_{\ast }\wedge r$. Suppose
further that Assumption R(i) is satisfied and that $2^{j_{n}}\sim
n^{1/(2\tau +1)}$ and $2^{J_{k}}\sim k^{1/(2\tau +1)}$. Then for every $%
\varepsilon >0$ there exists a positive real number $M(\varepsilon )$ and a
natural number $N(\varepsilon )$ such that 
\begin{equation}
\Pr \left( k^{1/2}\sup_{\theta \in \Theta }|\mathcal{Q}_{n,k}(\theta
)-Q_{n}(\theta )|>M(\varepsilon )\right) <\varepsilon  \label{approx}
\end{equation}%
holds for all $n\geq N(\varepsilon )$ and all $k\geq 1$.
\end{proposition}

\begin{proof}
First note that the supremum in (\ref{approx}) is measurable since $\mathcal{%
Q}_{n,k}(\theta )$ and $Q_{n}(\theta )$ are continuous in $\theta $ as noted
before, cf. Section \ref{proof_consist}. For given $\varepsilon >0$ choose $%
N(\varepsilon )$ large enough such that for $n\geq N(\varepsilon )$ we have $%
\Pr \left( A_{n}^{\ast }\right) >1-\varepsilon $ where $A_{n}^{\ast
}=\left\{ \inf_{y\in \lbrack 0,1]}p_{n,j_{n},r_{\ast }}(y)\geq \xi
_{0}/2\right\} $. This is possible by Corollary \ref{consistency_2}. A
simple calculation shows that on the event $A_{n}^{\ast }$%
\begin{equation*}
\mathcal{Q}_{n,k}(\theta )-Q_{n}(\theta )=\int_{0}^{1}(p_{k,J_{k},r}(\theta
)-p(\theta ))\left[ \frac{p_{k,J_{k},r}(\theta )+p(\theta )}{%
p_{n,j_{n},r_{\ast }}}-2\right]
\end{equation*}%
holds. Choose $s$ to satisfy $1/2<s<\tau \wedge 1$. Applying Corollaries \ref%
{consistency} and \ref{consistency_2} (with $t=s$) shows that for the given $%
\varepsilon >0$ there exists a positive finite $D$ such that the events 
\begin{equation*}
A_{n,k}^{\ast \ast }=\left\{ \sup_{\theta \in \Theta }\left\Vert
p_{k,J_{k},r}(\theta )\right\Vert _{s,2}\leq D,\left\Vert p_{n,j_{n},r_{\ast
}}\right\Vert _{s,2}\leq D\right\}
\end{equation*}%
have probability not less than $1-\varepsilon $ for every $k\geq 1$ and $%
n\geq 1$. Applying Proposition \ref{elem} in Appendix \ref{App_Spline}, we
conclude that there exists a finite positive $D^{\prime }$, depending only
on $D$, $\xi _{0}$, and $\sup_{\theta \in \Theta }\left\Vert p(\theta
)\right\Vert _{s,2}$ (which is finite by Assumption P1(ii) and continuous
embedding of $\mathsf{B}_{\tau }$ in $\mathsf{B}_{s}$), such that on $%
A_{n}^{\ast }\cap A_{n,k}^{\ast \ast }$%
\begin{equation*}
\sup_{\theta \in \Theta }\left\Vert (p_{k,J_{k},r}(\theta )+p(\theta
))p_{n,j_{n},r_{\ast }}^{-1}-2\right\Vert _{s,2}\leq D^{\prime }
\end{equation*}%
holds. Thus for every $M>0$, all $k\geq 1$, and all $n\geq N(\varepsilon )$%
\begin{eqnarray*}
&&\Pr \left( \sqrt{k}\sup_{\theta \in \Theta }|\mathcal{Q}_{n,k}(\theta
)-Q_{n}(\theta )|>M\right) \\
&\leq &\Pr \left( \left\{ \sqrt{k}\sup_{\theta \in \Theta }\sup_{\Vert
f\Vert _{s,2}\leq D^{\prime }}\left\vert \int_{0}^{1}(p_{k,J_{k},r}(\theta
)-p(\theta ))fd\lambda \right\vert >M\right\} \cap A_{n}^{\ast }\cap
A_{n,k}^{\ast \ast }\right) +2\varepsilon \\
&\leq &\Pr \left( \left\{ \sqrt{k}\sup_{\theta \in \Theta }\left\Vert
P_{k,J_{k},r}(\theta )-P(\theta )\right\Vert _{\mathcal{F}}>M\right\}
\right) +2\varepsilon
\end{eqnarray*}%
where $\mathcal{F}$ denotes $\left\{ f\in \mathsf{B}_{s}:\Vert f\Vert
_{s,2}\leq D^{\prime }\right\} $ and $\left\Vert \cdot \right\Vert _{%
\mathcal{F}}$ is defined before Theorem \ref{UCLT}. Choose an $s^{\prime }$
satisfying $1/2<s^{\prime }<s$. Then Theorem \ref{UCLT} (applied with $%
t=\tau $) implies for every $k\geq 1$%
\begin{eqnarray*}
\sqrt{k}\sup_{\theta \in \Theta }\left\Vert P_{k,J_{k},r}(\theta )-P(\theta
)\right\Vert _{\mathcal{F}} &\leq &\sqrt{k}\sup_{\theta \in \Theta
}\left\Vert P_{k,J_{k},r}(\theta )-P_{k}(\theta )\right\Vert _{\mathcal{F}}+%
\sqrt{k}\sup_{\theta \in \Theta }\left\Vert P_{k}(\theta )-P(\theta
)\right\Vert _{\mathcal{F}} \\
&=&O_{p}\left( \sqrt{k}2^{-J_{k}(\tau +s)}+2^{-J_{k}(s-s^{\prime
})}+1\right) =O_{p}(1).
\end{eqnarray*}%
[Measurability of the suprema on the r.h.s. in the first line of the above
display is established in the proof of Theorem \ref{UCLT}. The argument
given there also establishes measurability of the supremum on the l.h.s.]
This completes the proof (noting that the l.h.s. in the above display is
certainly a \emph{real-valued} random variable for every $k$).
\end{proof}

\bigskip

The closeness of $\mathcal{Q}_{n,k}$ and $Q_{n}$ expressed in the previous
result translates into closeness of the minimizers of these functions with
the help of the following simple but useful lemma which is taken from Gach
(2010). Note that $M_{2}$ below is smooth but $M_{1}$ need not be so. This
is relevant as $\mathcal{Q}_{n,k}$ is not guaranteed to be smooth under the
assumptions of Part a1 of Theorem \ref{main}, whereas $Q_{n}$ is in view of
Assumption P1.

\begin{lemma}
\label{gachwurzel} Let $U$ be a nonempty convex open subset of $\mathbb{R}%
^{b}$. Suppose we are given functions $M_{1}:U\rightarrow \mathbb{R}$ and $%
M_{2}:U\rightarrow \mathbb{R}$, such that $M_{2}$ is twice partially
differentiable on $U$ with Hessian satisfying 
\begin{equation}
\inf_{x\in U}y^{\prime }\nabla _{x}^{2}M_{2}(x)y\geq c\left\Vert
y\right\Vert ^{2}  \label{pos_def}
\end{equation}%
for every $y\in \mathbb{R}^{b}$ and some $0<c<\infty $. If $m_{1}\in U$ and $%
m_{2}\in U$ minimize $M_{1}$ and $M_{2}$ over $U$, respectively, we have 
\begin{equation*}
\left\Vert m_{1}-m_{2}\right\Vert \leq 2c^{-1/2}\sqrt{\sup_{u\in
U}\left\vert M_{1}(u)-M_{2}(u)\right\vert }
\end{equation*}%
where $\left\Vert \cdot \right\Vert $ denotes the Euclidean norm on $\mathbb{%
R}^{b}$.
\end{lemma}

\begin{proof}
Assume that minimizers $m_{1}$ and $m_{2}$ exist, since otherwise there is
nothing to prove. [By convexity of $U$ and the assumption on the Hessian the
minimizer $m_{2}$ is unique.] Since $m_{2}$ is a minimizer of the twice
partially differentiable function $M_{2}$ on the convex open set $U$, we
have 
\begin{equation*}
M_{2}(m_{1})=M_{2}(m_{2})+2^{-1}(m_{1}-m_{2})^{\prime }\nabla _{x}^{2}M_{2}(%
\tilde{m})(m_{1}-m_{2})
\end{equation*}%
(using a pathwise Taylor series expansion) where $\tilde{m}$ lies in the
convex hull of $\{m_{1},m_{2}\}$. We conclude from the assumption on the
Hessian that 
\begin{equation}
\left\Vert m_{1}-m_{2}\right\Vert \leq (2c^{-1})^{1/2}\sqrt{\left\vert
M_{2}(m_{1})-M_{2}(m_{2})\right\vert }.  \label{quadratic}
\end{equation}%
Observe next that 
\begin{equation*}
M_{1}(m_{1})-M_{2}(m_{2})\leq M_{1}(m_{2})-M_{2}(m_{2})\leq \sup_{u\in
U}|M_{1}(u)-M_{2}(u)|
\end{equation*}%
and 
\begin{equation*}
M_{1}(m_{1})-M_{2}(m_{2})\geq M_{1}(m_{1})-M_{2}(m_{1})\geq -\sup_{u\in
U}|M_{1}(u)-M_{2}(u)|
\end{equation*}%
so that 
\begin{equation*}
\left\vert M_{1}(m_{1})-M_{2}(m_{2})\right\vert \leq \sup_{u\in
U}|M_{1}(u)-M_{2}(u)|.
\end{equation*}%
Consequently, 
\begin{equation*}
|M_{2}(m_{1})-M_{2}(m_{2})|\leq
|M_{2}(m_{1})-M_{1}(m_{1})|+|M_{1}(m_{1})-M_{2}(m_{2})|\leq 2\sup_{u\in
U}|M_{1}(u)-M_{2}(u)|,
\end{equation*}%
which, when plugged into (\ref{quadratic}), proves the lemma.
\end{proof}

\bigskip

The proof of Part a1 of Theorem \ref{main} is now as follows: Let $%
U\subseteq B(\theta _{0})$ be a sufficiently small open ball around $\theta
_{0}$ such that the smallest eigenvalues of $\nabla _{\theta }^{2}Q(\theta )$
are bounded away from zero by a positive constant, $\eta $ say, uniformly in 
$\theta \in U$. Such an $U$ exists, since $\nabla _{\theta }^{2}Q(\theta )$
is continuous on $B(\theta _{0})$, as shown in Section \ref{intermed},\ and
since $\nabla _{\theta }^{2}Q(\theta _{0})$ is positive definite by
Assumption P1. Now apply Lemma \ref{gachwurzel} with $M_{1}=\mathcal{Q}%
_{n,k(n)}$, $M_{2}=Q_{n}$, and the set $U$ just mentioned. Note that
condition (\ref{pos_def}) is then satisfied for $M_{2}=Q_{n}$\ and $c=\eta
/2 $ on an event $E_{n}$ that has probability converging to $1$ in view of
the choice of $U$ and since it was shown in the proof of Theorem \ref{0815}
that $\nabla _{\theta }^{2}Q_{n}(\theta )$ converges to $\nabla _{\theta
}^{2}Q(\theta )$ uniformly on $B(\theta _{0})$ in probability. Observe also
that Proposition \ref{wurzel0} implies%
\begin{equation*}
\sup_{\theta \in \Theta }|\mathcal{Q}_{n,k(n)}(\theta )-Q_{n}(\theta
)|=O_{p}(k(n)^{-1/2}).
\end{equation*}%
Taken together, this implies 
\begin{equation}
\left\Vert \hat{\theta}_{n,k(n)}-\hat{\theta}_{n}\right\Vert
=O_{p}(k(n)^{-1/4}),  \label{nearness}
\end{equation}%
which is $o_{p}(n^{-1/2})$ in view of Assumption S1. Part a1 of Theorem \ref%
{main} now follows from asymptotic normality of $\sqrt{n}\left( \hat{\theta}%
_{n}-\theta _{0}\right) $ which has already been established in Theorem \ref%
{0815}.

\subsection{Proof of the Remaining Parts of Theorem \protect\ref{main}}

Observe first that it suffices to show that every subsequence $n_{i}$ of $n$
contains a further subsequence $n_{i(l)}$ along which the claimed asymptotic
normality result holds. Given $n_{i}$, we may choose the subsequence $%
n_{i(l)}$ in such a way that $\lim_{l\rightarrow \infty
}k(n_{i(l)})/n_{i(l)}^{2}$ exists (possibly being $\infty $) since the
extended real line is compact. But the sequence $k(n_{i(l)})$ can be viewed
as the subsequence $\bar{k}(n_{i(l)})$ of a sequence $\bar{k}(n)$ for which $%
\lim_{n\rightarrow \infty }\bar{k}(n)/n^{2}$ exists (and necessarily equals $%
\lim_{l\rightarrow \infty }k(n_{i(l)})/n_{i(l)}^{2}$). This shows that for
the proof we may assume without loss of generality that $\lim_{n\rightarrow
\infty }k(n)/n^{2}$ exists (possibly being $\infty $). In the case where
this limit is infinite, the results then follow from Part a1 which has
already been proved in Section \ref{proof_parta1}. Thus we may assume
without loss of generality not only that the limit of $k(n)/n^{2}$ exists,
but also that%
\begin{equation}
\lim_{n\rightarrow \infty }k(n)/n^{2}<\infty .  \label{wlog}
\end{equation}%
We shall make this assumption for the remainder of this section.

Under Assumption R and if $r\geq 4$ the mapping 
\begin{equation*}
\theta \mapsto p_{k,J,r}(\theta
,y)=\sum_{l=-r+1}^{2^{J}-1}\sum_{m=-r+1}^{2^{J}-1}2^{J}g_{J}^{(r)lm}\left(
k^{-1}\sum_{i=1}^{k}N_{mJ}^{(r)}(\rho (V_{i},\theta ))\right) N_{lJ}^{(r)}(y)
\end{equation*}%
is twice continuously differentiable on $B(\theta _{0})$ for every $y$ and
every realization of $V_{1},\ldots ,V_{k}$ by the chain rule. Similarly as
in the proof of Theorem \ref{0815}, it suffices to work only on the event $%
A_{n}^{\ast }\cap \left\{ \hat{\theta}_{n,k(n)}\in B(\theta _{0})\right\} $
which has probability converging to $1$ in view of Proposition \ref{consist}
(applied with $\delta >1/2$ sufficiently close to $1/2$) and Corollary \ref%
{consistency_2} (applied with some $t$ satisfying $1/2<t\leq \tau \wedge 1$%
). Note that $\left\Vert p(\theta _{0})^{-1}\right\Vert _{\infty }\leq \xi
_{0}$, and that $\left\Vert p_{n,j_{n},r_{\ast }}^{-1}\right\Vert _{\infty
}\leq 2/\xi _{0}$ holds on the before mentioned event; we shall use these
facts repeatedly in the sequel. Using this, (\ref{sup-norm_estim}),
boundedness of $N_{mJ}^{(r)}$ and of its first two derivatives as well as
Assumption R, one concludes from the dominated convergence theorem that also
the objective function $\mathcal{Q}_{n,k}$ defined in (\ref{objective}) is
twice continuously differentiable on the neighborhood $B(\theta _{0})$ with
derivatives (measurable for every $\theta \in B(\theta _{0})$)%
\begin{equation*}
\nabla _{\theta }\mathcal{Q}_{n,k}(\theta
)=-2\int_{0}^{1}(p_{n,j_{n},r_{\ast }}-p_{k,J_{k},r}(\theta
))p_{n,j_{n},r_{\ast }}^{-1}\nabla _{\theta }p_{k,J_{k},r}(\theta )d\lambda ,
\end{equation*}%
\begin{eqnarray}
\nabla _{\theta }^{2}\mathcal{Q}_{n,k}(\theta )
&=&2\int_{0}^{1}p_{n,j_{n},r_{\ast }}^{-1}\nabla _{\theta
}p_{k,J_{k},r}(\theta )\nabla _{\theta }p_{k,J_{k},r}(\theta )d\lambda 
\notag \\
&&-2\int_{0}^{1}(p_{n,j_{n},r_{\ast }}-p_{k,J_{k},r}(\theta
))p_{n,j_{n},r_{\ast }}^{-1}\nabla _{\theta }^{2}p_{k,J_{k},r}(\theta
)d\lambda .  \label{Q_nk_2nd_deriv}
\end{eqnarray}%
Since $\hat{\theta}_{n,k(n)}$ is an interior maximizer of $\mathcal{Q}%
_{n,k(n)}$ (on the event considered), we clearly have that $\nabla _{\theta }%
\mathcal{Q}_{n,k(n)}(\hat{\theta}_{n,k(n)})=0$. Consequently, a standard
Taylor expansions gives%
\begin{equation}
0=\nabla _{\theta }\mathcal{Q}_{n,k(n)}(\hat{\theta}_{n,k(n)})=\nabla
_{\theta }\mathcal{Q}_{n,k(n)}(\theta _{0})+\nabla _{\theta }^{2}\mathcal{Q}%
_{n,k(n)}^{\ast }(\hat{\theta}_{n,k(n)}-\theta _{0}),  \label{classic}
\end{equation}%
where the $i$-th row of $\nabla _{\theta }^{2}\mathcal{Q}_{n,k(n)}^{\ast }$
equals the corresponding row of $\nabla _{\theta }^{2}\mathcal{Q}_{n,k(n)}$
evaluated at a mean-value $\tilde{\theta}_{n,k(n)}^{(i)}$ which may depend
on the row-index (measurability of the mean-value being of no concern). We
next show that $\sqrt{n}\nabla _{\theta }\mathcal{Q}_{n,k(n)}(\theta _{0})$
is asymptotically normal and that $\nabla _{\theta }^{2}\mathcal{Q}%
_{n,k(n)}^{\ast }$ converges in (outer) probability to the positive definite
matrix $\nabla _{\theta }^{2}Q(\theta _{0})$. The asymptotic normality of $%
\sqrt{n}\left( \hat{\theta}_{n,k(n)}-\theta _{0}\right) $ then follows along
the same lines as in the last paragraph of the proof of Theorem \ref{0815}.

\textbf{Step 1:} CLT for the score $\sqrt{n}\nabla _{\theta }\mathcal{Q}%
_{n,k(n)}(\theta _{0})$.

We decompose the score as follows: 
\begin{eqnarray*}
&&\nabla _{\theta }\mathcal{Q}_{n,k(n)}(\theta _{0}) \\
&=&-2\int_{0}^{1}(p_{n,j_{n},r_{\ast }}-p(\theta _{0}))p(\theta
_{0})^{-1}\nabla _{\theta }p(\theta _{0})d\lambda \\
&&+2\int_{0}^{1}(p_{k(n),J_{k(n)},r}(\theta _{0})-p(\theta _{0}))p(\theta
_{0})^{-1}\nabla _{\theta }p(\theta _{0})d\lambda \\
&&+2\int_{0}^{1}(p_{n,j_{n},r_{\ast }}-p_{k(n),J_{k(n)},r}(\theta
_{0}))\left( p(\theta _{0})^{-1}\nabla _{\theta }p(\theta
_{0})-p_{n,j_{n},r_{\ast }}^{-1}\nabla _{\theta }p_{k(n),J_{k(n)},r}(\theta
_{0})\right) d\lambda \\
&=&I+II+III,
\end{eqnarray*}%
with each of the terms being measurable. We further observe that the terms $%
I $ and $II$ are independent by construction of the simulation mechanism.

\textit{About Term I:} As shown in the proof of Theorem \ref{0815}%
\begin{equation*}
\sqrt{n}I\rightarrow ^{d}N(0,\Sigma )
\end{equation*}%
where 
\begin{equation*}
\Sigma =4\int_{0}^{1}\nabla _{\theta }p(\theta _{0})\nabla _{\theta
}p(\theta _{0})^{\prime }p(\theta _{0})^{-1}d\lambda .
\end{equation*}

\textit{About Term II:} Exactly the same argument as given in the proof of
Theorem \ref{0815} for term $I$, except for using Theorem \ref{UCLT} instead
of Theorem \ref{UCLT_2}, establishes that 
\begin{equation*}
\sqrt{k(n)}II\rightarrow ^{d}N\left( 0,\Sigma \right) .
\end{equation*}%
But then%
\begin{equation*}
\sqrt{n}II=\sqrt{n/k(n)}\sqrt{k(n)}II\rightarrow ^{d}N\left( 0,\frac{1}{%
\kappa }\Sigma \right)
\end{equation*}%
under Assumption S3, and $\sqrt{n}II$ converges to zero in probability under
Assumption S2.

\textit{About Term III:} By Cauchy-Schwarz and the triangle inequality we
have the bound%
\begin{eqnarray*}
\left\Vert III\right\Vert &\leq &2\left\Vert p_{n,j_{n},r_{\ast
}}-p_{k(n),J_{k(n)},r}(\theta _{0})\right\Vert _{2}\left[ \left\Vert \left(
p(\theta _{0})^{-1}-p_{n,j_{n},r_{\ast }}^{-1}\right) \nabla _{\theta
}p(\theta _{0})\right\Vert _{2}\right. \\
&&\left. +\left\Vert p_{n,j_{n},r_{\ast }}^{-1}\left( \nabla _{\theta
}p_{k(n),J_{k(n)},r}(\theta _{0})-\nabla _{\theta }p(\theta _{0})\right)
\right\Vert _{2}\right] \\
&\leq &2\left\Vert p_{n,j_{n},r_{\ast }}-p_{k(n),J_{k(n)},r}(\theta
_{0})\right\Vert _{2}\left[ (2/\xi _{0}^{2})\left\Vert \left(
p_{n,j_{n},r_{\ast }}-p(\theta _{0})\right) \nabla _{\theta }p(\theta
_{0})\right\Vert _{2}\right. \\
&&\left. +(2/\xi _{0})\left\Vert \nabla _{\theta }p_{k(n),J_{k(n)},r}(\theta
_{0})-\nabla _{\theta }p(\theta _{0})\right\Vert _{2}\right] \\
&\leq &(4/\xi _{0})\left[ \left\Vert p_{n,j_{n},r_{\ast }}-p(\theta
_{0})\right\Vert _{2}+\left\Vert p(\theta _{0})-p_{k(n),J_{k(n)},r}(\theta
_{0})\right\Vert _{2}\right] \times \\
&&\left[ (1/\xi _{0})\left\Vert p_{n,j_{n},r_{\ast }}-p(\theta
_{0})\right\Vert _{2}\left\Vert \nabla _{\theta }p(\theta _{0})\right\Vert
_{\infty }+\left\Vert \nabla _{\theta }p_{k(n),J_{k(n)},r}(\theta
_{0})-\nabla _{\theta }p(\theta _{0})\right\Vert _{2}\right]
\end{eqnarray*}%
with $\left\Vert \nabla _{\theta }p(\theta _{0})\right\Vert _{\infty }$
being finite in view of Assumption P1(iv) and Proposition \ref{elem} in
Appendix \ref{App_Spline}. The r.h.s. of the above display is now%
\begin{equation*}
O_{p}\left( \left( \sqrt{\frac{2^{j_{n}}}{n}}+2^{-j_{n}\tau }+\sqrt{\frac{%
2^{J_{k(n)}}}{k(n)}}+2^{-J_{k(n)}\tau }\right) \left( \sqrt{\frac{2^{j_{n}}}{%
n}}+2^{-j_{n}\tau }+\sqrt{\frac{2^{3J_{k(n)}}}{k(n)}}+2^{-J_{k(n)}s}\right)
\right)
\end{equation*}%
for every $0<s<r$, $s\leq \varsigma $ in view of Assumptions P1 and R as
well as Lemmata \ref{havoc0} and \ref{havoc}. Fixing such an $s>1/2$, the
expression in the above display is seen to be $o_{p}(n^{-1/2})$ under the
assumptions of Part a2 or Part b (in particular, $\tau >3/2$), showing that $%
\sqrt{n}III$ is asymptotically negligible.

This completes Step 1 and shows that%
\begin{equation*}
\sqrt{n}\nabla _{\theta }\mathcal{Q}_{n,k(n)}(\theta _{0})\rightarrow
^{d}N\left( 0,(1+\kappa ^{-1})\Sigma \right)
\end{equation*}%
under the assumptions of Part b, whereas under the assumptions of Part a2 
\begin{equation*}
\sqrt{n}\nabla _{\theta }\mathcal{Q}_{n,k(n)}(\theta _{0})\rightarrow
^{d}N\left( 0,\Sigma \right) .
\end{equation*}

\bigskip

\textbf{Step 2:} Convergence of second order derivatives.

We have 
\begin{equation*}
\left\Vert \nabla _{\theta }^{2}\mathcal{Q}_{n,k(n)}^{\ast }-\nabla _{\theta
}^{2}Q(\theta _{0})\right\Vert \leq \left\Vert \nabla _{\theta }^{2}\mathcal{%
Q}_{n,k(n)}^{\ast }-\nabla _{\theta }^{2}Q_{n}^{\dag }\right\Vert
+\left\Vert \nabla _{\theta }^{2}Q_{n}^{\dag }-\nabla _{\theta }^{2}Q(\theta
_{0})\right\Vert
\end{equation*}%
where $\nabla _{\theta }^{2}Q_{n}^{\dag }$ is the matrix $\nabla _{\theta
}^{2}Q_{n}$ row-wise evaluated at the mean-values $\tilde{\theta}%
_{n,k(n)}^{(i)}$. In view of (\ref{secondhalf}), consistency of $\hat{\theta}%
_{n,k(n)}$, and continuity of $\nabla _{\theta }^{2}Q$ at $\theta _{0}$, the
second term on the r.h.s. above converges to zero in (outer) probability. We
now show the same for the first term on the r.h.s. in the above display:
Note that the argument leading to (\ref{nearness}) is also valid under the
current assumptions, and therefore we can conclude from (\ref{nearness}), (%
\ref{wlog}), and Theorem \ref{0815} that $\left\Vert \hat{\theta}%
_{n,k(n)}-\theta _{0}\right\Vert =O_{p}(k(n)^{-1/4})$. Consequently, it
suffices to show that 
\begin{equation*}
\sup_{\theta \in B(\theta _{0}),\left\Vert \theta -\theta _{0}\right\Vert
\leq Mk(n)^{-1/4}}\left\Vert \nabla _{\theta }^{2}\mathcal{Q}%
_{n,k(n)}(\theta )-\nabla _{\theta }^{2}Q_{n}(\theta )\right\Vert
\rightarrow 0
\end{equation*}%
in probability for every $0<M<\infty $, the above supremum being measurable
(as the functions involved are continuous). Now, by (\ref{Q_nk_2nd_deriv})
and (\ref{Q_n_2nd_deriv})%
\begin{eqnarray*}
&&\frac{1}{2}\left( \nabla _{\theta }^{2}\mathcal{Q}_{n,k(n)}(\theta
)-\nabla _{\theta }^{2}Q_{n}(\theta )\right)
=\int_{0}^{1}(p_{n,j_{n},r_{\ast }}-p(\theta ))p_{n,j_{n},r_{\ast
}}^{-1}\left( \nabla _{\theta }^{2}p(\theta )-\nabla _{\theta
}^{2}p_{k(n),J_{k(n)},r}(\theta )\right) d\lambda \\
&&-\int_{0}^{1}(p(\theta )-p_{k(n),J_{k(n)},r}(\theta ))p_{n,j_{n},r_{\ast
}}^{-1}\nabla _{\theta }^{2}p_{k(n),J_{k(n)},r}(\theta )d\lambda \\
&&+\int_{0}^{1}p_{n,j_{n},r_{\ast }}^{-1}\left( \nabla _{\theta
}p_{k(n),J_{k(n)},r}(\theta )\nabla _{\theta }p_{k(n),J_{k(n)},r}(\theta
)^{\prime }-\nabla _{\theta }p(\theta )\nabla _{\theta }p(\theta )^{\prime
}\right) d\lambda =I-II+III.
\end{eqnarray*}

\textit{About Term I: }By the Cauchy-Schwarz and the triangle inequalities%
\begin{eqnarray*}
\left\Vert I\right\Vert &\leq &2\xi _{0}^{-1}\left[ \left\Vert
p_{n,j_{n},r_{\ast }}-p(\theta _{0})\right\Vert _{2}+\left\Vert p(\theta
_{0})-p(\theta )\right\Vert _{2}\right] \times \\
&&\left[ \left\Vert \nabla _{\theta }^{2}p_{k(n),J_{k(n)},r}(\theta
)-E\nabla _{\theta }^{2}p_{k(n),J_{k(n)},r}(\theta )\right\Vert
_{2}+\left\Vert \nabla _{\theta }^{2}p(\theta )-E\nabla _{\theta
}^{2}p_{k(n),J_{k(n)},r}(\theta )\right\Vert _{2}\right] .
\end{eqnarray*}

The first term on the r.h.s. of the above display is $O_{p}(n^{-\tau /(2\tau
+1)})$ in view of Lemma \ref{havoc0} and the choice of $j_{n}$. For the
second term, observe that in view of Assumption P1(iii) we have $p(\theta
,x)-p(\theta _{0},x)=\nabla _{\theta }p(\breve{\theta}(x),x)^{\prime
}(\theta -\theta _{0})$ by the pathwise mean value theorem, and hence 
\begin{equation*}
\left\Vert p(\theta _{0})-p(\theta )\right\Vert _{2}\leq \left(
\int_{0}^{1}\sup_{\theta \in B(\theta _{0})}\left\Vert \nabla _{\theta
}p(\theta ,x)\right\Vert ^{2}dx\right) ^{1/2}\left\Vert \theta -\theta
_{0}\right\Vert =O(\left\Vert \theta -\theta _{0}\right\Vert )
\end{equation*}%
holds for all $\theta \in B(\theta _{0})$. In view of Lemma \ref{uniformrate}
and the choice of $J_{k(n)}$, the supremum over $B(\theta _{0})$ of the
third term is $O_{p}(k(n)^{(2-\tau )/(2\tau +1)}\sqrt{\log k(n)})$.
Furthermore, note that 
\begin{equation}
E\frac{\partial ^{2}p_{k(n),J_{k(n)},r}(\theta )}{\partial \theta
_{i}\partial \theta _{i^{\prime }}}=\pi _{J_{k(n)}}^{(r)}\left( \frac{%
\partial ^{2}p(\theta )}{\partial \theta _{i}\partial \theta _{i^{\prime }}}%
\right)  \label{proj_2nd_deriv}
\end{equation}%
holds for $\theta \in B(\theta _{0})$. [This is proved analogously as (\ref%
{interchange}) in Section \ref{rates}, making use of the dominance
assumptions on $\nabla _{\theta }^{2}p$ in Assumption P1, the uniform
boundedness assumption on the derivatives of $\rho $ in assumption R(ii),
the boundedness of the B-spline basis functions and their first two
derivatives (as $r\geq 4$ holds), as well as using that $\frac{\partial
^{2}p(\theta )}{\partial \theta _{i}\partial \theta _{i^{\prime }}}\in 
\mathcal{L}^{2}$ in view of Assumption P2(ii).] The above established
relation, together with the fact that the spectral matrix norm is bounded by
the Frobenius norm, implies that the supremum over $B(\theta _{0})$ of the
fourth term is bounded by 
\begin{equation*}
\sup_{\theta \in B(\theta _{0})}\sum_{i,i^{\prime }=1}^{b}\left\Vert \frac{%
\partial ^{2}p(\theta )}{\partial \theta _{i}\partial \theta _{i^{\prime }}}%
-\pi _{J_{k(n)}}^{(r)}\left( \frac{\partial ^{2}p(\theta )}{\partial \theta
_{i}\partial \theta _{i^{\prime }}}\right) \right\Vert _{2}\leq \sup_{\theta
\in B(\theta _{0})}\sum_{i,i^{\prime }=1}^{b}\left\Vert \frac{\partial
^{2}p(\theta )}{\partial \theta _{i}\partial \theta _{i^{\prime }}}%
\right\Vert _{2}<\infty
\end{equation*}%
the last inequality following from Assumption P2(ii). Consequently, in view
of (\ref{wlog}),%
\begin{eqnarray*}
&&\sup_{\theta \in B(\theta _{0}),\left\Vert \theta -\theta _{0}\right\Vert
\leq Mk(n)^{-1/4}}\left\Vert I\right\Vert \\
&\leq &\left[ O_{p}(n^{-\tau /(2\tau +1)})+O(k(n)^{-1/4})\right] \left[
O_{p}(k(n)^{(2-\tau )/(2\tau +1)}\sqrt{\log k(n)})+const\right] =o_{p}(1)
\end{eqnarray*}%
under either the assumptions of Part a2 or Part b (since $\tau >3/2>4/3$).

\textit{About Term II: }By the Cauchy-Schwarz and the triangle inequalities%
\begin{eqnarray}
&&\sup_{\theta \in B(\theta _{0})}\left\Vert II\right\Vert \leq 2\xi
_{0}^{-1}\sup_{\theta \in B(\theta _{0})}\left\Vert p(\theta
)-p_{k(n),J_{k(n)},r}(\theta )\right\Vert _{2}\times  \notag \\
&&\sup_{\theta \in B(\theta _{0})}\left[ \left\Vert \nabla _{\theta
}^{2}p_{k(n),J_{k(n)},r}(\theta )-E\nabla _{\theta
}^{2}p_{k(n),J_{k(n)},r}(\theta )\right\Vert _{2}+\left\Vert E\nabla
_{\theta }^{2}p_{k(n),J_{k(n)},r}(\theta )\right\Vert _{2}\right]  \notag \\
&=&O_{p}(k(n)^{-\tau /(2\tau +1)}\sqrt{\log k(n)})\left[ O_{p}(k(n)^{(2-\tau
)/(2\tau +1)}\sqrt{\log k(n)})+const\right]  \label{term2}
\end{eqnarray}%
where we have made use of Lemmata \ref{havoc0} and \ref{uniformrate}; and we
have used the bound%
\begin{equation*}
\sup_{\theta \in B(\theta _{0})}\left\Vert E\nabla _{\theta
}^{2}p_{k(n),J_{k(n)},r}(\theta )\right\Vert _{2}\leq \sup_{\theta \in
B(\theta _{0})}\sum_{i,i^{\prime }=1}^{b}\left\Vert \frac{\partial
^{2}p(\theta )}{\partial \theta _{i}\partial \theta _{i^{\prime }}}%
\right\Vert _{2}<\infty
\end{equation*}%
which follows from (\ref{proj_2nd_deriv}) and Assumption P2(ii). The r.h.s.
of (\ref{term2}) is now $o_{p}(1)$ since $\tau >3/2>1$.

\textit{About Term III: }By the Cauchy-Schwarz and the triangle inequalities%
\begin{equation*}
\left\Vert III\right\Vert \leq 2\xi _{0}^{-1}\left\Vert \nabla _{\theta
}p_{k(n),J_{k(n)},r}(\theta )-\nabla _{\theta }p(\theta )\right\Vert _{2} 
\left[ \left\Vert \nabla _{\theta }p_{k(n),J_{k(n)},r}(\theta )-\nabla
_{\theta }p(\theta )\right\Vert _{2}+2\left\Vert \nabla _{\theta }p(\theta
)\right\Vert _{2}\right] .
\end{equation*}%
Now%
\begin{equation*}
\sup_{\theta \in B(\theta _{0})}\left\Vert \nabla _{\theta
}p_{k(n),J_{k(n)},r}(\theta )-E\nabla _{\theta }p_{k(n),J_{k(n)},r}(\theta
)\right\Vert _{2}=O_{p}(k(n)^{(1-\tau )/(2\tau +1)}\sqrt{\log k(n)})=o_{p}(1)
\end{equation*}%
by Lemma \ref{uniformrate} and since $\tau >3/2>1$. Furthermore,%
\begin{eqnarray*}
\sup_{\theta \in B(\theta _{0})}\left\Vert E\nabla _{\theta
}p_{k(n),J_{k(n)},r}(\theta )-\nabla _{\theta }p(\theta )\right\Vert _{2}
&\leq &\sum_{i=1}^{b}\sup_{\theta \in B(\theta _{0})}\left\Vert E\frac{%
\partial p_{k(n),J_{k(n)},r}(\theta )}{\partial \theta _{i}}-\frac{\partial
p(\theta )}{\partial \theta _{i}}\right\Vert _{2} \\
&=&\sum_{i=1}^{b}\sup_{\theta \in B(\theta _{0})}\left\Vert \pi
_{J_{k(n)}}^{(r)}\left( \frac{\partial p(\theta )}{\partial \theta _{i}}%
\right) -\frac{\partial p(\theta )}{\partial \theta _{i}}\right\Vert _{2},
\end{eqnarray*}%
the last equality holding as shown in (\ref{interchange}) in Section \ref%
{rates}. By Proposition \ref{proj_error} in Appendix \ref{App_Spline} and
Assumption P2(i) the r.h.s. in the above display is now $o(1)$. Taken
together, this provides a bound for $\sup_{\theta \in B(\theta
_{0}),\left\Vert \theta -\theta _{0}\right\Vert \leq Mk(n)^{-1/4}}\left\Vert
III\right\Vert $ which converges to zero in probability. This completes the
proof of Step 2.

\subsection{Proof of Proposition \protect\ref{vc}\label{proof_vc}}

Since $\bar{\theta}_{n,k}\rightarrow \theta _{0}$ by assumption, since $\Phi
(\theta ):=\int_{0}^{1}\nabla _{\theta }p(\theta )\nabla _{\theta }p(\theta
)^{\prime }p(\theta _{0})^{-1}d\lambda $ is continuous on the neighborhood $%
B(\theta _{0})$ of $\theta _{0}$ by dominated convergence and Assumption
P1(iii), and since $\Phi (\theta _{0})$ is positive definite by the same
assumption, it suffices to show that, uniformly over $B(\theta _{0})$, the
expression $\hat{\Phi}(\theta )=\int_{0}^{1}\nabla _{\theta
}p_{k,J_{k}^{\prime },r^{\prime }}(\theta )\nabla _{\theta
}p_{k,J_{k}^{\prime },r^{\prime }}(\theta )^{\prime }p_{n,j_{n}^{\prime
},r_{\ast }^{\prime }}^{-1}d\lambda $ converges to $\Phi (\theta )$ in
probability as $n\wedge k\rightarrow \infty $. Note that $\hat{\Phi}(\theta
) $ is well-defined on the event $A_{n}^{\ast }$ which has probability
converging to $1$ in view of Corollary \ref{consistency_2}. In the sequel we
only work on that event. Now%
\begin{eqnarray*}
\left\vert \hat{\Phi}(\theta )-\Phi (\theta )\right\vert &\leq &\left\vert
\int_{0}^{1}\nabla _{\theta }p(\theta )\nabla _{\theta }p(\theta )^{\prime
}\left( p_{n,j_{n}^{\prime },r_{\ast }^{\prime }}^{-1}-p(\theta
_{0})^{-1}\right) d\lambda \right\vert \\
&&+\left\vert \int_{0}^{1}\left( \nabla _{\theta }p_{k,J_{k}^{\prime
},r^{\prime }}(\theta )\nabla _{\theta }p_{k,J_{k}^{\prime },r^{\prime
}}(\theta )^{\prime }-\nabla _{\theta }p(\theta )\nabla _{\theta }p(\theta
)^{\prime }\right) p_{n,j_{n}^{\prime },r_{\ast }^{\prime }}^{-1}d\lambda
\right\vert \\
&\leq &2\xi _{0}^{-2}\left\Vert p_{n,j_{n},r_{\ast }}-p(\theta
_{0})\right\Vert _{\infty }\int_{0}^{1}\sup_{\theta \in B(\theta
_{0})}\left\Vert \nabla _{\theta }p(\theta )\right\Vert ^{2}d\lambda \\
&&+2\xi _{0}^{-1}\left\Vert \nabla _{\theta }p_{k,J_{k},r^{\prime }}(\theta
)-\nabla _{\theta }p(\theta )\right\Vert _{2}\left[ \left\Vert \nabla
_{\theta }p_{k,J_{k},r^{\prime }}(\theta )-\nabla _{\theta }p(\theta
)\right\Vert _{2}+2\left\Vert \nabla _{\theta }p(\theta )\right\Vert _{2}%
\right] .
\end{eqnarray*}%
The first term on the r.h.s. is independent of $\theta $ and converges to
zero in probability by Corollary \ref{consistency_2}. The supremum over $%
B(\theta _{0})$ of the second term converges to zero by essentially
repeating the argument that has been used in the very last step of the proof
of Theorem \ref{main}.

\section{Rates of Convergence for Spline Projection Estimators\label{rates}}

This section contains the main stochastic bounds used to control remainder
terms in the proofs in Section \ref{proof_main}. We first collect some
simple facts about the B-splines $N^{(r)}$ that will repeatedly be used in
this section:%
\begin{equation}
\left\Vert N^{(r)}\right\Vert _{\infty ,\mathbb{R}}\leq 1,\quad \left\Vert
N^{(r)}\right\Vert _{1,\mathbb{R}}=1,\quad \left\Vert N^{(r)}\right\Vert _{2,%
\mathbb{R}}\leq 1\quad \text{for }r\geq 1.  \label{norm_ineq}
\end{equation}%
The first relation is a direct consequence of the definition of $N^{(r)}$,
the second one follows since $N(r)$ is -- as a convolution of probability
densities -- a probability density again, and the third relation is a
consequence of Young's inequality. Furthermore, it is easy to see that $%
N^{(r)}$ is continuously differentiable for $r\geq 3$ with derivative $%
N^{(r)\prime }$ given by%
\begin{equation}
N^{(r)\prime }=N^{(r-1)}-N^{(r-1)}(\cdot -1).  \label{deriv}
\end{equation}%
For $r=2$, the B-spline $N^{(2)}$ is Lipschitz and only has a weak
derivative $N^{(2)\prime }$ which, in order to have it defined everywhere,
will always be taken as $N^{(1)}-N^{(1)}(\cdot -1)$. The bounds%
\begin{equation}
\left\Vert N^{(r)\prime }\right\Vert _{\infty ,\mathbb{R}}\leq 1,\quad
\left\Vert N^{(r)\prime }\right\Vert _{1,\mathbb{R}}\leq 2,\quad \left\Vert
N^{(r)\prime }\right\Vert _{2,\mathbb{R}}\leq 2\quad \text{for }r\geq 2
\label{norm_ineq_2}
\end{equation}%
are then an immediate consequence of (\ref{norm_ineq}), (\ref{deriv}), and
the fact that $N^{(r-1)}$ is nonnegative. By repeated application of (\ref%
{deriv}) we can obtain bounds for higher-order derivatives, for example, we
shall need%
\begin{equation}
\left\Vert N^{(r)\prime \prime }\right\Vert _{\infty ,\mathbb{R}}\leq
2,\quad \left\Vert N^{(r)\prime \prime }\right\Vert _{2,\mathbb{R}}\leq
4\quad \text{for }r\geq 3,\text{ \ \ and }\left\Vert N^{(r)\prime \prime
\prime }\right\Vert _{\infty ,\mathbb{R}}\leq 4\text{ \ \ for }r\geq 4.
\label{norm_ineq_3}
\end{equation}%
The above discussion also implies that $N^{(r)}$ for $r\geq 2$, $%
N^{(r)\prime }$ for $r\geq 3$, and $N^{(r)\prime \prime }$ for $r\geq 4$ are
globally Lipschitz on $\mathbb{R}$ with Lipschitz constants bounded by $1$, $%
2$, and $4$, respectively.

For $f\in \mathcal{S}_{j}(r)$, $r\geq 3$, we denote in the following by $%
f^{\prime }$ its derivative (using one-sided derivatives on the boundary of $%
[0,1]$); for $r=2$ we use $f^{\prime }$ to denote the weak derivative.

\begin{lemma}
\label{basic}Let $f=\sum_{l=-r+1}^{2^{j}-1}\alpha _{l}N_{lj}^{(r)}$ where $%
\alpha _{l}$ are real numbers and $r\geq 1$, i.e., $f\in \mathcal{S}_{j}(r)$%
. Then 
\begin{equation}
\left\Vert f\right\Vert _{2}\leq 2^{-j/2}\left(
\sum_{l=-r+1}^{2^{j}-1}\alpha _{l}^{2}\right) ^{1/2},  \label{L_2}
\end{equation}%
\begin{equation}
\left\Vert f^{\prime }\right\Vert _{2}\leq 2^{1+j/2}\left(
\sum_{l=-r+1}^{2^{j}-1}\alpha _{l}^{2}\right) ^{1/2}\quad \ \ \quad \text{%
for }r\geq 2,  \label{L_2_deriv}
\end{equation}%
and 
\begin{equation}
\left\Vert f^{\prime \prime }\right\Vert _{2}\leq 2^{2+3j/2}\left(
\sum_{l=-r+1}^{2^{j}-1}\alpha _{l}^{2}\right) ^{1/2}\quad \ \ \quad \text{%
for }r\geq 3.  \label{L_2_deriv2}
\end{equation}%
Furthermore, for every $0<s^{\prime }\leq 1$ there exists a finite constant $%
C_{0}(s^{\prime })$ such that for every $r\geq 2$ and $f$ as above%
\begin{equation}
\left\Vert f\right\Vert _{s^{\prime },2}\leq C_{0}(s^{\prime
})2^{j(s^{\prime }-1/2)}\left( \sum_{l=-r+1}^{2^{j}-1}\alpha _{l}^{2}\right)
^{1/2}.  \label{bes_0}
\end{equation}
\end{lemma}

\begin{proof}
The first claim is well-known, see, e.g., DeVore and Lorentz (1993), Theorem
5.4.2. To prove (\ref{L_2_deriv}), use (\ref{deriv}) and the fact that $%
N^{(r-1)}$ vanishes outside of $(0,r-1)$ for $r\geq 3$ and outside of $[0,1)$
for $r=2$, to obtain (interpreting the equality modulo $\lambda $-nullsets
in case $r=2$)%
\begin{eqnarray*}
f^{\prime }(x) &=&2^{j}\sum_{l=-r+1}^{2^{j}-1}\alpha _{l}N^{(r)\prime
}(2^{j}x-l)=2^{j}\sum_{l=-r+1}^{2^{j}-1}\alpha _{l}\left[
N^{(r-1)}(2^{j}x-l)-N^{(r-1)}(2^{j}x-l-1)\right] \\
&=&2^{j}\sum_{l=-(r-1)+1}^{2^{j}-1}\alpha
_{l}N^{(r-1)}(2^{j}x-l)-2^{j}\sum_{l=-(r-1)+1}^{2^{j}-1}\alpha
_{l-1}N^{(r-1)}(2^{j}x-l)=:f_{1}+f_{2}.
\end{eqnarray*}%
Using (\ref{L_2}) for $f_{1}$ and $f_{2}$, we obtain%
\begin{equation*}
\left\Vert f^{\prime }\right\Vert _{2}\leq \left\Vert f_{1}\right\Vert
_{2}+\left\Vert f_{2}\right\Vert _{2}\leq 2^{1+j/2}\left(
\sum_{l=-r+1}^{2^{j}-1}\alpha _{l}^{2}\right) ^{1/2}.
\end{equation*}%
The third claim is proved similarly. To prove the final claim, we use the
following interpolation inequality: for every $0<s^{\prime }\leq 1$ there
exists a finite constant $C^{\ast }(s^{\prime })$ such that for every $h\in 
\mathcal{W}_{2}^{1}$%
\begin{equation}
\Vert h\Vert _{s^{\prime },2}\leq C^{\ast }(s^{\prime })(\Vert h\Vert
_{2}+\Vert D_{w}h\Vert _{2})^{s^{\prime }}\Vert h\Vert _{2}^{1-s^{\prime }}
\label{interpol}
\end{equation}%
holds. [This follows from (\ref{sob}) if $s^{\prime }=1$; if $s^{\prime }<1$
it follows from Theorem 6.7.1 in DeVore and Lorentz (1993) applied to the
intermediate spaces $(\mathbb{R},\mathbb{R})_{s^{\prime },\infty }$, $(%
\mathcal{L}^{2},\mathcal{W}_{2}^{1})_{s^{\prime },\infty }$, and to the
operator that maps any real number $a$ into $ah$, observing that $(\mathcal{L%
}^{2},\mathcal{W}_{2}^{1})_{s^{\prime },\infty }$ is equal to $\mathcal{B}%
_{s^{\prime }}$ up to a equivalence of norms, cf. p.196 in DeVore and
Lorentz (1993).] Observe that $f\in \mathcal{W}_{2}^{1}$ if $r\geq 2$. Now,
using (\ref{interpol}) with $h=f$, (\ref{L_2}), and (\ref{L_2_deriv})
completes the proof upon setting $C_{0}(s^{\prime })=\left( 2.5\right)
^{s^{\prime }}C^{\ast }(s^{\prime })$.
\end{proof}

\begin{lemma}
\label{havoc0} Assume $r\geq 1$ and let $\theta \in \Theta $.

a. Suppose the density $p(\theta )$ is bounded. Then for all $k\geq 1$ and $%
J\geq 1$%
\begin{equation*}
E\left\Vert p_{k,J,r}(\theta )-Ep_{k,J,r}(\theta )\right\Vert _{2}^{2}\leq
C_{1}(\theta ,r)\frac{2^{J}}{k},
\end{equation*}%
where $C_{1}(\theta ,r)=(\frac{r+1}{2})d_{r}^{2}\left\Vert p(\theta
)\right\Vert _{\infty }$ with $d_{r}$ defined in Proposition \ref{gram}.
Furthermore, for $r\geq 2$ and $0<s^{\prime }\leq 1$ 
\begin{equation*}
E\left\Vert p_{k,J,r}(\theta )-Ep_{k,J,r}(\theta )\right\Vert _{s^{\prime
},2}^{2}\leq C_{0}(s^{\prime })^{2}C_{1}(\theta ,r)\frac{2^{J(2s^{\prime
}+1)}}{k}
\end{equation*}%
holds for all $k\geq 1$ and $J\geq 1$, where $C_{0}(s^{\prime })$ is given
in Lemma \ref{basic}.

b. If $p(\theta )\in \mathcal{L}^{2}$, then for every $k$ 
\begin{equation*}
\lim_{J\rightarrow \infty }\left\Vert Ep_{k,J,r}(\theta )-p(\theta
)\right\Vert _{2}=0.
\end{equation*}%
If $p(\theta )\in \mathcal{B}_{t}$ for some $0<t<r$ then for all $k\geq 1$
and $J\geq 1$%
\begin{equation*}
\left\Vert Ep_{k,J,r}(\theta )-p(\theta )\right\Vert _{2}\leq
2^{-Jt}c_{t}^{\prime }\left\Vert p(\theta )\right\Vert _{t,2},
\end{equation*}%
where $c_{t}^{\prime }$ is the constant given in Proposition \ref{proj_error}
in Appendix \ref{App_Spline}.

c. If the assumptions of Part a (Part b) hold for (a version of) $p_{0}$ and 
$r_{\ast }$ in place of $p(\theta )$ and $r$, respectively, then the results
in Part a (Part b) also apply mutatis mutandis to $p_{n,j,r_{\ast }}$.
\end{lemma}

\begin{proof}
In view of Lemma \ref{basic}, the definition of $p_{k,J,r}(\theta )$, (\ref%
{L_2}) and (\ref{bes_0}), it suffices to bound $E\left( \hat{\gamma}%
_{lJ}^{(r)}(\theta )-E\hat{\gamma}_{lJ}^{(r)}(\theta )\right) ^{2}$ in order
to prove Part a. We obtain%
\begin{eqnarray}
&&E\left( \hat{\gamma}_{lJ}^{(r)}(\theta )-E\hat{\gamma}_{lJ}^{(r)}(\theta
)\right) ^{2}  \notag \\
&\leq &\frac{2^{2J}}{k}E\left(
\sum_{m=-r+1}^{2^{J}-1}g_{J}^{(r)lm}N_{mJ}^{(r)}(\rho (V_{i},\theta
))\right) ^{2}=\frac{2^{2J}}{k}\dint\limits_{0}^{1}\left(
\sum_{m=-r+1}^{2^{J}-1}g_{J}^{(r)lm}N_{mJ}^{(r)}(x)\right) ^{2}p(\theta ,x)dx
\notag \\
&\leq &\frac{2^{2J}}{k}\left\Vert p(\theta )\right\Vert _{\infty }\left\Vert
\sum_{m=-r+1}^{2^{J}-1}g_{J}^{(r)lm}N_{mJ}^{(r)}\right\Vert _{2}^{2}\leq 
\frac{2^{J}}{k}\left\Vert p(\theta )\right\Vert _{\infty
}\sum_{m=-r+1}^{2^{J}-1}\left( g_{J}^{(r)lm}\right) ^{2}  \notag \\
&\leq &\frac{2^{J}}{k}\left\Vert p(\theta )\right\Vert _{\infty }\left(
\sum_{m=-r+1}^{2^{J}-1}\left\vert g_{J}^{(r)lm}\right\vert \right) ^{2}\leq 
\frac{2^{J}}{k}d_{r}^{2}\left\Vert p(\theta )\right\Vert _{\infty },
\label{bound_gamma}
\end{eqnarray}%
where we have used independence, (\ref{L_2}), and Proposition \ref{gram}.
This establishes Part a. [Measurability of the $\mathcal{L}^{2}$-norm is
obvious, and measurability of the Besov-norm follows from Appendix \ref%
{App_Meas}.] Since $Ep_{k,J,r}(\theta )=\pi _{J}^{(r)}(p(\theta ))$, Part b
follows from Proposition \ref{proj_error} in Appendix \ref{App_Spline}. Part
c is proved completely analogously.
\end{proof}

\begin{lemma}
\label{havoc} Assume $r\geq 3$ and let $\theta $ be an interior point of $%
\Theta $ such that the partial derivative $\frac{\partial \rho (v,\theta )}{%
\partial \theta _{q}}$ at $\theta $ exists for every $v\in \mathcal{V}$.

a. Suppose the density $p(\theta )$ is bounded and $\sup_{v\in \mathcal{V}%
}\left\vert \frac{\partial \rho (v,\theta )}{\partial \theta _{q}}%
\right\vert <\infty $. Then for all $k\geq 1$ and $J\geq 1$%
\begin{equation*}
E\left\Vert \frac{\partial p_{k,J,r}(\theta )}{\partial \theta _{q}}-E\frac{%
\partial p_{k,J,r}(\theta )}{\partial \theta _{q}}\right\Vert _{2}^{2}\leq
C_{2}(\theta ,r)\frac{2^{3J}}{k},
\end{equation*}%
where $C_{2}(\theta ,r)=2(r+1)d_{r}^{2}\left\Vert p(\theta )\right\Vert
_{\infty }\sup_{v\in \mathcal{V}}\left\vert \frac{\partial \rho (v,\theta )}{%
\partial \theta _{q}}\right\vert ^{2}$.

b. Suppose there exists an open ball $B(\theta )\subseteq \Theta $ with
center $\theta $ such that $\frac{\partial p(\cdot ,x)}{\partial \theta _{q}}
$ and $\frac{\partial \rho (v,\cdot )}{\partial \theta _{q}}$ exist on $%
B(\theta )$ for every $x\in \lbrack 0,1]$ and $v\in \mathcal{V}$, suppose $%
\frac{\partial p(\theta ,\cdot )}{\partial \theta _{q}}$ belongs to $%
\mathcal{B}_{s}$ for some $0<s<r$, and that%
\begin{equation*}
\int_{0}^{1}\sup_{\theta ^{\prime }\in B(\theta )}\left\vert \frac{\partial
p(\theta ^{\prime },x)}{\partial \theta _{q}}\right\vert dx<\infty ,\text{ \
\ }\int_{\mathcal{V}}\sup_{\theta ^{\prime }\in B(\theta )}\left\vert \frac{%
\partial \rho (v,\theta ^{\prime })}{\partial \theta _{q}}\right\vert d\mu
(v)<\infty .
\end{equation*}%
Then for all $k\geq 1$ and $J\geq 1$%
\begin{equation*}
\left\Vert E\frac{\partial p_{k,J,r}(\theta )}{\partial \theta _{q}}-\frac{%
\partial p(\theta )}{\partial \theta _{q}}\right\Vert _{2}\leq
2^{-Js}c_{s}^{\prime }\left\Vert \frac{\partial p(\theta )}{\partial \theta
_{q}}\right\Vert _{s,2},
\end{equation*}%
where the constant $c_{s}^{\prime }$ is defined in Proposition \ref%
{proj_error} in Appendix \ref{App_Spline}. [If $\frac{\partial p(\theta
,\cdot )}{\partial \theta _{q}}\in \mathcal{B}_{s}$ is weakened to $\frac{%
\partial p(\theta ,\cdot )}{\partial \theta _{q}}\in \mathcal{L}^{2}$, then $%
\lim_{J\rightarrow \infty }\left\Vert E\frac{\partial p_{k,J,r}(\theta )}{%
\partial \theta _{q}}-\frac{\partial p(\theta )}{\partial \theta _{q}}%
\right\Vert =0$ holds.]
\end{lemma}

\begin{proof}
Observe that $p_{k,J,r}$ is differentiable at $\theta $ because $r\geq 3$ is
assumed. To prove Part a note that 
\begin{equation*}
\frac{\partial p_{k,J,r}(\theta )}{\partial \theta _{q}}-E\frac{\partial
p_{k,J,r}(\theta )}{\partial \theta _{q}}=\sum_{l=-r+1}^{2^{J}-1}\left( 
\frac{\partial \hat{\gamma}_{lJ}^{(r)}(\theta )}{\partial \theta _{q}}-E%
\frac{\partial \hat{\gamma}_{lJ}^{(r)}(\theta )}{\partial \theta _{q}}%
\right) N_{lJ}^{(r)},
\end{equation*}%
and that the $\mathcal{L}^{2}$-norm of this expression is measurable by
Fubini's Theorem; also note that the expectations in the above display exist
since the B-spline basis functions are bounded and since $\sup_{v\in 
\mathcal{V}}\left\vert \frac{\partial \rho (v,\theta )}{\partial \theta _{q}}%
\right\vert <\infty $ has been assumed. Now, using the chain rule and (\ref%
{L_2_deriv}), we obtain%
\begin{eqnarray}
&&E\left( \frac{\partial \hat{\gamma}_{lJ}^{(r)}(\theta )}{\partial \theta
_{q}}-E\frac{\partial \hat{\gamma}_{lJ}^{(r)}(\theta )}{\partial \theta _{q}}%
\right) ^{2}\leq \frac{2^{2J}}{k}E\left( \frac{\partial \rho (V_{i},\theta )%
}{\partial \theta _{q}}\sum_{m=-r+1}^{2^{J}-1}g_{J}^{(r)lm}N_{mJ}^{(r)\prime
}(x)_{\mid x=\rho (V_{i},\theta )}\right) ^{2}  \notag \\
&\leq &\frac{2^{2J}}{k}\sup_{v\in \mathcal{V}}\left\vert \frac{\partial \rho
(v,\theta )}{\partial \theta _{q}}\right\vert ^{2}\int_{0}^{1}\left(
\sum_{m=-r+1}^{2^{J}-1}g_{J}^{(r)lm}N_{mJ}^{(r)\prime }(x)\right)
^{2}p(\theta ,x)dx  \label{bound_6} \\
&\leq &\frac{2^{2J}}{k}\sup_{v\in \mathcal{V}}\left\vert \frac{\partial \rho
(v,\theta )}{\partial \theta _{q}}\right\vert ^{2}\left\Vert p(\theta
)\right\Vert _{\infty }\left\Vert
\sum_{m=-r+1}^{2^{J}-1}g_{J}^{(r)lm}N_{mJ}^{(r)\prime }\right\Vert
_{2}^{2}\leq \frac{2^{3J+2}}{k}d_{r}^{2}\sup_{v\in \mathcal{V}}\left\vert 
\frac{\partial \rho (v,\theta )}{\partial \theta _{q}}\right\vert
^{2}\left\Vert p(\theta )\right\Vert _{\infty }.  \notag
\end{eqnarray}%
An application of Lemma \ref{basic} then completes the proof of Part a.

To prove Part b, note that 
\begin{eqnarray}
\int_{\mathcal{V}}\frac{\partial }{\partial \theta _{q}}N_{mJ}^{(r)}(\rho
(v,\theta ))d\mu (v) &=&\frac{\partial }{\partial \theta _{q}}\int_{\mathcal{%
V}}N_{mJ}^{(r)}(\rho (v,\theta ))d\mu (v)  \notag \\
&=&\frac{\partial }{\partial \theta _{q}}\int_{0}^{1}N_{mJ}^{(r)}(x)p(\theta
,x)dx=\int_{0}^{1}N_{mJ}^{(r)}(x)\frac{\partial }{\partial \theta _{q}}%
p(\theta ,x)dx,  \notag
\end{eqnarray}%
where the two-fold interchange of integration and differentiation is
permitted by dominated convergence in view of the maintained dominance
assumptions on the derivatives of $\rho $ and $p$ as well as the boundedness
of the B-spline basis functions and their first derivative. Consequently, 
\begin{eqnarray}
E\frac{\partial p_{k,J,r}(\theta ,y)}{\partial \theta _{q}}
&=&2^{J}\sum_{l=-r+1}^{2^{J}-1}\sum_{m=-r+1}^{2^{J}-1}g_{J}^{(r)lm}\int_{%
\mathcal{V}}\frac{\partial }{\partial \theta _{q}}N_{mJ}^{(r)}(\rho
(v,\theta ))d\mu (v)N_{lJ}^{(r)}(y)  \label{interchange} \\
&=&2^{J}\sum_{l=-r+1}^{2^{J}-1}\sum_{m=-r+1}^{2^{J}-1}g_{J}^{(r)lm}%
\int_{0}^{1}N_{mJ}^{(r)}(x)\frac{\partial }{\partial \theta _{q}}p(\theta
,x)dxN_{lJ}^{(r)}(y)=\pi _{J}^{(r)}\left( \frac{\partial }{\partial \theta
_{q}}p(\theta )\right) ,  \notag
\end{eqnarray}%
and Part b now follows immediately from Proposition \ref{proj_error} in
Appendix \ref{App_Spline}.
\end{proof}

\begin{lemma}
\label{uniformrate}a. Suppose Assumption R(i) is satisfied, $r\geq 2$, $%
\Theta $ is a bounded subset of $\mathbb{R}^{b}$, and $\sup_{\theta \in
\Theta }\left\Vert p(\theta )\right\Vert _{\infty }<\infty $. Then there
exist finite positive constants $C_{3}$ and $C_{4}$, depending only on $%
\Theta $, $b$, $\rho $, $r$, and $\sup_{\theta \in \Theta }\left\Vert
p(\theta )\right\Vert _{\infty }$ but not on $k$ and $J$, such that 
\begin{equation*}
E\sup_{\theta \in \Theta }\Vert p_{k,J,r}(\theta )-Ep_{k,J,r}(\theta )\Vert
_{2}^{2}\leq C_{3}\frac{2^{J}J}{k},
\end{equation*}%
holds for all $k\geq 1$ and $J\geq 1$ satisfying $2^{J}J\leq C_{4}k$.
Furthermore, for $0<s^{\prime }\leq 1$%
\begin{equation}
E\sup_{\theta \in \Theta }\Vert p_{k,J,r}(\theta )-Ep_{k,J,r}(\theta )\Vert
_{s^{\prime },2}^{2}\leq C_{0}(s^{\prime })^{2}C_{3}\frac{2^{J(2s^{\prime
}+1)}J}{k}  \label{bes}
\end{equation}%
holds for all $k\geq 1$ and $J\geq 1$ satisfying $2^{J}J\leq C_{4}k$ where $%
C_{0}(s^{\prime })$ is given in Lemma \ref{basic}.

b. Suppose Assumption R(ii) is satisfied for some interior point $\theta
_{0} $ of $\Theta $, $\sup_{\theta \in B(\theta _{0})}\left\Vert p(\theta
)\right\Vert _{\infty }<\infty $ and $r\geq 3$ hold. Then there exist finite
positive constants $C_{5}$ and $C_{6}$, depending only on $B(\theta _{0})$, $%
b$, $\rho $, $r$ and $\sup_{\theta \in B(\theta _{0})}\left\Vert p(\theta
)\right\Vert _{\infty }$ but not on $k$ and $J$, such that for every $%
q=1,\ldots ,b$%
\begin{equation*}
E\sup_{\theta \in B(\theta _{0})}\left\Vert \frac{\partial }{\partial \theta
_{q}}p_{k,J,r}(\theta )-E\frac{\partial }{\partial \theta _{q}}%
p_{k,J,r}(\theta )\right\Vert _{2}^{2}\leq C_{5}\frac{2^{3J}J}{k}
\end{equation*}%
holds for all $k\geq 1$ and $J\geq 1$ satisfying $2^{J}J\leq C_{6}k$.

c. Suppose the assumptions of Part b are satisfied except that now $r\geq 4$%
. Then there exist finite positive constants $C_{7}$ and $C_{8}$, depending
only on $B(\theta _{0})$, $b$, $\rho $, $r$ and $\sup_{\theta \in B(\theta
_{0})}\left\Vert p(\theta )\right\Vert _{\infty }$ but not on $k$ and $J$,
such that for every $q,q^{\prime }=1,\ldots ,b$%
\begin{equation*}
E\sup_{\theta \in B(\theta _{0})}\left\Vert \frac{\partial ^{2}}{\partial
\theta _{q}\partial \theta _{q^{\prime }}}p_{k,J,r}(\theta )-E\frac{\partial
^{2}}{\partial \theta _{q}\partial \theta _{q^{\prime }}}p_{k,J,r}(\theta
)\right\Vert _{2}^{2}\leq C_{7}\frac{2^{5J}J}{k}
\end{equation*}%
holds for all $k\geq 1$ and $J\geq 1$ satisfying $2^{J}J\leq C_{8}k$.
\end{lemma}

\begin{proof}
a. By Lemma \ref{basic} we have 
\begin{equation*}
E\sup_{\theta \in \Theta }\Vert p_{k,J,r}(\theta )-Ep_{k,J,r}(\theta )\Vert
_{2}^{2}\leq 2^{-J}\sum_{l=-r+1}^{2^{J}-1}E\sup_{\theta \in \Theta }\left( 
\hat{\gamma}_{lJ}^{(r)}(\theta )-E\hat{\gamma}_{lJ}^{(r)}(\theta )\right)
^{2}.
\end{equation*}%
Note that the suprema in the above display are measurable as the functions
over which the suprema are taken depend continuously on $\theta $ in view of
assumption R(i) and $r\geq 2$. We bound the r.h.s. in the above display by
applying the moment inequality given in Proposition \ref{exp} in Appendix %
\ref{App_Moment}: fix an arbitrary $l$ and express the corresponding summand
in the above display as%
\begin{equation}
E\sup_{\theta \in \Theta }\left( \hat{\gamma}_{lJ}^{(r)}(\theta )-E\hat{%
\gamma}_{lJ}^{(r)}(\theta )\right) ^{2}=\frac{2^{2J}}{k^{2}}E\sup_{\theta
\in \Theta }\left\vert \sum_{i=1}^{k}h_{\theta ,l}(V_{i})\right\vert ^{2}
\label{coeff_0}
\end{equation}%
where 
\begin{equation*}
h_{\theta ,l}(v)=\sum_{m=-r+1}^{2^{J}-1}g_{J}^{(r)lm}\left[
N_{mJ}^{(r)}(\rho (v,\theta ))-EN_{mJ}^{(r)}(\rho (V_{i},\theta ))\right]
\end{equation*}%
and set $\mathcal{H}_{l,J,r}=\left\{ h_{\theta ,l}:\theta \in \Theta
\right\} $. Furthermore, set $U=d_{r}\max \left( 2,\sup_{\theta \in \Theta
}\left\Vert p(\theta )\right\Vert _{\infty }^{1/2}\right) $ and $\sigma
^{2}=2^{-J}U^{2}$. Then $0<\sigma \leq U$ holds, and using the calculations
that have led to (\ref{bound_gamma}) we obtain for every $\theta \in \Theta $%
\begin{equation*}
Eh_{\theta ,l}^{2}(V_{i})\leq E\left(
\sum_{m=-r+1}^{2^{J}-1}g_{J}^{(r)lm}N_{mJ}^{(r)}(\rho (v,\theta ))\right)
^{2}\leq 2^{-J}d_{r}^{2}\left\Vert p(\theta )\right\Vert _{\infty }\leq
2^{-J}d_{r}^{2}\sup_{\theta \in \Theta }\left\Vert p(\theta )\right\Vert
_{\infty }\leq \sigma ^{2}.
\end{equation*}%
Furthermore, using (\ref{norm_ineq}), we obtain for every $\theta \in \Theta 
$%
\begin{equation*}
\sup_{v\in \mathcal{V}}\left\vert h_{\theta ,l}\right\vert \leq
2d_{r}\left\Vert N^{(r)}\right\Vert _{\infty ,\mathbb{R}}\leq 2d_{r}\leq U.
\end{equation*}%
We next bound the uniform $\mathsf{L}^{\infty }$-covering numbers of $%
\mathcal{H}_{l,J,r}$: observe that the elements of $\mathcal{H}_{l,J,r}$
satisfy for $\theta $, $\theta ^{\prime }\in \Theta $%
\begin{equation}
\sup_{v\in \mathcal{V}}\left\vert h_{\theta ,l}(v)-h_{\theta ^{\prime
},l}(v)\right\vert \leq 2^{J+1}d_{r}L\left\Vert \theta -\theta ^{\prime
}\right\Vert ^{\alpha },  \label{Hoeld}
\end{equation}%
where $L$, $\alpha $ are the H\"{o}lder constants from Assumption R(i) and
where we have made use of the fact that $N^{(r)}$ has Lipschitz constant
bounded by $1$ for $r\geq 2$; cf. the discussion at the beginning of this
section. Since $\Theta $ is assumed to be bounded in $\mathbb{R}^{b}$, it
can be covered by fewer than $M/\delta ^{b}$ open balls with centers $\theta
_{i}\in \Theta $ and radius $\delta $, for $0<\delta \leq 1$ where $M$
depends only on $\Theta $. By (\ref{Hoeld}), the functions $h_{\theta
_{i},l} $ in $\mathcal{H}_{l,J,r}$ corresponding to the $\theta _{i}$'s give
rise to a covering of $\mathcal{H}_{l,J,r}$ by sup-norm balls of radius $%
2^{J+1}d_{r}L\delta ^{\alpha }$. Consequently, the $\mathsf{L}^{\infty }$%
-covering numbers satisfy%
\begin{equation}
N(\mathcal{H}_{l,J,r},\mathsf{L}^{\infty }(\mathcal{V}),\varepsilon )\leq
M\left( \frac{2^{J+1}d_{r}L}{\varepsilon }\right) ^{b/\alpha }\text{ \ \ for 
}0<\varepsilon \leq 2^{J+1}d_{r}L.  \label{covering}
\end{equation}%
Replacing $M$ by $M_{\ast }=M\max \left( 1,\left( U/(2d_{r}L)\right)
^{b/\alpha }\right) $ in (\ref{covering}), guarantees that (\ref{covering})
then holds for $0<\varepsilon \leq 2U$, which leads to%
\begin{equation}
N(\mathcal{H}_{l,J,r},\mathsf{L}^{\infty }(\mathcal{V}),\varepsilon )\leq
(AU/\varepsilon )^{v}\text{ \ \ for }0<\varepsilon \leq 2U,
\label{covering1}
\end{equation}%
for $v=\max (b/\alpha ,2)$ and $A=\max \left( 2^{J+1}M_{\ast }^{\alpha
/b}d_{r}LU^{-1},2e\right) $, where we have also enforced $v\geq 2$ and $A>e$%
. Note that, apart from the factor $2^{J}$, $A$ depends only on $\Theta $, $%
b $, $\rho $ (via $\alpha $ and $L$), $r$ (via $d_{r}$), and $\sup_{\theta
\in \Theta }\left\Vert p(\theta )\right\Vert _{\infty }$. Observe that $%
\mathcal{H}_{l,J,r}$ contains a countable sup--norm dense subset in view of (%
\ref{Hoeld}) and separability of $\Theta $. Hence the expectation bound in
Part a of Proposition \ref{exp} in Appendix \ref{App_Moment} applied to this
subset and with $b_{0}=v^{-1}$ now yields the existence of positive finite
constants $C_{3}^{\prime }$ and $C_{4}^{\prime }$ both depending only on $%
\Theta $, $b$, $\rho $, $r$, and $\sup_{\theta \in \Theta }\left\Vert
p(\theta )\right\Vert _{\infty }$, such that for all $J\in \mathbb{N}$ and
all $k\geq C_{4}^{\prime }2^{J}J$%
\begin{equation}
E\sup_{\theta \in \Theta }\left\vert \sum_{i=1}^{k}h_{\theta
,l}(V_{i})\right\vert ^{2}\leq C_{3}^{\prime }k2^{-J}J.  \label{coeff}
\end{equation}%
Since this bound does not depend on the summation index $l$, the proof of
the first claim is complete upon setting $C_{3}=(r+1)C_{3}^{\prime }/2$ and $%
C_{4}=1/C_{4}^{\prime }$. The second claim follows immediately from applying
(\ref{bes_0}) in Lemma \ref{basic} to the l.h.s. of (\ref{bes}) and using (%
\ref{coeff_0}) and (\ref{coeff}), the measurability of the supremum in (\ref%
{bes}) following from Appendix \ref{App_Meas}.

b. Observe that $p_{k,J,r}$ is continuously differentiable on $B(\theta
_{0}) $ because of $r\geq 3$ and Assumption R(ii). Similarly as in Part a we
have measurability of the suprema and obtain from Lemma \ref{basic}%
\begin{eqnarray*}
&&E\sup_{\theta \in B(\theta _{0})}\left\Vert \frac{\partial }{\partial
\theta _{q}}p_{k,J,r}(\theta )-E\frac{\partial }{\partial \theta _{q}}%
p_{k,J,r}(\theta )\right\Vert _{2}^{2} \\
&\leq &2^{-J}\sum_{l=-r+1}^{2^{J}-1}E\sup_{\theta \in B(\theta _{0})}\left( 
\frac{\partial }{\partial \theta _{q}}\hat{\gamma}_{lJ}^{(r)}(\theta )-E%
\frac{\partial }{\partial \theta _{q}}\hat{\gamma}_{lJ}^{(r)}(\theta
)\right) ^{2} \\
&=&2^{-J}\sum_{l=-r+1}^{2^{J}-1}\frac{2^{4J}}{k^{2}}E\sup_{\theta \in
B(\theta _{0})}\left\vert \sum_{i=1}^{k}h_{\theta
,l}^{(1)}(V_{i})\right\vert ^{2}
\end{eqnarray*}%
where 
\begin{equation*}
h_{\theta ,l}^{(1)}(v)=\frac{\partial \rho (v,\theta )}{\partial \theta _{q}}%
\sum_{m=-r+1}^{2^{J}-1}g_{J}^{(r)lm}\left[ N^{(r)\prime }(2^{J}\rho
(v,\theta )-m)-EN^{(r)\prime }(2^{J}\rho (V_{i},\theta )-m)\right] .
\end{equation*}%
Set $\mathcal{H}_{l,J,r}^{(1)}\mathcal{=}\left\{ h_{\theta ,l}^{(1)}:\theta
\in B(\theta _{0})\right\} $ and define%
\begin{equation*}
U=2d_{r}\sup_{\theta \in B(\theta _{0})}\sup_{v\in \mathcal{V}}\left\vert 
\frac{\partial \rho (v,\theta )}{\partial \theta _{q}}\right\vert \max
\left( 1,\sup_{\theta \in B(\theta _{0})}\left\Vert p(\theta )\right\Vert
_{\infty }^{1/2}\right)
\end{equation*}%
and $\sigma ^{2}=2^{-J}U^{2}$. Then $0<\sigma \leq U$ holds (where we
exclude the trivial case $U=0$). Observing that $N_{mJ}^{(r)\prime
}(x)=2^{J}N^{(r)\prime }(2^{J}x-m)$ by the chain rule, we obtain, using the
same calculations that have led to (\ref{bound_6}), for $\theta \in B(\theta
_{0})$%
\begin{equation*}
Eh_{\theta ,l}^{(1)2}(V_{i})\leq 2^{-J+2}d_{r}^{2}\sup_{\theta \in B(\theta
_{0})}\sup_{v\in \mathcal{V}}\left\vert \frac{\partial \rho (v,\theta )}{%
\partial \theta _{q}}\right\vert ^{2}\sup_{\theta \in B(\theta
_{0})}\left\Vert p(\theta )\right\Vert _{\infty }\leq \sigma ^{2}\text{.}
\end{equation*}%
Furthermore, for every $\theta \in B(\theta _{0})$%
\begin{equation*}
\sup_{v\in \mathcal{V}}\left\vert h_{\theta ,l}^{(1)}\right\vert \leq
2\sup_{v\in \mathcal{V}}\left\vert \frac{\partial \rho (v,\theta )}{\partial
\theta _{q}}\right\vert d_{r}\left\Vert N^{(r)\prime }\right\Vert _{\infty ,%
\mathbb{R}}\leq 2d_{r}\sup_{\theta \in B(\theta _{0})}\sup_{v\in \mathcal{V}%
}\left\vert \frac{\partial \rho (v,\theta )}{\partial \theta _{q}}%
\right\vert \leq U,
\end{equation*}%
where we have made use of (\ref{norm_ineq_2}). To bound the uniform $\mathsf{%
L}^{\infty }$-covering numbers of $\mathcal{H}_{l,J,r}^{(1)}$, observe that
the elements of $\mathcal{H}_{l,J,r}^{(1)}$ satisfy for $\theta $, $\theta
^{\prime }\in B(\theta _{0})$%
\begin{eqnarray*}
&&\sup_{v\in \mathcal{V}}\left\vert h_{\theta ,l}^{(1)}(v)-h_{\theta
^{\prime },l}^{(1)}(v)\right\vert \leq \\
&&2d_{r}\left\Vert N^{(r)\prime }\right\Vert _{\infty ,\mathbb{R}%
}\sup_{\theta \in B(\theta _{0})}\sup_{v\in \mathcal{V}}\left\Vert \nabla
_{\theta }^{2}\rho (v,\theta )\right\Vert \left\Vert \theta -\theta ^{\prime
}\right\Vert +2^{J+1}d_{r}\sup_{\theta \in B(\theta _{0})}\sup_{v\in 
\mathcal{V}}\left\Vert \nabla _{\theta }\rho (v,\theta )\right\Vert
^{2}\left\Vert \theta -\theta ^{\prime }\right\Vert \\
&\leq &2^{J+1}d_{r}\left\{ \sup_{\theta \in B(\theta _{0})}\sup_{v\in 
\mathcal{V}}\left\Vert \nabla _{\theta }^{2}\rho (v,\theta )\right\Vert
+\sup_{\theta \in B(\theta _{0})}\sup_{v\in \mathcal{V}}\left\Vert \nabla
_{\theta }\rho (v,\theta )\right\Vert ^{2}\right\} \left\Vert \theta -\theta
^{\prime }\right\Vert \leq 2^{J}c_{\ast }\left\Vert \theta -\theta ^{\prime
}\right\Vert ,
\end{eqnarray*}%
where we have made use of (\ref{norm_ineq_2}), of the bound on the Lipschitz
constant of $N^{(r)\prime }$ given at the beginning of this section, and of
the boundedness of $B(\theta _{0})$; the constant $c_{\ast }$ is finite and
depends only on $\rho $, $r$, and $B(\theta _{0})$. Proceeding as in the
proof of Part a we obtain%
\begin{equation*}
N(\mathcal{H}_{l,J,r}^{(1)},\mathsf{L}^{\infty }(\mathcal{V}),\varepsilon
)\leq (AU/\varepsilon )^{v}\text{ \ \ for }0<\varepsilon \leq 2U,
\end{equation*}%
for $v=\max (b,2)$ and $A=\max \left( 2^{J}M^{1/b}\max (c_{\ast
}U^{-1},1),2e\right) $ with $M$ only depending on $B(\theta _{0})$. Note
that, apart from the factor $2^{J}$, $A$ depends only on $B(\theta _{0})$, $%
b $, $\rho $, $r$ and $\sup_{\theta \in B(\theta _{0})}\left\Vert p(\theta
)\right\Vert _{\infty }$. Part a of Proposition \ref{exp} in Appendix \ref%
{App_Moment} applied to a countable sup-norm dense subset of $\mathcal{H}%
_{l,J,r}^{(1)}$ and with $b_{0}=v^{-1}$ now yields the existence of positive
finite constants $C_{5}^{\prime }$ and $C_{6}^{\prime }$ depending only on $%
B(\theta _{0})$, $b$, $\rho $, $r$ and $\sup_{\theta \in B(\theta
_{0})}\left\Vert p(\theta )\right\Vert _{\infty }$, such that for all $J\in 
\mathbb{N}$ and all $k\geq C_{6}^{\prime }2^{J}J$%
\begin{equation*}
E\sup_{\theta \in \Theta }\left\vert \sum_{i=1}^{k}h_{\theta
,l}^{(1)}(V_{i})\right\vert ^{2}\leq C_{5}^{\prime }k2^{-J}J
\end{equation*}%
holds. Since this bound does not depend on $l$, the proof is complete upon
setting $C_{5}=(r+1)C_{5}^{\prime }/2$ and $C_{6}=1/C_{6}^{\prime }$.

c. The proof is similar to the proof of Part b: Observe that $p_{k,J,r}$ is
twice continuously differentiable on $B(\theta _{0})$ because of $r\geq 4$
and Assumption R(ii). By Lemma \ref{basic} we have%
\begin{equation*}
E\sup_{\theta \in B(\theta _{0})}\left\Vert \frac{\partial ^{2}}{\partial
\theta _{q}\partial \theta _{q^{\prime }}}p_{k,J,r}(\theta )-E\frac{\partial
^{2}}{\partial \theta _{q}\partial \theta _{q^{\prime }}}p_{k,J,r}(\theta
)\right\Vert _{2}^{2}\leq \frac{2^{5J}}{k^{2}}\sum_{l=-r+1}^{2^{J}-1}E\sup_{%
\theta \in B(\theta _{0})}\left\vert \sum_{i=1}^{k}h_{\theta
,l}^{(2)}(V_{i})\right\vert ^{2}
\end{equation*}%
where%
\begin{eqnarray*}
h_{\theta ,l}^{(2)}(v) &=&2^{-J}\frac{\partial ^{2}\rho (v,\theta )}{%
\partial \theta _{q}\partial \theta _{q^{\prime }}}%
\sum_{m=-r+1}^{2^{J}-1}g_{J}^{(r)lm}\left[ N^{(r)\prime }(2^{J}\rho
(v,\theta )-m)-EN^{(r)\prime }(2^{J}\rho (V_{i},\theta )-m)\right] + \\
&&\frac{\partial \rho (v,\theta )}{\partial \theta _{q}}\frac{\partial \rho
(v,\theta )}{\partial \theta _{q^{\prime }}}%
\sum_{m=-r+1}^{2^{J}-1}g_{J}^{(r)lm}\left[ N^{(r)\prime \prime }(2^{J}\rho
(v,\theta )-m)-EN^{(r)\prime \prime }(2^{J}\rho (V_{i},\theta )-m)\right] .
\end{eqnarray*}%
Set $\mathcal{H}_{l,J,r}^{(2)}\mathcal{=}\left\{ h_{\theta ,l}^{(2)}:\theta
\in B(\theta _{0})\right\} $, set%
\begin{eqnarray*}
U &=&d_{r}\max \left\{ \sup_{\theta \in B(\theta _{0})}\sup_{v\in \mathcal{V}%
}\left\Vert \nabla _{\theta }^{2}\rho (v,\theta )\right\Vert +4\sup_{\theta
\in B(\theta _{0})}\sup_{v\in \mathcal{V}}\left\Vert \nabla _{\theta }\rho
(v,\theta )\nabla _{\theta }\rho (v,\theta )^{\prime }\right\Vert ,\right. \\
&&\left. \sup_{\theta \in B(\theta _{0})}\left\Vert p(\theta )\right\Vert
_{\infty }^{1/2}\left[ 2\sup_{\theta \in B(\theta _{0})}\sup_{v\in \mathcal{V%
}}\left\Vert \nabla _{\theta }^{2}\rho (v,\theta )\right\Vert
^{2}+32\sup_{\theta \in B(\theta _{0})}\sup_{v\in \mathcal{V}}\left\Vert
\nabla _{\theta }\rho (v,\theta )\nabla _{\theta }\rho (v,\theta )^{\prime
}\right\Vert ^{2}\right] ^{1/2}\right\}
\end{eqnarray*}%
and $\sigma ^{2}=2^{-J}U^{2}$. Then $0<\sigma \leq U$ holds (where we
exclude the trivial case $U=0$), and for $\theta \in B(\theta _{0})$ we have%
\begin{eqnarray*}
Eh_{\theta ,l}^{(2)2}(V_{i}) &\leq &2^{3-3J}d_{r}^{2}\sup_{\theta \in
B(\theta _{0})}\left\Vert p(\theta )\right\Vert _{\infty }\sup_{\theta \in
B(\theta _{0})}\sup_{v\in \mathcal{V}}\left\vert \frac{\partial ^{2}\rho
(v,\theta )}{\partial \theta _{q}\partial \theta _{q^{\prime }}}\right\vert
^{2} \\
&&+2^{5-J}d_{r}^{2}\sup_{\theta \in B(\theta _{0})}\left\Vert p(\theta
)\right\Vert _{\infty }\sup_{\theta \in B(\theta _{0})}\sup_{v\in \mathcal{V}%
}\left\vert \frac{\partial \rho (v,\theta )}{\partial \theta _{q}}\frac{%
\partial \rho (v,\theta )}{\partial \theta _{q^{\prime }}}\right\vert
^{2}\leq \sigma ^{2},
\end{eqnarray*}%
using a calculation similar to the one that has led to (\ref{bound_6}) and
making use of Lemma \ref{basic}. Similarly, for $\theta \in B(\theta _{0})$
we obtain%
\begin{equation*}
\sup_{v\in \mathcal{V}}\left\vert h_{\theta ,l}^{(2)}(v)\right\vert \leq
2d_{r}\left\{ 2^{-J}\sup_{v\in \mathcal{V}}\left\vert \frac{\partial
^{2}\rho (v,\theta )}{\partial \theta _{q}\partial \theta _{q^{\prime }}}%
\right\vert \left\Vert N^{(r)\prime }\right\Vert _{\infty ,\mathbb{R}%
}+\sup_{v\in \mathcal{V}}\left\vert \frac{\partial \rho (v,\theta )}{%
\partial \theta _{q}}\frac{\partial \rho (v,\theta )}{\partial \theta
_{q^{\prime }}}\right\vert \left\Vert N^{(r)\prime \prime }\right\Vert
_{\infty ,\mathbb{R}}\right\} \leq U,
\end{equation*}%
using $\left\Vert N^{(r)\prime }\right\Vert _{\infty ,\mathbb{R}}\leq 1$ and 
$\left\Vert N^{(r)\prime \prime }\right\Vert _{\infty ,\mathbb{R}}\leq 2$,
cf. (\ref{norm_ineq_2}), (\ref{norm_ineq_3}). Furthermore, for $\theta $, $%
\theta ^{\prime }\in B(\theta _{0})$ we get again using (\ref{norm_ineq_2}),
(\ref{norm_ineq_3}), the bounds for the Lipschitz constants of $N^{(r)\prime
}$ and $N^{(r)\prime \prime }$ given at the beginning of this section, and
boundedness of $B(\theta _{0})$%
\begin{eqnarray*}
\sup_{v\in \mathcal{V}}\left\vert h_{\theta ,l}^{(2)}(v)-h_{\theta ^{\prime
},l}^{(2)}(v)\right\vert &\leq &2^{1-J}d_{r}L^{\prime }\left\Vert \theta
-\theta ^{\prime }\right\Vert ^{\beta } \\
&&+12d_{r}\sup_{\theta \in B(\theta _{0})}\sup_{v\in \mathcal{V}}\left\Vert
\nabla _{\theta }\rho (v,\theta )\right\Vert \sup_{\theta \in B(\theta
_{0})}\sup_{v\in \mathcal{V}}\left\Vert \nabla _{\theta }^{2}\rho (v,\theta
)\right\Vert \left\Vert \theta -\theta ^{\prime }\right\Vert \\
&&+2^{J+3}d_{r}\sup_{\theta \in B(\theta _{0})}\sup_{v\in \mathcal{V}%
}\left\Vert \nabla _{\theta }\rho (v,\theta )\right\Vert \sup_{\theta \in
B(\theta _{0})}\sup_{v\in \mathcal{V}}\left\Vert \nabla _{\theta }\rho
(v,\theta )\right\Vert ^{2}\left\Vert \theta -\theta ^{\prime }\right\Vert \\
&\leq &2^{J}c_{\ast \ast }\left\Vert \theta -\theta ^{\prime }\right\Vert
^{\beta }
\end{eqnarray*}%
with the constant $c_{\ast \ast }$ being finite and depending only on $%
B(\theta _{0})$, $r$, $\rho $. Proceeding as in the proof of Part a we obtain%
\begin{equation*}
N(\mathcal{H}_{l,J,r}^{(2)},\mathsf{L}^{\infty }(\mathcal{V}),\varepsilon
)\leq (AU/\varepsilon )^{v}\text{ \ \ for }0<\varepsilon \leq 2U,
\end{equation*}%
where now $v=\max (b/\beta ,2)$ and $A=\max \left( 2^{J}M^{\beta /b}\max
(c_{\ast \ast }U^{-1},1),2e\right) $ with $M$ only depending on $B(\theta
_{0})$. Again, apart from the factor $2^{J}$, $A$ depends only on $B(\theta
_{0})$, $b$, $\rho $, $r$, and $\sup_{\theta \in B(\theta _{0})}\left\Vert
p(\theta )\right\Vert _{\infty }$. Part a of Proposition \ref{exp} in
Appendix \ref{App_Moment} applied to a countable sup-norm dense subset of $%
\mathcal{H}_{l,J,r}^{(2)}$ and with $b_{0}=v^{-1}$ now yields the existence
of positive finite constants $C_{7}^{\prime }$ and $C_{8}^{\prime }$
depending only on $B(\theta _{0})$, $b$, $\rho $, $r$, and $\sup_{\theta \in
B(\theta _{0})}\left\Vert p(\theta )\right\Vert _{\infty }$, such that for
all $J\in \mathbb{N}$ and all $k\geq C_{8}^{\prime }2^{J}J$%
\begin{equation*}
E\sup_{\theta \in \Theta }\left\vert \sum_{i=1}^{k}h_{\theta
,l}^{(2)}(V_{i})\right\vert ^{2}\leq C_{7}^{\prime }k2^{-J}J
\end{equation*}%
holds. Since this bound does not depend on $l$, the proof is complete upon
setting $C_{7}=(r+1)C_{7}^{\prime }/2$ and $C_{8}=1/C_{8}^{\prime }$.
\end{proof}

\begin{corollary}
\label{consistency} Suppose Assumption R(i) is satisfied and $r\geq 2$.
Suppose further that $\Theta $ is a bounded subset of $\mathbb{R}^{b}$ and
that $\{p(\theta ):\theta \in \Theta \}$ is bounded in $\mathsf{B}_{t}$ for
some $1/2<t\leq 1$. If $J_{k}\in \mathbb{N}$ satisfies 
\begin{equation}
\sup_{k\geq 1}2^{J_{k}(2t+1)}J_{k}/k<\infty ,  \label{cond}
\end{equation}%
then $\sup_{\theta \in \Theta }\Vert p_{k,J_{k},r}(\theta )\Vert _{t,2}$ is
stochastically bounded, i.e., 
\begin{equation*}
\lim_{M\rightarrow \infty }\sup_{k\geq 1}\Pr \left( \sup_{\theta \in \Theta
}\Vert p_{k,J_{k},r}(\theta )\Vert _{t,2}>M\right) =0.
\end{equation*}%
If (\ref{cond}) holds and $J_{k}\rightarrow \infty $ for $k\rightarrow
\infty $, then, for every $0<t^{\prime }<t$, $\sup_{\theta \in \Theta }\Vert
p_{k,J_{k},r}(\theta )-p(\theta )\Vert _{t^{\prime },2}$ as well as $%
\sup_{\theta \in \Theta }\Vert p_{k,J_{k},r}(\theta )-p(\theta )\Vert
_{\infty }$ converge to zero in (outer) probability as $k\rightarrow \infty $%
.
\end{corollary}

\begin{proof}
Observe that under (\ref{cond}) we have $2^{J_{k}}J_{k}\leq C_{4}k$ for $k$
large enough, where $C_{4}$ is as in Lemma \ref{uniformrate}, and that $%
\{p(\theta ):\theta \in \Theta \}$ is sup-norm bounded. Now, using Lemma \ref%
{uniformrate} together with Ljapunov's inequality as well as Proposition \ref%
{proj_error_2} in Appendix \ref{App_Spline}, we arrive, for $k$ large
enough, at%
\begin{eqnarray*}
E\sup_{\theta \in \Theta }\Vert p_{k,J_{k},r}(\theta )\Vert _{t,2} &\leq
&E\sup_{\theta \in \Theta }\Vert p_{k,J_{k},r}(\theta
)-Ep_{k,J_{k},r}(\theta )\Vert _{t,2}+\sup_{\theta \in \Theta }\Vert
Ep_{k,J_{k},r}(\theta )\Vert _{t,2} \\
&\leq &C_{0}(t)\sqrt{C_{3}}2^{J_{k}t}\sqrt{\frac{2^{J_{k}}J_{k}}{k}}%
+\sup_{\theta \in \Theta }\Vert \pi _{J_{k}}^{(r)}(p(\theta ))\Vert _{t,2} \\
&\leq &C_{0}(t)\sqrt{C_{3}}\sup_{k\geq 1}2^{J_{k}t}\sqrt{\frac{2^{J_{k}}J_{k}%
}{k}}+c_{t}^{\prime \prime }\sup_{\theta \in \Theta }\Vert p(\theta )\Vert
_{t,2}<\infty ,
\end{eqnarray*}%
where we have used the already established fact that $Ep_{k,J_{k},r}(\theta
)=\pi _{J_{k}}^{(r)}(p(\theta ))$. [Measurability of $\sup_{\theta \in
\Theta }\Vert p_{k,J_{k},r}(\theta )\Vert _{t,2}$ follows from Appendix \ref%
{App_Meas}.] Together with the observation that $E\sup_{\theta \in \Theta
}\Vert p_{k,J_{k},r}(\theta )\Vert _{t,2}<\infty $ for every $k\geq 1$, this
completes the proof of the first claim. Next, Lemma \ref{uniformrate}
(applied with $s^{\prime }=t^{\prime }$) gives for $k$ large enough ($%
E^{\ast }$ denoting outer expectation)%
\begin{eqnarray*}
E^{\ast }\sup_{\theta \in \Theta }\Vert p_{k,J_{k},r}(\theta )-p(\theta
)\Vert _{t^{\prime },2} &\leq &E\sup_{\theta \in \Theta }\Vert
p_{k,J_{k},r}(\theta )-Ep_{k,J_{k},r}(\theta )\Vert _{t^{\prime
},2}+\sup_{\theta \in \Theta }\Vert \pi _{J_{k}}^{(r)}(p(\theta ))-p(\theta
)\Vert _{t^{\prime },2} \\
&\leq &C_{0}(t^{\prime })\sqrt{C_{3}}2^{J_{k}t^{\prime }}\sqrt{\frac{%
2^{J_{k}}J_{k}}{k}}+2^{-J_{k}(t-t^{\prime })}c_{t,t^{\prime }}^{\prime
\prime \prime }\sup_{\theta \in \Theta }\Vert p(\theta )\Vert _{t,2},
\end{eqnarray*}%
where we have used Proposition \ref{proj_error_2} in Appendix \ref%
{App_Spline} in the final step. The upper bound now converges to zero as $%
k\rightarrow \infty $. The claim regarding the sup-norm now follows from
Proposition \ref{elem} in Appendix \ref{App_Spline}.
\end{proof}

\bigskip

The following corollary is proved analogously using Lemma \ref{havoc0}
instead of Lemma \ref{uniformrate}, with measurability of the relevant
quantities following from Appendix \ref{App_Meas}.

\begin{corollary}
\label{consistency_2} Suppose $r_{\ast }\geq 2$ and that $p_{0}\in \mathcal{B%
}_{t}$ for some $1/2<t\leq 1$. If $j_{n}\in \mathbb{N}$ satisfies 
\begin{equation}
\sup_{n\geq 1}2^{j_{n}(2t+1)}/n<\infty ,  \label{cond_2}
\end{equation}%
then $\Vert p_{n,j_{n},r_{\ast }}\Vert _{t,2}$ is stochastically bounded,
i.e., 
\begin{equation*}
\lim_{M\rightarrow \infty }\sup_{n\geq 1}\Pr \left( \Vert p_{n,j_{n},r_{\ast
}}\Vert _{t,2}>M\right) =0.
\end{equation*}%
If (\ref{cond_2}) holds and $j_{n}\rightarrow \infty $ for $n\rightarrow
\infty $, then, for every $0<t^{\prime }<t$, $\Vert p_{n,j_{n},r_{\ast
}}-p_{0}\Vert _{t^{\prime }}$ as well as $\Vert p_{n,j_{n},r_{\ast }}-\tilde{%
p}_{0}\Vert _{\infty }$\ converge to zero in probability as $n\rightarrow
\infty $, where $\tilde{p}_{0}$ is the continuous version of $p_{0}$.
\end{corollary}

\section{Uniform Central Limit Theorems for Spline Projection Estimators}

We now study the difference between the random (signed) measure $%
P_{k,J,r}(\theta )$ given by%
\begin{equation*}
dP_{k,J,r}(\theta )(y)=p_{k,J,r}(\theta ,y)dy
\end{equation*}%
and $P_{k}(\theta )$, acting on Besov classes by integration. In the
following $\left\Vert \nu \right\Vert _{\mathcal{F}}$ stands for $\sup_{f\in 
\mathcal{F}}\left\vert \nu (f)\right\vert $, where $\nu $ is a (signed)
measure.

\begin{theorem}
\label{UCLT} Suppose Assumption R(i) is satisfied, $r\geq 2$, $\Theta $ is a
bounded subset of $\mathbb{R}^{b}$, and $\left\{ p(\theta ):\theta \in
\Theta \right\} $ is a bounded subset of $\mathcal{B}_{t}$ for some $t$, $%
0<t<r$. Let $\mathcal{F}$ be a (non-empty) bounded subset of $\mathsf{B}_{s}$
for some $s$, $1/2<s<1$. Then for every $1/2<s^{\prime }\leq s$ there is a
finite positive constant $C_{9}$, depending only on $s$, $s^{\prime }$, $t$, 
$\mathcal{F}$, $\Theta $, $b$, $\alpha $, $L$, and $\left\{ p(\theta
):\theta \in \Theta \right\} $ but not on $J$ and $k$, such that for every $%
J\geq 1$ and $k\geq 1$ 
\begin{equation}
E\sup_{\theta \in \Theta }\left\Vert P_{k,J,r}(\theta )-P_{k}(\theta
)\right\Vert _{\mathcal{F}}\leq C_{9}(2^{-J(t+s)}+2^{-J(s-s^{\prime
})}k^{-1/2}).  \label{rate}
\end{equation}%
Furthermore, 
\begin{equation}
\sup_{\theta \in \Theta }\Vert P_{k}(\theta )-P(\theta )\Vert _{\mathcal{F}%
}=O_{p}(k^{-1/2})  \label{rate_2}
\end{equation}%
holds. Finally, if $J_{k}\rightarrow \infty $ as $k\rightarrow \infty $
satisfies $2^{-J_{k}(t+s)}=o(k^{-1/2})$, then for every $\theta \in \Theta $ 
\begin{equation*}
\sqrt{k}\left( P_{k,J_{k},r}(\theta )-P(\theta )\right) \rightsquigarrow
_{\ell ^{\infty }(\mathcal{F})}G_{P(\theta )},
\end{equation*}%
where $G_{P(\theta )}$ is a sample-bounded and sample-continuous generalized 
$P(\theta )$-Brownian bridge indexed by $\mathcal{F}$. Here $%
\rightsquigarrow _{\ell ^{\infty }(\mathcal{F})}$denotes convergence in law
as defined in Chapter 1 of van der Vaart and Wellner (1996).
\end{theorem}

\begin{proof}
We first note that $\sup_{\theta \in \Theta }\left\Vert P_{k,J,r}(\theta
)-P_{k}(\theta )\right\Vert _{\mathcal{F}}$ and $\sup_{\theta \in \Theta
}\Vert P_{k}(\theta )-P(\theta )\Vert _{\mathcal{F}}$ are measurable since
they can be represented as suprema over countable dense subsets of $\Theta $
and $\mathcal{F}$ in view of Assumption R(i), $r\geq 2$, and separability of 
$\mathcal{F}$. For $f\in \mathcal{F}$ we can write, using (\ref{aux_est_2a}%
), (\ref{aux_est_2b}), (\ref{proj_kern}) and symmetry of the projection
kernel $K_{J}^{(r)}$,%
\begin{eqnarray*}
&&(P_{k,J,r}(\theta )-P_{k}(\theta ))(f)=\frac{1}{k}\sum_{i=1}^{k}\left(
\int_{0}^{1}f(y)K_{J}^{(r)}(X_{i}(\theta ),y)dy-f(X_{i}(\theta ))\right) \\
&=&\frac{1}{k}\sum_{i=1}^{k}(\pi _{J}^{(r)}(f)-f)(X_{i}(\theta
))=(P_{k}(\theta )-P(\theta ))(\pi _{J}^{(r)}(f)-f)+\int_{0}^{1}(\pi
_{J}^{(r)}(f)-f)(y)p(\theta )(y)dy \\
&=&A+B.
\end{eqnarray*}%
Consider first term B: Using $f\in \mathcal{L}^{2}$, $p(\theta )\in \mathcal{%
L}^{2}$, self-adjointness and idempotency of the projection $Id-\pi
_{J}^{(r)}$ we obtain%
\begin{eqnarray}
\left\vert \int_{0}^{1}\left( \pi _{J}^{(r)}(f)-f\right) (y)p(\theta
)(y)dy\right\vert &=&\left\vert \int_{0}^{1}\left( \left( Id-\pi
_{J}^{(r)}\right) f\right) (y)\left( \left( Id-\pi _{J}^{(r)}\right)
p(\theta )\right) (y)dy\right\vert  \notag \\
&\leq &\left\Vert f-\pi _{J}^{(r)}(f)\right\Vert _{2}\left\Vert p(\theta
)-\pi _{J}^{(r)}(p(\theta ))\right\Vert _{2}  \notag \\
&\leq &c_{s}^{\prime }c_{t}^{\prime }\left\Vert f\right\Vert
_{s,2}\left\Vert p(\theta )\right\Vert _{t,2}2^{-J(s+t)},  \label{bound_B}
\end{eqnarray}%
where we have used Proposition \ref{proj_error} for the last inequality.
Consider next the term A: Define for $J\geq 1$ the class of functions 
\begin{eqnarray}
\mathcal{F}_{J,r,\rho } &=&\left\{ \int_{0}^{1}K_{J}^{(r)}(\rho (\cdot
,\theta ),y)f(y)dy-f(\rho (\cdot ,\theta )):f\in \mathcal{F},\,\theta \in
\Theta \right\}  \notag \\
&=&\left\{ (\pi _{J}^{(r)}(f)-f)(\rho (\cdot ,\theta )):f\in \mathcal{F}%
,\,\theta \in \Theta \right\} ,  \label{F_J}
\end{eqnarray}%
which allows us to write 
\begin{equation}
E\sup_{\theta \in \Theta }\sup_{f\in \mathcal{F}}\left\vert (P_{k}(\theta
)-P(\theta ))(\pi _{J}^{(r)}(f)-f)\right\vert =\frac{1}{k}E\sup_{h\in 
\mathcal{F}_{J,r,\rho }}\left\vert \sum_{i=1}^{k}\left(
h(V_{i})-Eh(V_{i})\right) \right\vert \text{.}  \label{A}
\end{equation}%
Choose an arbitrary $s^{\prime }$ satisfying $1/2<s^{\prime }\leq s$ and
observe that $(\pi _{J}^{(r)}(f)-f)\in \mathsf{B}_{s}\subseteq \mathsf{B}%
_{s^{^{\prime }}}$ since $\mathcal{F}\subseteq \mathsf{B}_{s}$ by assumption
and that $\mathcal{S}_{J}(r)\subseteq \mathsf{B}_{s}\subseteq \mathsf{B}%
_{s^{^{\prime }}}$ in view of $s<1<r-1/2$. Propositions \ref{elem} and \ref%
{proj_error_2} in Appendix \ref{App_Spline} then give 
\begin{eqnarray*}
\sup_{h\in \mathcal{F}_{J,r,\rho }}\sup_{v\in \mathcal{V}}\left\vert
h(v)-Eh(V_{i})\right\vert &\leq &2\sup_{h\in \mathcal{F}_{J,r,\rho
}}\sup_{v\in \mathcal{V}}\left\vert h(v)\right\vert \leq 2\sup_{f\in 
\mathcal{F}}\left\Vert \pi _{J}^{(r)}(f)-f\right\Vert _{\infty } \\
&\leq &2c_{s^{\prime }}\sup_{f\in \mathcal{F}}\left\Vert \pi
_{J}^{(r)}(f)-f\right\Vert _{s^{\prime },2}\leq 2c_{s^{\prime
}}c_{s,s^{\prime }}^{\prime \prime \prime }\sup_{f\in \mathcal{F}}\left\Vert
f\right\Vert _{s,2}2^{-J(s-s^{\prime })}=:U
\end{eqnarray*}%
where $U<\infty $ since $\mathcal{F}$ is a (non-empty) bounded subset of $%
\mathsf{B}_{s}$. We may assume $U>0$, the case $U=0$ being trivial. Since $%
\mathcal{F}_{J,r,\rho }$ contains a countable sup-norm dense subset in view
of Proposition \ref{ent1} below, we may apply the moment inequality from
Proposition \ref{exp}, part b, in Appendix \ref{App_Moment} to (\ref{A})
(with $U$ as above, $\sigma =U$, $A^{\prime }=c^{\ast s^{\prime }}/\left(
2c_{s^{\prime }}c_{s,s^{\prime }}^{\prime \prime \prime }\sup_{f\in \mathcal{%
F}}\left\Vert f\right\Vert _{s,2}\right) $, and with $w=1/s^{\prime }$) and
make use of the entropy bound in Proposition \ref{ent1} below with $%
\varepsilon ^{\ast }=4c_{s^{\prime }}c_{s,s^{\prime }}^{\prime \prime \prime
}\sup_{f\in \mathcal{F}}\left\Vert f\right\Vert _{s,2}\geq 2U$. This gives
the bound%
\begin{equation*}
E\sup_{\theta \in \Theta }\sup_{f\in \mathcal{F}}\left\vert (P_{k}(\theta
)-P(\theta ))(\pi _{J}^{(r)}(f)-f)\right\vert \leq 2^{-J(s-s^{\prime
})+1}k^{-1/2}c_{s^{\prime }}c_{s,s^{\prime }}^{\prime \prime \prime
}\sup_{f\in \mathcal{F}}\left\Vert f\right\Vert _{s,2}b_{2}
\end{equation*}%
where the constant $b_{2}$ only depends on $A^{\prime }$ and $w$. Together
with (\ref{bound_B}), this proves the bound (\ref{rate}). To prove the
second claim, define the class%
\begin{equation}
\mathcal{F}_{\rho }=\left\{ f(\rho (\cdot ,\theta )):f\in \mathcal{F}%
,\,\theta \in \Theta \right\}  \label{F_rho}
\end{equation}%
and note that $\mathcal{F}_{\rho }$ is uniformly bounded since $\mathcal{F}$
is and that%
\begin{equation*}
\sup_{\theta \in \Theta }\Vert P_{k}(\theta )-P(\theta )\Vert _{\mathcal{F}}=%
\frac{1}{k}\sup_{h\in \mathcal{F}_{\rho }}\left\vert \sum_{i=1}^{k}\left(
h(V_{i})-Eh(V_{i})\right) \right\vert .
\end{equation*}%
Now (\ref{rate_2}) follows since $\mathcal{F}_{\rho }$ is a universal
Donsker class by Proposition \ref{ent1} below. The third claim of the
theorem follows immediately from (\ref{rate}) with $s^{\prime }$ chosen to
satisfy $s^{\prime }<s$, from the assumptions on $J_{k}$, and from the
universal Donsker property of $\left\{ f(\rho (\cdot ,\theta )):f\in 
\mathcal{F}\right\} $ for every $\theta $, which it inherits from $\mathcal{F%
}_{\rho }$.
\end{proof}

\begin{proposition}
\label{ent1}Suppose Assumption R(i) is satisfied, $r\geq 2$, and $\Theta $
is a bounded subset of $\mathbb{R}^{b}$. Let $\mathcal{F}$ be a (non-empty)
bounded subset of $\mathsf{B}_{s}$, $1/2<s<1$. Let $\mathcal{F}_{J,r,\rho }$
and $\mathcal{F}_{\rho }$ be defined as in (\ref{F_J}) and (\ref{F_rho}).
Then for every $1/2<s^{\prime }\leq s$ and every $\varepsilon ^{\ast }>0$
there exists a (positive) finite constant $c^{\ast }$, depending only on $s$%
, $s^{\prime }$, $\mathcal{F}$, $\Theta $, $b$, $\alpha $, $L$, and $%
\varepsilon ^{\ast }$ but not on $J$, such that for every $J\geq 1$ 
\begin{equation}
\log N(\mathcal{F}_{J,r,\rho },\mathsf{L}^{\infty }(\mathcal{V}),\varepsilon
)\leq 2^{-J(s-s^{\prime })/s^{\prime }}c^{\ast }\varepsilon ^{-1/s^{\prime
}}\quad \text{for }0<\varepsilon \leq \varepsilon ^{\ast }  \label{bound_1}
\end{equation}%
holds. Furthermore, for every $\varepsilon ^{\ast }>0$ there exists a
(positive) finite constant $c^{\ast \ast }$ (depending only on $s$, $%
\mathcal{F}$, $\Theta $, $b$, $\alpha $, $L$, and $\varepsilon ^{\ast }$)
such that%
\begin{equation}
\log N(\mathcal{F}_{\rho },\mathsf{L}^{\infty }(\mathcal{V}),\varepsilon
)\leq c^{\ast \ast }\varepsilon ^{-1/s}\quad \text{for }0<\varepsilon \leq
\varepsilon ^{\ast }  \label{bound_2}
\end{equation}%
holds. In particular, $\mathcal{F}_{\rho }$ and $\mathcal{F}_{J,r,\rho }$
are universal Donsker classes.
\end{proposition}

\begin{proof}
Let $s^{\prime }$ be as in the proposition. By Proposition \ref{proj_error_2}%
\begin{equation}
\sup_{f\in \mathcal{F}}\left\Vert \pi _{J}^{(r)}(f)-f\right\Vert _{s^{\prime
},2}\leq 2^{-J(s-s^{\prime })}c_{s,s^{\prime }}^{\prime \prime \prime
}\sup_{f\in \mathcal{F}}\left\Vert f\right\Vert _{s,2}=2^{-J(s-s^{\prime
})}D<\infty ,  \label{sob_norm_res}
\end{equation}%
where the constant $D$ depends only on $s$, $s^{\prime }$, and $\mathcal{F}$%
. As a consequence, 
\begin{equation*}
\mathcal{G}_{J}:=\left\{ (\pi _{J}^{(r)}(f)-f):f\in \mathcal{F}\right\}
\end{equation*}%
is contained in a ball $\mathcal{U}_{J}$ in $\mathsf{B}_{s^{\prime }}$ of
radius $2^{-J(s-s^{\prime })}D$. Using entropy bounds for balls in Besov
spaces (e.g., Theorem 15.6.1 in Lorentz, v.Golitschek, and Makovoz (1996))
we obtain%
\begin{equation*}
\log N(\mathcal{G}_{J},\mathsf{L}^{\infty }([0,1]),\varepsilon )\leq
2^{-J(s-s^{\prime })/s^{\prime }}c(s,s^{\prime },\mathcal{F})\varepsilon
^{-1/s^{\prime }}\quad \text{\ for }0<\varepsilon <\infty
\end{equation*}%
where the finite and positive constant $c(s,s^{\prime },\mathcal{F})$
depends only on $s$, $s^{\prime }$, and $\mathcal{F}$ (in particular, it is
independent of $J$). [Setting $p=2$, $q=\infty $ in Lorentz, v.Golitschek,
and Makovoz (1996) we actually obtain the above bound only in the ess-sup
norm. However, since $\mathcal{G}_{J}$ consists of \textit{continuous}
functions only and since we can always assume that the centers of the
covering ess-sup norm balls belong to $\mathcal{G}_{J}$ (perhaps at the
expense of doubling $\varepsilon $), we immediately obtain the same bound
for the supremum-norm.]

To prove the entropy bound for $\mathcal{F}_{J,r,\rho }=\left\{ g(\rho
(\cdot ,\theta )):g\in \mathcal{G}_{J},\,\theta \in \Theta \right\} $ we
proceed as follows: Note that the elements of $\mathcal{G}_{J}$ are H\"{o}%
lder continuous of order $s^{\prime }-1/2$ with H\"{o}lder constants
uniformly bounded by $2^{-J(s-s^{\prime })}c_{1}(s^{\prime },D)$, with $%
0<c_{1}(s^{\prime },D)<\infty $ depending only on $s^{\prime }$ and $D$,
since $\mathcal{G}_{J}\subseteq \mathcal{U}_{J}\subseteq \mathsf{B}%
_{s^{\prime }}$ and since for $1/2<s^{\prime }<1$ the space $\mathsf{B}%
_{s^{\prime }}$ is continuously embedded into $\mathsf{C}^{s^{\prime }-1/2}$%
, cf. Proposition \ref{elem} in Appendix \ref{App_Spline}. Define $\eta
=(\alpha (s^{\prime }-1/2))^{-1}$ with $\alpha $ defined in Assumption R1.
For $0<\varepsilon \leq 1$ set $\delta =\left( 2^{J(s-s^{\prime
})}\varepsilon \right) ^{\eta }$ and cover $\Theta $ by $\delta $-balls with
centers $\theta _{1},\ldots ,\theta _{N(\delta ,\Theta )}$ where $N(\delta
,\Theta )$ satisfies $N(\delta ,\Theta )\leq \max (1,M(\Theta )/\delta ^{b})$
for some constant $M(\Theta )$ only depending on $\Theta $. Let $%
g_{1},\ldots ,g_{N(\mathcal{G}_{J},\mathsf{L}^{\infty }([0,1]),\varepsilon
)} $ be the centers of $\mathsf{L}^{\infty }([0,1])$-balls of radius $%
\varepsilon $ covering $\mathcal{G}_{j}$. We then have for $g(\rho (\cdot
,\theta ))\in \mathcal{F}_{J,r,\rho }$ using Assumption R1%
\begin{eqnarray*}
&&\sup_{v\in \mathcal{V}}\left\vert g(\rho (v,\theta ))-g_{i}(\rho (v,\theta
_{l})\right\vert \\
&\leq &\sup_{v\in \mathcal{V}}\left\vert g(\rho (v,\theta ))-g(\rho
(v,\theta _{l}))\right\vert +\sup_{v\in \mathcal{V}}\left\vert g(\rho
(v,\theta _{l}))-g_{i}(\rho (v,\theta _{l}))\right\vert \\
&\leq &2^{-J(s-s^{\prime })}c_{1}(s^{\prime },D)\left( L\left\vert \theta
-\theta _{l}\right\vert ^{\alpha }\right) ^{s^{\prime }-1/2}+\sup_{x\in
\lbrack 0,1]}\left\vert g(x)-g_{i}(x)\right\vert \leq \left( c_{1}(s^{\prime
},D)L^{1/\eta }+1\right) \varepsilon
\end{eqnarray*}%
for suitable choice of $i$ and $l$. Consequently, we obtain for $%
0<\varepsilon \leq 1$ 
\begin{eqnarray*}
&&\log N(\mathcal{F}_{J,r,\rho },\mathsf{L}^{\infty }(\mathcal{V}),\left(
c_{1}(s^{\prime },D)L^{1/\eta }+1\right) \varepsilon )\leq \log N(\mathcal{G}%
_{J},\mathsf{L}^{\infty }([0,1]),\varepsilon )+\log N(\delta ,\Theta ) \\
&\leq &c(s,s^{\prime },\mathcal{F})\left( 2^{J(s-s^{\prime })}\varepsilon
\right) ^{-1/s^{\prime }}+\log ^{+}\left( M(\Theta )/(2^{J(s-s^{\prime
})}\varepsilon )^{b\eta }\right) \leq c_{\bullet }2^{-J(s-s^{\prime
})/s^{\prime }}\varepsilon ^{-1/s^{\prime }},
\end{eqnarray*}%
for a suitable finite constant $c_{\bullet }$ only depending on $s$, $%
s^{\prime }$, $\mathcal{F}$, $\Theta $, $b$, and $\alpha $, but not on $J$.
After a simple substitution, this gives (\ref{bound_1}) for $0<\varepsilon
\leq c_{1}(s^{\prime },D)L^{1/\eta }+1$. Appropriately adjusting the
multiplicative constant in this so-obtained bound gives (\ref{bound_1}) for
all $0<\varepsilon \leq \varepsilon ^{\ast }$; note that the adjustment of
the constant only introduces an additional dependence on $\varepsilon ^{\ast
}$ (but no dependence on $J$). The entropy bound (\ref{bound_2}) for $%
\mathcal{F}_{\rho }$\ is proved in a similar (even simpler) way. The Donsker
property of $\mathcal{F}_{J,r,\rho }$ and $\mathcal{F}_{\rho }$ now follows
from (\ref{bound_1}), (\ref{bound_2}) and Theorem 2.8.4 in van der Vaart and
Wellner (1996), noting that $\mathcal{F}_{J,r,\rho }$ and $\mathcal{F}_{\rho
}$ are uniformly bounded in view of Proposition \ref{elem} and that the
bracketing covering numbers are dominated by the sup-norm covering numbers.
\end{proof}

\bigskip

An analogous result holds for the random (signed) measure $P_{n,j,r_{\ast
}}\ $given by $dP_{n,j,r_{\ast }}(y)=p_{n,j,r_{\ast }}(y)dy$. The proof of
this result is similar to, in fact simpler than, the proof of Theorem \ref%
{UCLT} and thus is omitted.

\begin{theorem}
\label{UCLT_2} Suppose $r_{\ast }\geq 2$, and $p_{0}\in \mathcal{B}_{t}$ for
some $t$, $0<t<r_{\ast }$. Let $\mathcal{F}$ be a (non-empty) bounded subset
of $\mathsf{B}_{s}$ for some $s$, $1/2<s<1$. Then for every $1/2<s^{\prime
}\leq s$ there is a finite positive constant $C_{10}$ independent of $j$
(only depending on $s$, $s^{\prime }$, $t$, $\mathcal{F}$, and $p_{0}$) such
that for every $j\geq 1$ and $k\geq 1$ 
\begin{equation*}
E\left\Vert P_{n,j,r_{\ast }}-P_{n}\right\Vert _{\mathcal{F}}\leq
C_{10}(2^{-j(t+s)}+2^{-j(s-s^{\prime })}n^{-1/2}).
\end{equation*}%
Furthermore, $\Vert P_{n}-P\Vert _{\mathcal{F}}=O_{p}(n^{-1/2})$ holds.
Finally, if $j_{n}\rightarrow \infty $ as $n\rightarrow \infty $ satisfies $%
2^{-j_{n}(t+s)}=o(n^{-1/2})$, then%
\begin{equation*}
\sqrt{n}\left( P_{n,j_{n},r_{\ast }}-P\right) \rightsquigarrow _{\ell
^{\infty }(\mathcal{F})}G_{P},
\end{equation*}%
where $G_{P}$ is a sample-bounded and sample-continuous generalized $P$%
-Brownian bridge indexed by $\mathcal{F}$.
\end{theorem}

\appendix

\section{Appendix: Some Properties of Besov Spaces and Approximation by
Splines\label{App_Spline}}

In the following, we summarize some simple properties of the spaces $%
\mathcal{B}_{s}$. For $0<s\leq 1$ and bounded $f:[0,1]\rightarrow \mathbb{R}$
denote by 
\begin{equation*}
\left\Vert f\right\Vert _{s,\infty }=\left\Vert f\right\Vert _{\infty
}+\sup_{x,y\in \lbrack 0,1],x\neq y}\frac{\left\vert f(x)-f(y)\right\vert }{%
\left\vert x-y\right\vert ^{s}}
\end{equation*}%
the usual H\"{o}lder norm and denote by $\mathsf{C}^{s}$ the set of all
functions $f$ with finite $\Vert f\Vert _{s,\infty }$. For simplicity we
restrict ourselves to the case $s<1$ in the following proposition.

\begin{proposition}
\label{elem} Let $1/2<s<1$.

a. Every $f\in \mathcal{B}_{s}$ is $\lambda $-a.e. equal to a function $%
\tilde{f}\in \mathsf{C}^{s-1/2}$ and%
\begin{equation*}
\left\Vert \tilde{f}\right\Vert _{\infty }\leq \left\Vert \tilde{f}%
\right\Vert _{(s-1/2),\infty }\leq c_{s}\left\Vert \tilde{f}\right\Vert
_{s,2}=c_{s}\left\Vert f\right\Vert _{s,2}
\end{equation*}%
holds for some finite (positive) constant $c_{s}$ that depends only on $s$.

b. If $f\in \mathcal{B}_{s}$ and $h\in \mathcal{B}_{s}$, then $\left\Vert
fh\right\Vert _{s,2}\leq 2c_{s}\left\Vert f\right\Vert _{s,2}\left\Vert
h\right\Vert _{s,2}$. If $h\in \mathcal{B}_{s}$ satisfies $\zeta
:=\inf_{x\in \lbrack 0,1]}h(x)>0$, then $\left\Vert 1/h\right\Vert
_{s,2}\leq \zeta ^{-1}+\zeta ^{-2}\left\Vert h\right\Vert _{s,2}$.
\end{proposition}

\begin{proof}
a. Observe that $\mathcal{B}_{s}$ coincides (up to norm equivalence) with
the intermediate space $(\mathcal{L}^{2},\mathcal{W}_{2}^{1})_{s,\infty }$
(DeVore and Lorentz (1993), p.196) and hence coincides with the Besov space $%
\mathcal{B}^{s;2,\infty }((0,1))$ defined in Adams and Fournier (2003) (the
fact that the latter is defined on the open unit interval being irrelevant).
The claim then follows from applying Theorem 7.37 in Adams and Fournier
(2003) (with $m=n=1$, $j=0$, $p=2$, $q=\infty $).

b. Since $s<1$ by assumption, we may set $a=1$ in the definition of the
Besov (semi)norm. Elementary calculations then show that 
\begin{equation*}
\left\Vert fh\right\Vert _{s,2}\leq \left\Vert f\right\Vert _{s,2}\limfunc{%
esssup}\left\vert h\right\vert +\left\Vert h\right\Vert _{s,2}\limfunc{esssup%
}\left\vert f\right\vert \leq 2c_{s}\left\Vert f\right\Vert _{s,2}\left\Vert
h\right\Vert _{s,2}
\end{equation*}%
in view of Part a. The second claim follows since clearly $\left\Vert
1/h\right\Vert _{2}\leq \zeta ^{-1}$ and since elementary calculations give $%
\left\Vert \Delta _{z}h^{-1}\right\Vert _{2}\leq \zeta ^{-2}\left\Vert
\Delta _{z}h\right\Vert _{2}$.
\end{proof}

\bigskip

The above proposition, together with the continuous embedding of $\mathcal{B}%
_{t}$ into $\mathcal{B}_{s}$ for $t\geq s$ (DeVore and Lorentz (1993),
p.56), immediately guarantees for \textit{every} $t>1/2$ the existence of a
constant $c_{t}$, $0<c_{t}<\infty $, such that for every $f\in \mathcal{B}%
_{t}$ there exists a (unique) continuous $\tilde{f}$, $\lambda $-a.e. equal
to $f$, such that $\Vert \tilde{f}\Vert _{\infty }\leq c_{t}\Vert \tilde{f}%
\Vert _{t,2}=c_{t}\Vert f\Vert _{t,2}$. In particular, bounded subsets of $%
\mathsf{B}_{t}$, $t>1/2$, are sup-norm bounded.

As is well known, functions in $\mathcal{B}_{s}$ can be approximated by
elements of the Schoenberg spaces $\mathcal{S}_{j}(r)$, the error decreasing
as $j$ increases. We summarize these facts in the following proposition.

\begin{proposition}
\label{proj_error} Suppose $r\in \mathbb{N}$.

a. If $h\in \mathcal{L}^{2}$, then the ortho-projection operator $\pi
_{j}^{(r)}$ from $\mathcal{L}^{2}$ onto the Schoenberg space $\mathcal{S}%
_{j}(r)$ satisfies%
\begin{equation*}
\lim_{j\rightarrow \infty }\Vert \pi _{j}^{(r)}(h)-h\Vert _{2}=0.
\end{equation*}%
If $\mathcal{H}$ is a relatively compact subset of $\mathcal{L}^{2}$, then 
\begin{equation*}
\lim_{j\rightarrow \infty }\sup_{h\in \mathcal{H}}\Vert \pi
_{j}^{(r)}(h)-h\Vert _{2}=0.
\end{equation*}

b. If $h\in \mathcal{B}_{s}$ for some $s\in (0,r)$, then%
\begin{equation*}
\Vert \pi _{j}^{(r)}(h)-h\Vert _{2}\leq 2^{-js}c_{s}^{\prime }\Vert h\Vert
_{s,2},
\end{equation*}%
for every $j\in \mathbb{N}$, where the (positive) finite constant $%
c_{s}^{\prime }$ depends only on $s$.
\end{proposition}

\begin{proof}
To prove the first claim in Part a, observe that by Proposition 2.4.1 and
(12.3.2) in DeVore and Lorentz (1993) 
\begin{equation*}
\Vert \pi _{j}^{(r)}(h)-h\Vert _{2}\leq 2C^{(r)}\sup_{0<z\leq 2^{-j}}\Vert
\Delta _{z}^{r}(h)\Vert _{2}
\end{equation*}%
for some universal constant $C^{(r)}$. By continuity of translation in $%
\mathcal{L}^{2}(\mathbb{R})$ (cf., e.g., Folland (1999), Proposition 8.5)
the right-hand side converges to zero as $j\rightarrow \infty $ (note that $%
\Vert \Delta _{z}^{r}(h)\Vert _{2}$ is less than or equal to the
corresponding expression that is obtained when $h$ is viewed as a function
on $\mathbb{R}$ which is zero outside of $[0,1]$). The second claim in Part
a follows since for every $\varepsilon >0$ and $\varepsilon $-net $\left\{
h_{l}:1\leq l\leq N(\varepsilon )\right\} $ for $\mathcal{H}$ we have that $%
\Vert h-h_{l}\Vert _{2}\leq \varepsilon $ implies $\Vert \pi
_{j}^{(r)}(h)-\pi _{j}^{(r)}(h_{l})\Vert _{2}\leq \varepsilon $ and thus%
\begin{equation*}
\sup_{h\in \mathcal{H}}\Vert \pi _{j}^{(r)}(h)-h\Vert _{2}\leq \max_{1\leq
l\leq N(\varepsilon )}\Vert \pi _{j}^{(r)}(h_{l})-h_{l}\Vert
_{2}+2\varepsilon
\end{equation*}%
holds. For the proof of Part b use Proposition 2.4.1 and (12.3.2) in DeVore
and Lorentz (1993) (where one sets $p=2$, $n=2^{j}$) together with the
definition of the Besov-norm.
\end{proof}

\begin{proposition}
\label{proj_error_2} Suppose $r\in \mathbb{N}$. Let $h\in \mathcal{B}_{s}$
for some $s\in (0,r-1/2)$. Then%
\begin{equation*}
\Vert \pi _{j}^{(r)}(h)\Vert _{s,2}\leq c_{s}^{\prime \prime }\Vert h\Vert
_{s,2},
\end{equation*}%
for every $j\in \mathbb{N}$, where the (positive) finite constant $%
c_{s}^{\prime \prime }$ depends only on $s$. Furthermore, for every $%
s^{\prime }\in (0,s]$ 
\begin{equation*}
\Vert \pi _{j}^{(r)}(h)-h\Vert _{s^{\prime },2}\leq 2^{-j(s-s^{\prime
})}c_{s,s^{\prime }}^{\prime \prime \prime }\Vert h\Vert _{s,2}
\end{equation*}%
for every $j\in \mathbb{N}$, where the (positive) finite constant $%
c_{s,s^{\prime }}^{\prime \prime \prime }$ depends only on $s$ and $%
s^{\prime }$.
\end{proposition}

\begin{proof}
By Theorem 12.3.3. in DeVore and Lorentz (1993) (with $p=2$, $\lambda =r-1/2$%
, $q=\infty $, $\alpha =s$, and $d_{n,r}(\cdot )_{2}$ defined on p.358 of
that reference) we have 
\begin{eqnarray*}
\Vert \pi _{j}^{(r)}(h)\Vert _{s,2} &=&\Vert \pi _{j}^{(r)}(h)\Vert
_{2}+\sup_{0\neq |z|<1}|z|^{-s}\Vert \Delta _{z}^{r}(\pi _{j}^{(r)}(h))\Vert
_{2} \\
&\leq &\Vert h\Vert _{2}+e_{s}\sup_{n\geq 0}2^{ns}d_{n,r}(\pi
_{j}^{(r)}(h))_{2} \\
&\leq &\Vert h\Vert _{2}+e_{s}\sup_{n\geq 0}2^{ns}\Vert \pi _{n}^{(r)}(\pi
_{j}^{(r)}(h))-\pi _{j}^{(r)}(h)\Vert _{2} \\
&\leq &\Vert h\Vert _{2}+e_{s}\sup_{0\leq n<j}2^{ns}\Vert \pi
_{n}^{(r)}(h)-\pi _{j}^{(r)}(h)\Vert _{2}\leq \Vert h\Vert
_{s,2}+2e_{s}c_{s}^{\prime }\Vert h\Vert _{s,2}
\end{eqnarray*}%
for some universal constant $e_{s}$, where we have used Proposition \ref%
{proj_error} in the last step. To prove the second claim we argue as before
and then use Proposition \ref{proj_error} to obtain%
\begin{eqnarray*}
\Vert \pi _{j}^{(r)}(h)-h\Vert _{s^{\prime },2} &\leq &\Vert \pi
_{j}^{(r)}(h)-h\Vert _{2}+e_{s^{\prime }}\sup_{n\geq 0}2^{ns^{\prime }}\Vert
\pi _{n}^{(r)}(\pi _{j}^{(r)}(h)-h)-(\pi _{j}^{(r)}(h)-h)\Vert _{2} \\
&\leq &\Vert \pi _{j}^{(r)}(h)-h\Vert _{2}+e_{s^{\prime }}\left[
2^{js^{\prime }}\Vert \pi _{j}^{(r)}(h)-h\Vert _{2}+\sup_{n>j}2^{ns^{\prime
}}\Vert \pi _{n}^{(r)}(h)-h\Vert _{2}\right] \\
&\leq &2^{-js}c_{s}^{\prime }\Vert h\Vert _{s,2}+e_{s^{\prime }}\left[
2^{j(s^{\prime }-s)}c_{s}^{\prime }\Vert h\Vert
_{s,2}+\sup_{n>j}2^{n(s^{\prime }-s)}c_{s}^{\prime }\Vert h\Vert _{s,2}%
\right] \\
&\leq &2^{-j(s-s^{\prime })}(1+2e_{s^{\prime }})c_{s}^{\prime }\Vert h\Vert
_{s,2}.
\end{eqnarray*}
\end{proof}

\section{Appendix: Consistency of the Indirect Inference Estimator and
Measurability Issues\label{App_Meas}}

\begin{proof}[Proof of Proposition \protect\ref{exist}]
Because of continuity of the B-spline basis functions for $r\geq 2$ and
continuity of $\theta \rightarrow \rho (v,\theta )$ for every $v\in \mathcal{%
V}$, the map $\theta \rightarrow p_{k,J,r}(\theta )(y)$ is continuous for
every $y\in \lbrack 0,1]$. Furthermore, $p_{n,j,r_{\ast }}$ and $%
p_{k,J,r}(\theta )$ are bounded on $[0,1]$, the latter one uniformly in $%
\theta $, in view of the discussion surrounding (\ref{sup-norm_estim}). Next
note that the set $A_{n}$ appearing in the definition of $\mathcal{Q}_{n,k}$
coincides with the event $\left\{ \inf_{y\in \lbrack 0,1]}p_{n,j,r_{\ast
}}(y)>0\right\} $, since $p_{n,j,r_{\ast }}$ is continuous on $[0,1]$ in
case $r_{\ast }>1$, and is piecewise constant in case $r_{\ast }=1$. Hence,
by the dominated convergence theorem, $\mathcal{Q}_{n,k}$ is continuous (and
real-valued) on $\Theta $ if $p_{n,j,r_{\ast }}(y)>0$ for every $y\in
\lbrack 0,1]$; and the same conclusion trivially holds in the other case. As
mentioned before, $\mathcal{Q}_{n,k}(\theta ):[0,1]^{\infty }\times \mathcal{%
V}^{\infty }\rightarrow \mathbb{R}$ is $\mathfrak{B}_{[0,1]}^{\infty
}\otimes \mathfrak{V}^{\infty }$-measurable for every $\theta \in \Theta $.
Since $\Theta $ is compact, existence of a measurable minimizer then
follows, e.g., from Lemma A3 in P\"{o}tscher and Prucha (1997).
\end{proof}

\begin{proposition}
\label{exist_1} Suppose $\Theta $ is compact in $\mathbb{R}^{b}$, that the
map $\theta \rightarrow p(\theta ,x)$ is continuous on $\Theta $ for every $%
x\in \lbrack 0,1]$ and that $\sup_{\theta \in \Theta }\left\Vert p(\theta
)\right\Vert _{\infty }<\infty $. Furthermore, assume that $r_{\ast }\geq 1$
holds. Then there exists a $\mathfrak{B}_{[0,1]}^{\infty }\otimes \mathfrak{V%
}^{\infty }$-measurable $\hat{\theta}_{n}$ that minimizes $Q_{n}(\theta )$
over $\Theta $. (In fact, $\hat{\theta}_{n}$ is $\mathfrak{B}%
_{[0,1]}^{\infty }$-measurable as it does not depend on the simulations.)
\end{proposition}

\begin{proof}
Since $\left\Vert p_{n,j,r_{\ast }}\right\Vert _{\infty }<\infty $ and since
on the event $A_{n}$ also $\inf_{y\in \lbrack 0,1]}p_{n,j,r_{\ast }}>0$
holds, the assumptions on $p(\theta )$ and the dominated convergence theorem
imply that $Q_{n}$ is real-valued and continuous in $\theta $ on the event $%
A_{n}$; and the same conclusion trivially holds on the complement of $A_{n}$%
. Furthermore, $\mathfrak{B}_{[0,1]}^{\infty }\otimes \mathfrak{V}^{\infty }$%
-measurability of $Q_{n}(\theta ):[0,1]^{\infty }\times \mathcal{V}^{\infty
}\rightarrow \mathbb{R}$ for every $\theta \in \Theta $ follows from
Tonelli's Theorem since $p_{n,j,r_{\ast }}$ is jointly measurable (and $%
A_{n} $ is measurable). Since $\Theta $ is compact, existence of a
measurable minimizer then follows, e.g., from Lemma A3 in P\"{o}tscher and
Prucha (1997).
\end{proof}

\begin{proposition}
\label{cons}Suppose Assumptions P1(i),(ii) are satisfied and $r_{\ast }\geq
2 $ holds. If $j_{n}\rightarrow \infty $ as $n\rightarrow \infty $ in such a
way that for some $\delta >1/2$ we have $\sup_{n\geq 1}2^{j_{n}(2\delta
+1)}/n<\infty $ then%
\begin{equation*}
\hat{\theta}_{n}\rightarrow \theta _{0}\text{ in }\Pr \text{-probability as }%
n\rightarrow \infty ,
\end{equation*}%
where $\hat{\theta}_{n}$ has been defined in Section \ref{intermed}.
\end{proposition}

The proof of this result is completely analogous to the proof of Proposition %
\ref{consist} and is thus omitted.

\begin{remark}
(Measurability issues) (i) For every $J\geq 1$, $r\geq 1$, and $\theta \in
\Theta $, the expressions $\left\Vert p_{k,J,r}(\theta )\right\Vert _{2}$, $%
\left\Vert p_{k,J,r}(\theta )\right\Vert _{\infty }$, and $\left\Vert
p_{k,J,r}(\theta )\right\Vert _{s,2}$ (for $s\leq r-1/2$) are measurable
functions of $v_{1},\ldots ,v_{k}$, since the coefficients $\hat{\gamma}%
_{lJ}^{(r)}(\theta )$ are measurable. This is obvious for the $\mathcal{L}%
^{2}$-norm, but holds in general for the following reason: observe that any
one of the norms mentioned, when restricted to $\mathcal{S}_{J}(r)$, is a
continuous function of the coefficients $\hat{\gamma}_{lJ}^{(r)}(\theta )$
because $\mathcal{S}_{J}(r)$ is finite-dimensional. The same is true if $%
p_{k,J,r}(\theta )$ is replaced by $p_{k,J,r}(\theta )-Ep_{k,J,r}(\theta )$
or $p_{k,J,r}(\theta )-p(\theta )$, in the latter case provided the
respective norm of $p(\theta )$ is finite. [The argument is the same, except
that $\mathcal{S}_{J}(r)$ is to be replaced by the linear span of $\mathcal{S%
}_{J}(r)\cup \{p(\theta )\}$ for establishing the latter claim.] Analogous
statements obviously also hold for $p_{n,j,r_{\ast }}$ for every $j\geq 1$, $%
r\geq 1$. (ii) The reasoning just given in fact establishes that the above
mentioned norms of $p_{k,J,r}(\theta )$ and $p_{k,J,r}(\theta
)-Ep_{k,J,r}(\theta )$ are continuous functions of $\theta $, provided the
coefficients $\hat{\gamma}_{lJ}^{(r)}(\theta )$ (and $E\hat{\gamma}%
_{lJ}^{(r)}(\theta )$) are continuous in $\theta $ (which is, e.g., the case
if $r\geq 2$ and Assumption R(i) holds); consequently, suprema over $\theta $
of the above mentioned norms of $p_{k,J,r}(\theta )$ and $p_{k,J,r}(\theta
)-Ep_{k,J,r}(\theta )$ are then measurable. [We note that this argument does
not apply to suprema of norms of $p_{k,J,r}(\theta )-p(\theta )$, because $%
p(\theta )$ may not vary in a finite-dimensional space when $\theta $
varies.]
\end{remark}

\section{Appendix: Moment Bounds for Empirical Processes\label{App_Moment}}

The following moment inequalities can be deduced from a general theorem in
Gin\'{e} and Koltchinskii (2006) and a refinement with explicit constants in
Gin\'{e} and Nickl (2009a).

\begin{proposition}
\label{exp} Let $Z_{i}$, $i\in \mathbb{N}$, be i.i.d. random variables with
values in a measurable space $(S,\mathcal{A})$ and common law $R$. Let $%
\mathcal{F}$ be a countable $R$-centered class of real valued measurable
functions from $(S,\mathcal{A})$ to $\mathbb{R}$. Assume that $\mathcal{F}$
is uniformly bounded by a finite positive constant $U$ and let further $%
\sigma ,$ $0<\sigma \leq U$, be some constant satisfying $\sup_{f\in 
\mathcal{F}}Ef^{2}(Z_{i})\leq \sigma ^{2}$.

a. Assume that the $\mathcal{L}^{2}(Q)$-covering numbers satisfy 
\begin{equation*}
\sup_{Q}\log N(\mathcal{F},\mathcal{L}^{2}(Q),\tau )\leq v\log \left( \frac{%
AU}{\tau }\right) ,\ \ 0<\tau \leq 2U,
\end{equation*}%
for some $A>e$ and $v\geq 2$ (the supremum extending over all probability
measures $Q$ on $S$). Then, for every $b_{0}>0$ satisfying%
\begin{equation}
n\sigma ^{2}\geq b_{0}vU^{2}\log \left( 5AU/\sigma \right) \text{ \ \ for
all }n\in \mathbb{N},  \label{range}
\end{equation}%
there exists a finite positive constant $b_{1}(v,b_{0})$, that depends only
on $v$ and $b_{0}$, such that for every $n\in \mathbb{N}$ 
\begin{equation*}
E\left\Vert \sum_{i=1}^{n}f(Z_{i})\right\Vert _{\mathcal{F}}^{2}\leq
b_{1}(v,b_{0})n\sigma ^{2}\log \frac{AU}{\sigma }
\end{equation*}%
holds.

b. Assume that the $\mathcal{L}^{2}(Q)$-covering numbers satisfy 
\begin{equation*}
\sup_{Q}\log N(\mathcal{F},\mathcal{L}^{2}(Q),\tau )\leq \left( \frac{%
A^{\prime }U}{\tau }\right) ^{w},\ \ 0<\tau \leq 2U,
\end{equation*}%
for some $0<A^{\prime }<\infty $ and $0<w<2$. Then, for all $n\in \mathbb{N}$
and some positive constant $b_{2}$, that depends only on $A^{\prime },w$, we
have 
\begin{equation*}
E\left\Vert \sum_{i=1}^{n}f(Z_{i})\right\Vert _{\mathcal{F}}\leq b_{2}\sqrt{n%
}U.
\end{equation*}
\end{proposition}

\begin{proof}
Since the results depend only on the distribution of $\left\Vert
\sum_{i=1}^{n}f(Z_{i})\right\Vert _{\mathcal{F}}$, we may assume w.l.o.g.
that -- as in Gin\'{e} and Koltchinskii (2006) -- the random variables are
realized as coordinate projections on the infinite product space of $(S,%
\mathcal{A})$. The second claim of the proposition then follows directly
from Theorem 3.1 in Gin\'{e} and Koltchinskii (2006) applied to the class $%
\mathcal{F}^{\prime }=\left\{ f/U:f\in \mathcal{F}\right\} $ with envelope $%
F=1$ and $H(x)=\left( A^{\prime }x\right) ^{w}$ for $x\geq 1/2$ and $H(x)=0$
for $0\leq x<1/2$. The first claim is proved as follows: By Proposition 3.1
in Gin\'{e}, Lata\l a and Zinn (2000) (applied to $\mathcal{F\cup (-F)}$ and
observing that $\sigma ^{2}$ in that reference is bounded by $n\sigma ^{2}$
in our notation) we have%
\begin{equation*}
E\left\Vert \sum_{i=1}^{n}f(Z_{i})\right\Vert _{\mathcal{F}}^{2}\leq K^{2}%
\left[ \left( E\left\Vert \sum_{i=1}^{n}f(Z_{i})\right\Vert _{\mathcal{F}%
}\right) ^{2}+2n\sigma ^{2}+4U^{2}\right] ,
\end{equation*}%
where $K$ is a universal constant. We then bound the first term on the
right-hand side by using Proposition 3 in Gin\'{e} and Nickl (2009a) and
simplify the resulting bound using (\ref{range}), $A>e$, and $U/\sigma \geq
1 $ to arrive at the result.
\end{proof}

\section*{References}

\ \ \ \ \thinspace Adams, R.~A. \& J.~J.~F. Fournier (2003): \textit{Sobolev
Spaces. }2nd edition, Elsevier.

Altissimo, F. \& A. Mele (2009): Simulated nonparametric estimation of
dynamic models. \textit{Review of Economic Studies}, forthcoming.

Beran, R. (1977):\ Minimum Hellinger distance estimates for parametric
models. \textit{Annals of Statistics }5, 445-463.

Beran, R. \& P.~W. Millar (1987): Stochastic estimation and testing. \textit{%
Annals of Statistics }15, 1131-1154.

Bickel, P. \& Y. Ritov (2003): Nonparametric estimators that can be
`plugged-in'. \textit{Annals of Statistics }31, 1033-1053.

Carrasco, M., M. Chernov, J.~P. Florens \& E. Ghysels (2007): Efficient
estimation of general dynamic models with a continuum of moment conditions. 
\textit{Journal of Econometrics }140, 529-573.

DeVore, R.~A. \& G.~G. Lorentz (1993): \textit{Constructive Approximation. }%
Springer-Verlag.

Donoho, D.~L. \& R.~C. Liu (1988): The \textquotedblleft
automatic\textquotedblright\ robustness of minimum distance functionals. 
\textit{Annals of Statistics }16, 552-586.

Fermanian, J.~D. \& B. Salani\'{e} (2004): A nonparametric simulated maximum
likelihood estimation method. \textit{Econometric Theory }20, 701-734.

Folland, G. (1999): \textit{Real Analysis: Modern Techniques and Their
Applications, }2nd edition, Wiley.

Gach, F. (2010): \textit{Efficiency in Indirect Inference}. PhD Thesis,
University of Vienna.

Gallant, R. \& G. Tauchen (1996): Which moments to match? \textit{%
Econometric Theory }12, 657-681.

Gallant R. \& J. Long (1997): Estimating stochastic differential equations
efficiently by minimum chi-squared. \textit{Biometrika }84, 125-141.

Gin\'{e}, E. \& V. Koltchinskii (2006): Concentration inequalities and
asymptotic results for ratio type empirical processes. \textit{Annals of
Probability }34, 1143-1216.

Gin\'{e}, E., R. Lata\l a \& J. Zinn (2000): Exponential and moment
inequalities for U-statistics. In: Gin\'{e}, E., Mason, D.~M., Wellner,
J.~A. (eds.): High-dimensional Probability II, Progress in Probability 47,
13-38.

Gin\'{e}, E. \& R. Nickl (2008): Uniform central limit theorems for kernel
density estimators. \textit{Probability Theory and Related Fields} 141,
333-387.

Gin\'{e}, E. \& R. Nickl (2009a): An exponential inequality for the
distribution function of the kernel density estimator, with applications to
adaptive estimation. \textit{Probability Theory and Related Fields},
forthcoming.

Gin\'{e}, E. \& R. Nickl (2009b): Uniform limit theorems for wavelet density
estimators. \textit{Annals of Probability}, forthcoming.

Gourieroux, C., A. Monfort, \& E. Renault (1993): Indirect inference. 
\textit{Journal of Applied Econometrics }8, 85-118.

Gourieroux, C. \& A. Monfort (1996): \textit{Simulation-based econometric
methods. }Oxford University Press.

Huber, P.~J. (1972): Robust statistics: A review. \textit{Annals of
Mathematical Statistics} 43, 1041-1067.

Jiang, W. \& B. Turnbull (2004): The indirect method: Inference based on
intermediate statistics -- A synthesis and examples. \textit{Statistical
Science }19, 239-263.

Lindsay (1994):\ Efficiency versus robustness: The case for minimum
Hellinger distance and related methods. \textit{Annals of Statistics }22,
1081-1114.

Lorentz, G.~G., v.Golitschek, M. \& Y. Makovoz (1996): \textit{Constructive
Approximation: Advanced Problems. }Springer-Verlag.

Millar, P.~W. (1981): Robust estimation via minimum distance methods. 
\textit{Zeitschrift f\"{u}r Wahrscheinlichkeitstheorie und Verwandte Gebiete}
55, 73-89.

Nickl, R. (2007): Donsker-type theorems for nonparametric maximum likelihood
estimators. \textit{Probability Theory and Related Fields} 138, 411-449.

P\"{o}tscher, B.~M. \& I.~R. Prucha (1997): \textit{Dynamic Nonlinear
Econometric Models: Asymptotic Theory. }Springer-Verlag.

Shadrin, A.~Yu. (2001): The $L_{\infty }$-norm of the $L_{2}$-spline
projector is bounded independently of the knot sequence: a proof of de
Boor's conjecture. \textit{Acta Mathematica }187, 59-137.

Smith, A. (1993): Estimating nonlinear time-series models using simulated
vector autoregressions. \textit{Journal of Applied Econometrics }8, 63-84.

van der Vaart, A.~W. \& J.~A. Wellner (1996): \textit{Weak Convergence and
Empirical Processes With Applications to Statistics. }Springer-Verlag.

\pagebreak

\end{document}